\documentclass[11pt,twoside]{amsart}

\usepackage{amsmath,ifthen, amsfonts, amssymb,
srcltx, amsopn, color, enumerate, hyperref}
\usepackage[cmtip,arrow]{xy}
\usepackage{pb-diagram, pb-xy}
\usepackage{overpic}

\usepackage[T1]{fontenc}

\newcommand{\showcomments}{no}

\newsavebox{\commentbox}
{\ifthenelse{\equal{\showcomments}{yes}}%
{\footnotemark
        \begin{lrbox}{\commentbox}
        \begin{minipage}[t]{0.8in}\raggedright\sffamily\tiny
        \footnotemark[\arabic{footnote}]}
{\begin{lrbox}{\commentbox}}}%
{\ifthenelse{\equal{\showcomments}{yes}}%
{\end{minipage}\end{lrbox}\marginpar{\usebox{\commentbox}}}
{\end{lrbox}}}

\newcounter{ax}
\newtheorem{thm}{Theorem}[section]

\newtheorem{lem}[thm]{Lemma}

\newtheorem{prop}[thm]{Proposition}

\newtheorem{thmi}{Theorem}

\theoremstyle{definition}
\newtheorem{defn}[thm]{Definition}

\newtheorem{rem}[thm]{Remark}

\newtheorem{exmp}[thm]{Example}

\newtheorem{notation}[thm]{Notation}

\newtheorem{claim}{Claim}
\newtheorem*{outline}{Outline of the proof}

\newtheorem{claim*}{Claim}

\newtheorem*{rem*}{Remark}

\DeclareMathOperator{\dimension}{dim}
\DeclareMathOperator{\kernel}{ker}

\DeclareMathOperator{\link}{Lk}

\DeclareMathOperator{\stabilizer}{Stab}
\DeclareMathOperator{\diam}{diam}

\newcommand{\neb}{\mathcal N}

\newcommand{\field}[1]{\mathbb{#1}}
\newcommand{\integers}{\ensuremath{\field{Z}}}

\newcommand{\naturals}{\ensuremath{\field{N}}}
\newcommand{\reals}{\ensuremath{\field{R}}}

\newcommand{\boundary}{{\ensuremath \partial}}

\makeatletter

\newcommand{\Rmnum}[1]{\mathbf{{\expandafter\@slowromancap\romannumeral #1@}}}

\newcommand{\contact}[1]{\ensuremath{\mathcal C#1}}

\makeatother

\let\oldmarginpar\marginpar
\renewcommand\marginpar[1]{\-\oldmarginpar[\raggedleft\footnotesize #1]%
{\raggedright\footnotesize #1}}

\newcounter{enumitemp}

\newcommand{\dist}{\textup{\textsf{d}}}

\newcommand{\gate}{\mathfrak g}

\newcommand{\cay}{\mathrm{Cay}}

\newcommand{\hyp}{\mathcal H}

\begin{document}

\title{Acylindrical hyperbolicity of cubical small cancellation groups}

\author[G.N. Arzhantseva]{Goulnara N. Arzhantseva}
\address{University of Vienna, Faculty of Mathematics, Oskar-Morgenstern-Platz 1, 1090 Wien, Austria}
\email{goulnara.arzhantseva@univie.ac.at}
\thanks{G.N. Arzhantseva was partially supported by her ERC grant ANALYTIC no. 259527}

\author[M.F. Hagen]{Mark F. Hagen}
\address{Department of Pure Maths. and Math. Stats., University of Cambridge, Cambridge, UK}
\curraddr{School of Mathematics, University of Bristol, Bristol, UK}
\email{markfhagen@gmail.com}
\thanks{M.F. Hagen was initially supported by the EPSRC grant of Henry Wilton}

%\date{\today}
\keywords{Cubical small cancellation, acylindrically hyperbolic group}
\subjclass[2010]{20F06, 20F65, 20F67}

\maketitle

\begin{abstract}
We provide an analogue of Strebel's classification of geodesic triangles in classical $C'(\frac16)$ groups for groups given 
by Wise's cubical presentations satisfying sufficiently strong metric cubical small cancellation conditions.  Using our 
classification, we give conditions guaranteeing that a cubical small cancellation group is acylindrically hyperbolic.
\end{abstract}

\tableofcontents

\section{Introduction}\label{sec:intro}
A cubical presentation of a group is a high-dimensional generalization of both a ``classical'' and a ``graphical'' 
presentation of a group in terms of generators and relators. Cubical presentations, and the cubical 
small cancellation theory developed by Wise~\cite[Section 3]{Wise:QCH}, plays a significant role in geometric 
group theory, following spectacular solutions of the virtual Haken conjecture by Agol and of Baumslag's conjecture on 
one-relator groups with torsion by Wise.

A \emph{classical presentation} of a group $G$ 
consists of a wedge $X$ of circles and a collection of combinatorial immersions $Y_i\to X$ of circles so that the 
\emph{presentation complex} $X^*$ formed from $X$ by coning off the various $Y_i$ satisfies $\pi_1X^*\cong G$.  The 
$1$--skeleton $\cay(X^*)$ of the universal cover $\widetilde X^*$ of $X^*$ is a Cayley graph of $G$ with respect to the 
generating set implicit in the choice of $X$.  A \emph{graphical presentation} is a natural generalization of this:  $X$ is 
allowed to be an arbitrary graph, and each $Y_i\to X$ becomes an immersion of graphs.  

In~\cite{Wise:QCH}, it is observed that allowing even more flexibility in the choice of $X$ leads to more tractable 
``presentations''.  This leads to the notion of a \emph{cubical presentation} (see Section~\ref{sec:cubical_pres}): $X$ is now a connected nonpositively 
curved cube complex and each $Y_i$ is a connected nonpositively curved cube complex equipped with a local isometry $Y_i\to 
X$.  The presentation complex $X^*$ is defined analogously, and there is a \emph{generalized Cayley graph} $\cay(X^*)$ which 
is the cubical part of the universal cover of $X^*$.  The analogy with classical presentations is: the cube complex 
$X$ is a kind of ``high-dimensional generating set'', the CAT(0) cube complex $\widetilde X$ is the ``high-dimensional tree'' 
taking the place of the free group on the generating set in the classical case, and $\cay(X^*)$ corresponds to a Cayley 
graph.  

One can then impose cubical small cancellation conditions, in which ``generalized overlaps'' between the various 
$Y_i$ (i.e. shadows of $Y_i$ on $Y_j$, as propagated through the intervening cubes) are small in the appropriate metric 
sense.  In this setting, there are powerful tools -- specifically, the \emph{ladder theorem} and the \emph{cubical 
Greendlinger lemma/diagram trichotomy} (see Section~\ref{sec:cubical_pres}) -- that allow one to extract 
information about a group from a small cancellation cubical presentation.  The small cancellation conditions of interest in 
this paper are the \emph{cubical $C'(\alpha)$ conditions}, for $\alpha>0$.  These say that $|P|<\alpha\|Y_i\|$ for all 
$P,i$, where $\|Y_i\|$ denotes the length of a shortest essential (i.e. not homotopic to a constant map) closed combinatorial path in $Y_i$ and $|P|$ is the length 
of the geodesic \emph{piece} $P$ in $Y_i$.  A piece is, roughly, a path in the ``generalized overlap'' between distinct elevations to 
$\widetilde X$ of the various $Y_i$, or between such elevations and carriers of hyperplanes in $\widetilde X$ that do not 
cross $\widetilde Y_i$.  We also use the stronger \emph{uniform cubical $C''(\alpha)$ condition}, which asks that any geodesic piece $P$ in \emph{any} $Y_i$ is shorter than $\alpha\|Y_j\|$, for any 
relator $Y_j$.  See Definition~\ref{defn:metric_small_c}.

Many groups that do not admit classical presentations satisfying 
strong small cancellation conditions nonetheless admit cubical presentations with these properties.  For example, if $G$ is 
the fundamental group of a nonpositively curved cube complex $X$, then $G$ admits a cubical presentation with no relators, 
and therefore satisfies \emph{arbitrarily} strong cubical small cancellation conditions; on the other hand, $G$ does not in 
general satisfy strong classical small cancellation conditions, as can be seen by considering, for instance, right-angled 
Artin groups.  Later in this introduction, we list more examples of cubical small cancellation groups.

\subsection*{Classifying triangles}  Our first result is geometric. We classify geodesic triangles in $\cay(X^*)$ in terms 
of the disc diagrams that they bound in $\widetilde X^*$.  This 
is a cubical analogue of Strebel's classification of geodesic triangles in $C'(\frac16)$ groups (Theorem~43 
of~\cite{Strebel}), which says that any geodesic triangle bounds a disc diagram of one of a small number of specific 
combinatorial types:

\begin{thmi}[Classification of triangles]\label{thmi:classification_of_triangles}
There exists $\alpha>0$ so that the following holds.  Let $X$ be a connected nonpositively curved cube complex, let $\mathcal I$ be a 
(possibly infinite) index set and let $\{Y_i\to X\}_{i\in\mathcal I}$ be a set of local isometries of connected nonpositively curved cube complexes.  Suppose that the 
cubical presentation $\langle X\mid\{Y_i\}_{i\in\mathcal I}\rangle$ satisfies the cubical $C'(\alpha)$ condition. Let $X^*$ 
be the presentation complex and  $x,y,z$ be $0$--cells of the universal cover $\widetilde X^*$.  Then there exists a 
geodesic triangle $\Delta$ in $\widetilde X^*$, with corners $x,y,z$, so that $\Delta$ is the boundary path of a disc 
diagram $D\to X^*$ of one of $9$ types; in particular, $D$ is the union of $3$ \emph{padded ladders}.  Moreover, any other 
geodesic triangle with corners $x,y,z$ is square-homotopic to such a $\Delta$. 
\end{thmi}

\begin{figure}[h]
\includegraphics[width=\textwidth]{objeto_volador_no_identificado.pdf}
\caption{The enumerated cases from Theorem~\ref{thm:strebel_cubical_small_can} are shown, clockwise from the top left:
cases \eqref{item:3g},\eqref{item:2g},\eqref{item:1g},\eqref{item:0g},\eqref{item:0t},\eqref{item:1t},\eqref{item:2t},\eqref{item:3t}.  Any of the spurs may instead be exposed squares and vice versa.  The ladder 
case is not shown.  In each case, there is a ``central'' cell --- a $2$--cell or $0$--cell --- 
such that the diagram is formed from three padded ladders, each containing the central cell.}\label{fig:classification}
\end{figure}

The precise statement is Theorem~\ref{thm:strebel_cubical_small_can}, which explains what the ``$9$ types'' 
of disc diagram are (the nondegenerate ones are shown in Figure~\ref{fig:classification}); a \emph{padded ladder} is a disc 
diagram of the type in Figure~\ref{fig:better_ladder}; see 
Definition~\ref{defn:ladder}. 

\begin{figure}[h]
\begin{overpic}[width=0.9\textwidth]{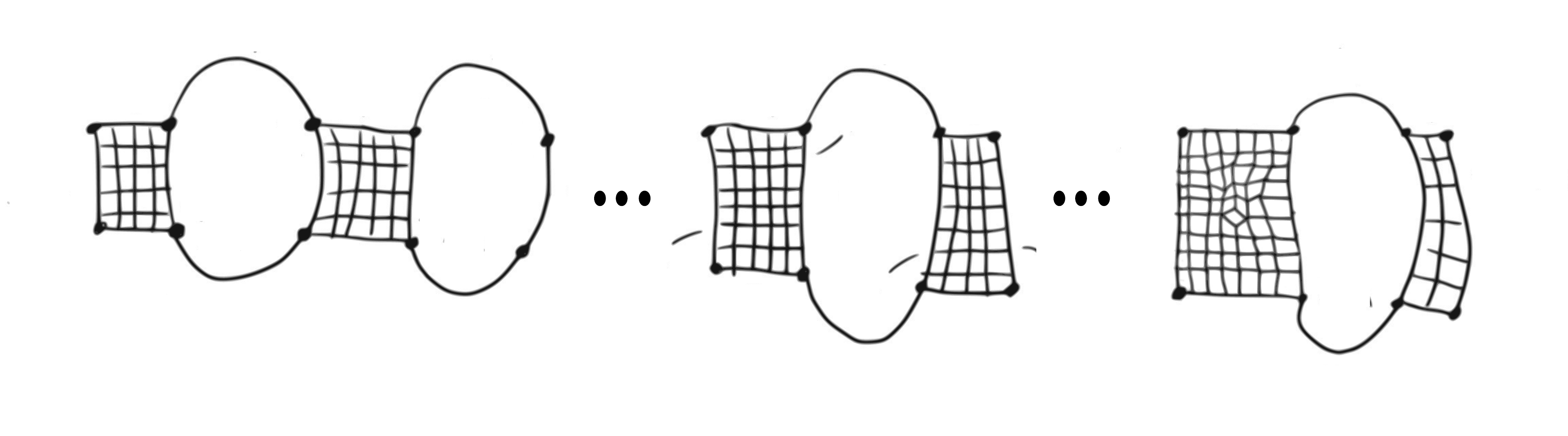}
 \put(15,25){$\alpha_1$}
 \put(30,24){$\alpha_2$}
 \put(55,24){$\alpha_i$}
 \put(85,23){$\alpha_n$}
 \put(14,8){$\gamma_1$}
 \put(29,7){$\gamma_2$}
 \put(54,4){$\gamma_i$}
 \put(84,3){$\gamma_n$}
 \put(47,21){$\rho_{i-1}$}
 \put(61,21){$\rho_i$}
 \put(47,7){$\vartheta_{i-1}$}
 \put(61,6){$\vartheta_i$}
 \put(67,11){$\mu_{i+1}$}
 \put(39,11){$\nu_{i-1}$}
 \put(54,20){$\mu_i$}
 \put(55,9){$\nu_i$}
 \put(55,15){$C_i$}
 \put(15,15){$C_1$}
 \put(85,15){$C_n$}
 \put(30,15){$C_2$}
 \put(3,15){$\nu_0$}
 \put(94,14){$\mu_{n+1}$}
\end{overpic}

\caption{A padded ladder, see Definition~\ref{defn:ladder}.}\label{fig:better_ladder}
\end{figure}

The $\alpha$ required in our proof is $\frac{1}{144}$.  Conceptually, our theorem says that any 
geodesic triangle bounds a disc diagram which is square-homotopic (fixing corners) to a disc diagram which is a ``thickened 
tripod''.  Theorem~\ref{thmi:classification_of_triangles} reduces to existing results in when $X$ or $\mathcal 
I$ are restricted. 

\begin{itemize}
\item If $\mathcal I=\emptyset$, then Theorem~\ref{thmi:classification_of_triangles} says that any three vertices in a 
CAT(0) cube complex determine a geodesic tripod, which is a consequence of the fact that CAT(0) cube complexes are simply 
connected cube complexes whose $1$--skeleta are median graphs~\cite{Chepoi:median}.

\item  If $X$ is a wedge of circles and each $Y_i$ is an immersed circle, then  $\langle X\mid\{Y_i\}_{i\in\mathcal I}\rangle$ is a classical 
$C'(\frac{1}{144})$ presentation, and the original Strebel classification for classical $C'(\frac16)$ groups applies, and 
follows from Theorem~\ref{thmi:classification_of_triangles}.

\item If X is a wedge of circles and each $Y_i$ is a graph with $Y_i\to X$ an immersion of graphs, then we have instances of graphical 
presentations.  In this setting, there is a classification of 
triangles that holds under weaker small cancellation conditions than are required in the cubical setting.  Indeed, the 
classification of triangles is completely combinatorial, and Strebel's proof actually applies in the setting of the 
\emph{$(3,7)$--diagrams} used by Gruber-Sisto in their proof of acylindrical hyperbolicity for graphical small cancellation 
groups~\cite{GruberSisto:graphical}; this combinatorial observation was made by Gruber~\cite[Remark~3.11]{Gruber:thesis}.  
While the result about $(3,7)$--diagrams suffices for graphical small cancellation groups, one cannot extend it directly to 
disk diagrams over cubical presentations since the presence of squares means that such diagrams need not satisfy the $(3,7)$ condition.
Moreover, Theorem~\ref{thmi:classification_of_triangles} includes the result on graphical presentations in their full generality (where $X$ is an arbitrary graph), under our small cancellation assumption.
\end{itemize}

One can construct examples of small cancellation cubical presentations covered by 
Theorem~\ref{thmi:classification_of_triangles} but not by the classical or graphical small cancellation conditions, and one 
cannot deduce Theorem~\ref{thmi:classification_of_triangles} from the corresponding purely cubical or graphical results.  
Some explicit examples of cubical small cancellation groups to which the theorem applies are discussed below.  

\subsection*{Applications of the classification}  One can imagine applications of 
Theorem~\ref{thmi:classification_of_triangles} to the thorough investigation of cubical small cancellation groups analogous 
to applications of Strebel's classification of triangles in classical small cancellation theory, e.g. conformal dimension of 
the boundary~\cite{Mackay:cd}, growth tightness \cite{Sambusetti:surf}, the rapid decay property~\cite{AD}, etc.

In this paper, we focus on acylindrical hyperbolicity, inspired by Gruber-Sisto's result for graphical small cancellation 
groups~\cite{GruberSisto:graphicalJ}.  A group $G$ is \emph{acylindrically hyperbolic} if it admits a nonelementary 
acylindrical action on a hyperbolic space (acylindricity generalizes \emph{uniform} properness).  Acylindrically hyperbolic 
groups were defined by Osin~\cite{Osin:acyl}, and the notion of an acylindrical action goes back at least to the work of 
Bowditch~\cite{Bowditch:tight}.  The 
notion of acylindrical hyperbolicity  unifies several generalizations of relative 
hyperbolicity~\cite{BestvinaFujiwara,DGO,Hamenstadt:cohom,Sisto:contracting} and provides a class of groups with many strong 
properties: if $G$ is acylindrically hyperbolic, then $G$ is SQ-universal, contains normal free subgroups, and is 
$C^*$--simple if and only if it has no finite normal subgroup~\cite{DGO}; $G$ contains Morse elements and thus all 
asymptotic cones of $G$ contain cut-points~\cite{Sisto:qc}; the bounded cohomology of $G$ has infinite dimension in 
dimensions $2$~\cite{HullOsin} and $3$~\cite{FPS}; every commensurating endomorphism of $G$ is an inner 
automorphism~\cite{AMS}, etc.  

The class of acylindrically hyperbolic groups is now known to be vast, see 
e.g.~\cite{Bowditch:tight,Osin:acyl,DGO,MinasyanOsin,Osin:l2,BestvinaFeighn,GruberSisto:graphicalJ,BHS:HHS_I,PrzytyckiSisto}. 
Our second result adds Wise's cubical small cancellation groups to this notable list.  To state it, we need some preliminary 
discussion.
We refer to Section~\ref{sec:cubical_pres} for the exact definition of a 
cubical presentation, and to Definition~\ref{defn:metric_small_c} for the cubical small cancellation conditions we mention 
shortly.  

\begin{defn}[Minimal cubical presentation]\label{defn:minimal}
We say that $\langle X\mid\{Y_i\}_{i\in\mathcal I}\rangle$ is a \emph{minimal cubical presentation} if for all $i\in\mathcal 
I$, the following holds.  Let $\widetilde Y_i\to Y_i$ be the universal cover, and let $\widetilde 
Y_i\hookrightarrow\widetilde X$ be a lift of $\widetilde Y_i\to Y_i\to X$.  Then $\stabilizer_{\pi_1X}(\widetilde Y_i)$ is 
conjugate in $\pi_1X$ to $\pi_1Y_i$ (rather than to a proper \emph{super}group of $\pi_1Y_i$).  For classical presentations, being 
minimal corresponds to the situation where relators 
are not proper powers. 
\end{defn}

Let $\langle X\mid\{Y_i\}_{i\in\mathcal I}\rangle$ be a cubical presentation.  Suppose for each $i\in\mathcal I$ 
that $\widehat Y_i\to Y_i$ is a finite connected regular cover.  Then the maps $\widehat Y_i\to Y_i\to X$ give a new cubical 
presentation $\langle X\mid\{\widehat Y_i\}_{i\in\mathcal I}\rangle$.  Suppose that $\langle X\mid\{Y_i\}_{i\in\mathcal I}\rangle$ satisfies 
the uniform $C''(\alpha)$ condition for some $\alpha\in[0,1)$.  Let $\beta=\frac{\|Y_i\|}{\|\widehat Y_i\|}$.  Then, as can 
be 
seen from Definition~\ref{defn:metric_small_c}, the cubical presentation $\langle X\mid\{\widehat Y_i\}_{i\in\mathcal I}\rangle$ satisfies 
the uniform $C''(\alpha\beta)$ condition.

% \begin{defn}[Good cubical presentation]\label{defn:good_cubical_presentation}
%   A minimal cubical presentation satisfying the uniform $C''(\alpha)$ condition for
% some $\alpha<1$ is called \emph{good}.
% \end{defn}

\begin{thmi}[Acylindrical hyperbolicity from cubical small cancellation]\label{thmi:main}
Let $X$ be a compact connected nonpositively curved cube complex and $\{Y_i\to X\}_{i\in\mathcal I}$ be a collection of 
local isometries with each $Y_i$ a compact connected nonpositively curved cube complex.
\begin{enumerate}[(I)]
     \item There exists a constant $L_0=L_0(X)$ so that the following 
holds.  Let $\alpha_0=\min\{\frac{1}{144},\frac{1}{7L_0}\}$ and let $\alpha\in[0,\alpha_0]$.  Let $\langle 
X\mid\{Y_i\}_{i\in\mathcal I}\rangle$ be a minimal cubical presentation satisfying the uniform $C''(\alpha)$ condition, and 
let $G$ be the group given by this cubical presentation.  Then one of the following holds:
\begin{enumerate}[(a)]
     \item  $G$ is finite or two-ended.\label{item:VC}
     \item The universal cover $\widetilde X$ of $X$ contains a convex $\pi_1X$--invariant subcomplex 
splitting as a product \label{item:prod}
of unbounded cube complexes, each $Y_i$ is contractible, and $G\cong\pi_1X$\label{item:product}.
\item $G$ is acylindrically hyperbolic.\label{item:AH}
\end{enumerate}
\item \label{item:good}Let $\langle X\mid\{Y_i\}_{i\in\mathcal I}\rangle$ be a cubical presentation satisfying the uniform $C''(\beta)$ condition for some $\beta<1$.  Then there exists 
$L_1=L_1(X,\{Y_i\to X\}_{i\in\mathcal I})$ such that the following holds.  For each $i\in\mathcal I$, let $\widehat Y_i\to Y_i$ be a 
finite connected regular cover such that the cubical presentation $\langle X\mid\{\widehat Y_i\}_{i\in\mathcal I}\rangle$ satisfies the 
uniform $C''(\alpha)$ condition for some $\alpha<\min\{\frac{1}{144},\frac{1}{7L_1}\}$.  Let $G$ denote the associated 
quotient of $\pi_1X$.  Then $$G\cong\pi_1X/\langle\langle\{\pi_1\widehat Y_i:i\in\mathcal I\}\rangle\rangle$$ satisfies one of ~\eqref{item:VC},\eqref{item:prod}, or~\eqref{item:AH}.  

In particular, $G$ satisfies one of the given conclusions whenever $\inf_{i\in\mathcal I}\|\widehat Y_i\|/\|Y_i\|$ is sufficiently large, in 
terms of $X$ and the $Y_i\to X$.
\end{enumerate}
\end{thmi}

The first statement is attractive since the small cancellation condition needed to get acylindrical hyperbolicity depends 
only on $X$.  The second statement is extremely useful in practice: one only needs to find a cubical presentation satisfying the uniform $C''(\beta)$ condition for some $\beta<1$ (but not 
necessarily satisfying e.g. the uniform $C''(\frac{1}{144})$ condition), with each $\pi_1Y_i$ residually finite, and then one can construct 
infinitely many acylindrically hyperbolic quotients of $\pi_1X$.

\begin{rem}[The constants $L_0$ and $L_1$]\label{rem:L} We give here a brief preview on the computation of these constants.
For simplicity, assume in this remark that $\widetilde X$ does not contain a proper $\pi_1X$--invariant convex subcomplex.
The general case is handled in the proof of Theorem~\ref{thmi:main} by replacing $\widetilde X$ by its Caprace-Sageev 
\emph{essential core}.  That is, in the general case, the constants $L_0,L_1$ are computed in the same way, except using the 
essential core of $\widetilde X$ instead of $\widetilde X$.

The constant $L_0(X)$ is defined as follows.  Let $\contact \widetilde X$ be the \emph{contact graph} of $\widetilde X$, 
which is the intersection graph of the set of hyperplane carriers; this was defined, and shown to be hyperbolic, 
in~\cite{Hagen:quasi_arb}.  Let $\mathcal L_0$ be the set of $\tilde g\in\pi_1X$ that act on 
$\contact \widetilde X$ as loxodromic isometries.  From results 
in~\cite{CapraceSageev:rank_rigidity} and~\cite{Hagen:boundary}, $\mathcal L_0\neq\emptyset$ if and only if 
$\widetilde X$ does not split as the product of unbounded CAT(0) cube 
complexes.  We can take $L_0(X)=\inf_{\tilde g\in\mathcal L_0}\inf_{\tilde a\in\widetilde 
X^{(0)}}\dist_{\widetilde X}(\tilde a,\tilde g\tilde a)$.  So, $L_0(X)<\infty$ unless $\widetilde X$ is a product.  Moreover, 
since we are considering combinatorial translation length, $L_0(X)\geqslant 1$.  We 
need minimality of the presentation only in the proof of Lemma~\ref{lem:acyl_regime}, where we need to know that an 
element $\tilde g\in\mathcal L_0$ of minimal translation length cannot stabilise some $\widetilde Y_i$.  Minimality 
guarantees that $\tilde g\in\stabilizer_{\pi_1X}(\widetilde Y_i)$ is conjugate into $\pi_1Y_i$; our small cancellation 
condition, which requires 
$\|Y_i\|>7L_0$, rules this out.

In case~\eqref{item:good}, we examine the elements $\tilde g\in\mathcal L_0$ that stay loxodromic on the hyperbolic graph 
obtained from the contact graph by coning off each $\widetilde Y_i$.  Lemma~\ref{lem:loxodromic_on_coned_contact_graph} says 
that either $\pi_1X$ stabilises some $\widetilde Y_i$ (leading to conclusion~\eqref{item:VC}), or the set $\mathcal L_1$ of 
such $\tilde 
g$ is nonempty.  We then take $L_1$ to be the minimal translation length of such $\tilde g$ on $\widetilde X$.  Note 
that $L_1$ depends on $X$ and the initial choices $Y_i\to X$, but the latter dependence is only via the action of $\pi_1X$ 
on $\widetilde X$ and the collection of $\widetilde Y_i\subseteq\widetilde X$.  So $L_1$ does not depend on the further covers 
$\widehat Y_i\to Y_i$.  The constant $L_1$ is used in one spot, in Lemma~\ref{lem:acyl_regime}, to ensure that $\tilde 
g\in\mathcal L_1$ does not stabilise any $\widetilde Y_i$, and can be chosen with translation length less than $\|\widehat 
Y_i\|/7$.
\end{rem}

\begin{rem}[The purely cubical case]\label{rem:purely_cubical}
In the ``purely cubical'' case, where $\mathcal I=\emptyset$, Theorem~\ref{thmi:main} 
says that the fundamental group of a compact nonpositively curved cube complex $X$ is either elementary or 
acylindrically hyperbolic, unless the universal cover of $X$ contains a $\pi_1X$--invariant convex subcomplex decomposing as a product with unbounded factors.  This  also follows from various results in 
the literature.  
In~\cite{BHS:HHS_I}, it is shown that, under 
natural extra hypotheses, the action of $\pi_1X$ on the contact graph $\contact\widetilde X$ is acylindrical, and $\contact\widetilde X$ is 
unbounded in the absence of an invariant product subcomplex.  Even without the extra hypotheses, any $g\in\pi_1X$ acting 
loxodromically on $\contact\widetilde X$ actually acts as a WPD element in the sense of~\cite{BestvinaFujiwara}, by~\cite[Proposition 
5.1]{BHS:HHS_I}.  Together with results in~\cite{Hagen:boundary} characterizing the loxodromic isometries of 
$\contact\widetilde X$, and a result of Osin connecting WPD elements to acylindricity~\cite{Osin:acyl}, this implies the 
virtually cyclic/product/acylindrically hyperbolic trichotomy of Theorem~\ref{thmi:main} in the case where $\mathcal 
I=\emptyset$.  This trichotomy, in the purely cubical case, also follows from the Caprace-Sageev rank rigidity 
theorem~\cite{CapraceSageev:rank_rigidity} and general results about groups acting on CAT(0) spaces and containing rank one 
elements~\cite{Osin:acyl,Sisto:contracting}.  

In the case where $\mathcal I=\emptyset$, our proof of Theorem~\ref{thmi:main} reduces to a proof that elements that are loxodromic on the contact graph 
act as WPD elements, along similar lines to the proof 
in~\cite{BHS:HHS_I}.
\end{rem}

\begin{rem}[Classical and graphical cases]  The comparison with the acylindrical hyperbolicity result of Gruber-Sisto, for 
graphical small cancellation groups (as formulated in~\cite{Gruber:thesis}), is interesting; our results about cubical small 
cancellation groups do not follow from corresponding results about graphical small cancellation presentations, since the 
latter viewpoint does not fully account for high-dimensional cubes.

On the other hand, restricting Theorems~\ref{thmi:classification_of_triangles} and~\ref{thmi:main} to the case where 
$\dimension X=1$ and each $Y_i$ is a graph, one does not reprove the results of~\cite{GruberSisto:graphicalJ} 
or~\cite{Strebel} in full generality, since the cubical $C'(\frac{1}{144})$ condition is more restrictive than the 
conditions needed in the classical and graphical cases (which are the classical $C'(\frac16)$ and the 
graphical $Gr(7)$ conditions, respectively).
\end{rem}

\begin{rem}[Acylindrical action on a quasi-tree]
Combining Theorem~\ref{thmi:main} with a recent result of Balasubramanya~\cite{Balasubramanya} shows that any group covered 
by Theorem~\ref{thmi:main} either satisfies one of the first two conclusions or 
acts acylindrically and non-elementarily on a quasi-tree.   
\end{rem}

\subsection*{On the proof of Theorem~\ref{thmi:main}} Theorem~\ref{thmi:main} is proved roughly as follows.  First, 
we create a hyperbolic $\pi_1X^*$--space $\hyp$ by coning off each hyperplane carrier $\neb(H)$ and each relator $Y_i$ in 
the generalized Cayley graph $\cay(X^*)$.  This procedure is a common generalization of the constructions used in the purely 
cubical case and in the graphical case from ~\cite{GruberSisto:graphicalJ}.

We apply 
Theorem~\ref{thmi:classification_of_triangles} to show that suitable $g\in\pi_1X^*$ acting loxodromically on $\hyp$ is a 
WPD isometry of $\hyp$; Osin's theorem says that $\pi_1X^*$ is acylindrically hyperbolic or virtually 
cyclic. This is done in Section~\ref{sec:metalemma}.  Hyperbolicity of $\hyp$ also uses
Theorem~\ref{thmi:classification_of_triangles} and is checked in Section~\ref{sec:hyp}.

It remains to find suitable loxodromic isometries of $\hyp$.  This is done in Section~\ref{sec:finding_loxodromic}. First, 
we show that if $\tilde g\in\pi_1X$ acts loxodromically on the contact graph  $\contact \widetilde X$, and axes of $\tilde g$ have 
suitably bounded interaction with elevations of relators $Y_i$, then the image $g\in\pi_1X^*$ of $\tilde g$ is loxodromic on 
$\hyp$. This is accomplished in Lemma~\ref{lem:loxodromic_persist_1}, Lemma~\ref{lem:characterising_a_shortest}, and 
Lemma~\ref{lem:loxodromics_persist_2}:
\begin{itemize}
     \item In the first lemma, we show that $\tilde g$ must remain loxodromic on the graph obtained from the contact graph 
by coning off the various subgraphs associated to the various $\widetilde Y_i\subseteq\widetilde X$, unless $\tilde g$ has a 
power conjugate into some $\pi_1Y_i$.
\item In the second lemma, we show that if $\tilde g$ has suitably bounded interaction with elevations of relators, then $g$ 
has infinite order and  $\tilde g$ is loxodromic on the coned-off contact graph.
\item In the third lemma, we deduce that $g$ must be loxodromic on $\hyp$.  The idea is to use disc diagrams in $\widetilde 
X^*$ to show that axes of $\tilde g$ in the coned-off contact graph project to quasigeodesic axes in $\hyp$ for $g$.  Since 
$\tilde g$ was loxodromic upstairs, $g$ must therefore be loxodromic downstairs.
\end{itemize}

The ``suitably bounded interaction'' hypothesis is made precise in 
Definition~\ref{defn:short_embeddable_asystolic}: $\tilde g$ must be \emph{asystolic}.  Up to this point, we only require 
the uniform $C''(\frac{1}{144})$ condition.

In Lemma~\ref{lem:acyl_regime}, we show that under either of the hypotheses in Theorem~\ref{thmi:main}, we can find an 
asystolic element.  The remainder of the proof is essentially an application of results 
in~\cite{CapraceSageev:rank_rigidity} and~\cite{Hagen:boundary} characterising when $\mathcal L_0\neq\emptyset$.

\begin{rem}[No proof by cubulation]\label{rem:defensive}
Cubical small cancellation theory is partly motivated by the fact that groups satisfying strong classical small cancellation 
conditions act nicely on CAT(0) cube complexes~\cite[Theorem 1.2]{Wise:small_can_cube}.  This generalizes in various ways to 
cubical presentations:
 if $\langle X\mid\{Y_i\}_{i\in\mathcal I}\rangle$ satisfies the \emph{generalized $B(6)$} condition, one can often cubulate the 
corresponding group; see, for instance,~\cite[Corollary 5.45]{Wise:QCH}.  

It is tempting to try to prove Theorem~\ref{thmi:main} using this approach, together with the above-mentioned results about 
acylindrical hyperbolicity of groups acting on cube complexes.  However, there are various problems with this approach.  
For example, the generalized $B(6)$ condition requires each $Y_i$ to have a wallspace structure, compatible with the local 
isometry $Y_i\to X$, generalizing the wallspace structure on a circle in which each wall is a pair of antipodal points.  
(Compare with the \emph{lacunary walling} condition on graphical presentations from~\cite{ArzhantsevaOsajda}.)

No cubical $C'(\alpha)$ small cancellation condition implies the generalized $B(6)$ condition, and indeed 
there are groups that are covered by Theorem~\ref{thmi:main} but which do not admit an action on a CAT(0) cube complex with 
no global fixed point.  This can already be seen in the $1$--dimensional case: Proposition~7.1 of~\cite{OllivierWise} yields, 
for any $\alpha>0$, a graphical presentation $\langle X\mid Y\rangle$, where $X$ is a graph and $Y\to X$ an immersed graph, 
satisfying the graphical (hence $1$--dimensional cubical) $C'(\alpha)$  condition, with the additional property that the 
group thus presented has Kazhdan's property $(T)$, and therefore cannot act fixed point-freely on a CAT(0) cube 
complex~\cite{NibloRoller}.
\end{rem}

\begin{rem}[Relationship with rotating families]\label{rem:rotating}
One can imagine an alternate approach to Theorem~\ref{thmi:main} using tools from~\cite{DGO}.  The idea would be to consider 
the action of $\pi_1X$ on (some graph quasi-isometric to) the contact graph $\contact\widetilde X$, and consider the images 
of the various $\widetilde Y_i$ under projection to $\contact\widetilde X$.  In cases where $\pi_1X$ acts acylindrically on 
$\contact\widetilde X$, it is not hard to deduce from the small cancellation conditions that $\pi_1X$ acts acylindrically on 
the graph $\widehat{\hyp}$ obtained by coning off the $\widetilde Y_i$-subgraphs of $\contact\widetilde X$, using results in~\cite{DGO}.  

However, we see the following challenge in a possible proof using the ``geometric small cancellation'' setup from~\cite{DGO}. Some hyperplane $H$ crossing $\widetilde Y_i$ can have unbounded intersection with $\widetilde Y_i$, because the 
definition of a wall-piece \emph{excludes the case of hyperplanes that actually intersect $\widetilde Y_i$} (see 
Definition~\ref{defn:abstract_wall_piece}).  So, $\stabilizer_{\pi_1X}(\widetilde Y_i)$ can contain elements of 
$\stabilizer_{\pi_1X}(H)$, which is a vertex-stabiliser in $\contact\widetilde X$.  So, $\stabilizer_{\pi_1X}(\widetilde 
Y_i)$ need not be purely loxodromic on $\contact\widetilde X$.  At the same time, we suspect it might still be possible to use a 
\emph{variant} of $\widehat{\hyp}$ where the $\stabilizer_{\pi_1X}(\widetilde Y_i)$ subgroups act as a \emph{rotating family} and 
allow one to use~\cite{DGO} to conclude that the action of $\pi_1X^*$ on this variant graph has WPD elements.

An explicit uniform $C''(\frac{1}{144})$ example where $\stabilizer_{\pi_1X}(\widetilde Y)$ is not purely loxodromic: let 
$X$ be a compact nonpositively curved cube complex so that some unbounded hyperplane $H$ of $\widetilde X$ has uniformly 
bounded coarse intersection with all other hyperplanes (including translates of $H$).  Let $\widetilde Y=H$, let $\bar 
Y=\stabilizer_{\pi_1X}(H)\backslash\, \mathcal N(H)$, and let $Y\to\bar Y$ be a regular cover of sufficiently large girth.  Then 
$\langle X\mid Y\to X\rangle$ is a cubical $C''(\frac{1}{144})$ presentation, but $\pi_1Y$ fixes a point in 
$\contact\widetilde X$.  For example, one can take a closed hyperbolic surface $S$, take a finite filling family of simple 
closed curves, no two of which are parallel, and take $\widetilde X$ to be the cube complex dual to the system of walls on 
$\widetilde S$ obtained by lifting the curves.  Then $H$ corresponds to an elevation of one of the curves, and $Y$ 
corresponds to a very long circle covering that curve.
% 
% Also, the $\pi_1X$--action on $\mathcal C\widetilde X$ is not known to be acylindrical without some additional 
% assumption on the compact nonpositively curved cube complex $X$.  This holds when $X$ is virtually special~\cite{BHS:HHS_I}, 
% and indeed holds for all compact nonpositively curved $X$ that we are aware of, by results in~\cite{HagenSusse,BHS:HHS_I}.  
% However, there is not yet a proof that $\pi_1X$ acts acylindrically on $\contact\widetilde X$ for an arbitrary compact $X$.  So 
% an approach using~\cite{DGO} would work in a less general setting than our Theorem~\ref{thmi:main}.
% 
% On the other hand, an approach using~\cite{DGO} would give a stronger conclusion, since it would prove acylindricity of the 
% action on a specific hyperbolic space $\hyp$.  Our approach only yields a WPD isometry of that space (and hence an 
% acylindrical action on some other space).
\end{rem}

\subsection*{Examples of cubical small cancellation groups}\label{subsec:examples}
We list here some examples of cubical small cancellation groups to which 
Theorem~\ref{thmi:classification_of_triangles} and Theorem~\ref{thmi:main} apply. We have earlier mentioned classical and 
graphical small cancellation presentations to which our results apply, as well as the purely cubical case where $X$ is a nonpositively 
curved cube complex and $\mathcal I=\emptyset$.

\begin{enumerate}
 \item Classical/RAAG hybrid: let $X$ be the Salvetti complex of a right-angled Artin group $A$, with presentation graph 
$\mathbb G$ (this means, by definition, that $A$ has a generator for each vertex of $\mathbb G$, with two generators commuting if and only if the corresponding vertices are adjacent in $\mathbb G$), and 
let $\{g_i\}_{i\in\mathcal I}$ be a finite collection of independent elements, none of which is supported 
on a 
proper join in $\mathbb G$ (i.e. each $g_i$ is a rank-one isometry of $\widetilde X$).  So, each $\langle g_i\rangle$ has a 
convex cocompact core $\widetilde Y_i$ in $\widetilde X$.  Then for each $i$ there exists $n_i>0$ so that, letting 
$Y_i=\langle g_i^{n_i}\rangle\backslash\widetilde Y_i$, the cubical presentation $\langle X\mid \{Y_i\}_{i\in\mathcal I}\rangle$ is a 
$C''(\frac{1}{144})$ presentation.  We will see below that $L_0(X)=2$ in this case.  So, if no $g_i$ is a proper power, 
Theorem~\ref{thmi:main} shows that $\langle X\mid \{Y_i\}_{i\in\mathcal I}\rangle$ presents a quotient of $A$ which is finite, virtually 
$\integers$, or acylindrically hyperbolic.  Even if the $g_i$ are proper powers, then taking $n_i$ sufficiently large (in 
terms of the $g_i$), we obtain the same conclusion from Theorem~\ref{thmi:main}.     

Furthermore, instead of cyclic subgroups, one could use appropriately chosen \emph{purely loxodromic} 
subgroups as described in~\cite{KMT}, which are necessarily free. 

 \item More generally, let $X$ be a compact connected nonpositively curved cube complex.  Let $\{Y_i\to 
X\}_{i\in\mathcal I}$ be a collection of local isometries with each $Y_i$ a compact connected nonpositively curved cube complex so that the resulting cubical presentation $\langle X\mid 
\{Y_i\}_{i\in\mathcal I}\rangle$ satisfies the (uniform) cubical $C''(\alpha)$ condition for some $\alpha>0$.  Suppose that each $Y_i$ has a 
residually finite fundamental group. Thus, for any $n\in\naturals$, there is a finite regular cover $\widehat Y_i\to 
Y_i$ with $\|\widehat Y_i\|\geqslant n\|Y_i\|$.  Thus, the related cubical presentation $\langle X\mid \{\widehat 
Y_i\}_{i\in\mathcal I}\rangle$ 
satisfies the (uniform) cubical $C''(\frac{\alpha}{n})$ condition.

If $\langle X\mid 
\{Y_i\}_{i\in\mathcal I}\rangle$ was a minimal presentation, $n$ can be chosen in terms of $X$ only when applying Theorem~\ref{thmi:main}.  
If not, then $n$ must be chosen sufficiently large in terms of $X$ and the initial $Y_i$.

 \item Given letters $x,y$ and $m\geqslant1$, let $(x,y)^m$ denote the first half of the word $(xy)^m$.  Consider the Artin 
group $$A=\langle a_1,a_2,\ldots,a_n\mid (a_i,a_j)^{m_{ij}}=(a_j,a_i)^{m_{ij}} \hbox{ whenever }i\neq j\rangle.$$
 (Note that we follow the convention of letting $m_{ij}=\infty$ to indicate that there is no relation between $a_i,a_j$.)
 Let $\widehat A=\langle a_1,a_2,\ldots,a_n\mid [a_i,a_j] \hbox{ whenever } m_{ij}=2\rangle$ be the underlying right-angled 
Artin group, $X$ be its Salvetti complex, and
 $\mathbb G$ be its presentation graph (so, a graph with a vertex for each $a_i$ and with an edge from $a_i$ to $a_j$ when 
$m_{ij}=2$).
That is, $A$ is a quotient of $\widehat A$ obtained by adding the relations $(a_i,a_j)^{m_{ij}}=(a_j,a_i)^{m_{ij}}$ when 
$m_{ij}\not\in\{2,\infty\}$.  Suppose that $\mathbb G$ does not decompose as a nontrivial join.  Moreover, suppose that for all 
$i,j$ with   $m_{ij}\not\in\{2,\infty\}$, the element $g_{ij}=(a_i,a_j)^{m_{ij}}(a_j,a_i)^{-m_{ij}}$ of $\widehat A$, is not 
supported in a join in $\mathbb G$.  Then $g_{ij}$ is a rank-one isometry of $\widetilde X$. 

Hence, there is a convex subcomplex $\widetilde Y_{ij}$ of 
$\widetilde X$ on which $\langle g_{ij}\rangle=\stabilizer_{\widehat A}(\widetilde Y_{ij})$ acts cocompactly, and which is 
just the convex hull of a combinatorial $g_{ij}$--axis.  Let $Y_{ij}$ be the quotient of $\widetilde Y_{ij}$ by the $\langle 
g_{ij}\rangle$--action, so that $\langle X\mid Y_{ij} \hbox{ whenever } 2<m_{ij}<\infty\rangle$ is a minimal cubical 
presentation for the Artin group $A$.  Clearly, $Y_{ij}$ has systole $2m_{ij}$.

If $\tilde P$ is a cone-piece between $\widetilde Y_{ij}$ and $\widetilde Y_{k\ell}$, then $|\tilde P|=1$.  Since $g_{ij}$ 
is not supported in a join, nontrivial wall-pieces have length $1$.

Finally, $L_0(X)=2$, since $\widehat A$ contains words of length $2$ that represent rank-one elements not stabilising 
hyperplanes.   
 
Suppose that $A$ satisfies the following: for all $i\neq j$, either $m_{ij}=2$ or $m_{ij}=\infty$ or $m_{ij}>72$ and no 
generator commutes with $a_i$ and $a_j$.  Then the above cubical presentation for  $A$ is $C''(\frac{1}{144})$ (and hence 
$C''(\frac{1}{7L_0})$).  So Theorem~\ref{thmi:main} implies $A$ is virtually cyclic or acylindrically hyperbolic.

There is a related recipe in Section~20 of~\cite{Wise:QCH}  for building $C(6)$ cubical presentations of Artin groups
(cf.~\cite{AppelSchupp}) but it is harder to see when these are $C''(\frac{1}{144})$.  
 \end{enumerate}

\subsection*{Outline of the paper}\label{subsec:outline}
In Section~\ref{sec:background}, we recall background on acylindrical hyperbolicity and WPD 
elements.  Section~\ref{sec:cubical_pres} contains a discussion of cubical presentations, disc diagrams, and 
the parts of cubical small cancellation theory needed in the proof of the classification of triangles, 
Theorem~\ref{thmi:classification_of_triangles}, which also occurs in this section.  The proof uses the theory developed 
in~\cite{Wise:QCH}.

In Section~\ref{sec:metalemma}, we give 
a list of conditions on a cubical small cancellation group $G$ acting on a hyperbolic space $\hyp$ sufficient to ensure that 
$G$ contains an element $g$ acting on $\hyp$ as a WPD element.  Specifically, we 
use Theorem~\ref{thmi:classification_of_triangles} to show that any $g\in G$ acting loxodromically on a space $\hyp$ 
satisfying the given conditions acts as a WPD element.    In Section~\ref{sec:hyp}, we produce such a space $\hyp$, formed 
from a generalized Cayley graph $\cay(X^*)$ by coning off both the relators and the hyperplane carriers.  
Theorem~\ref{thmi:classification_of_triangles} is also used here to check that $\hyp$ is hyperbolic.  Finally, in 
Section~\ref{sec:finding_loxodromic}, we prove Theorem~\ref{thmi:main}.

We assume basic knowledge of CAT(0) and nonpositively curved cube complexes and cubical presentations; we 
refer the reader to~\cite{Wise:QCH} for most of the background.  Most of the 
material that we will need 
from~\cite{Wise:QCH} is restated below, although for some more technical points we will refer the reader to~\cite{Wise:QCH} 
with various citations.

\subsection*{Acknowledgments}\label{subsec:Acknowledgments}
MFH thanks the Fakult\"at f\"ur Mathematik of Universit\"at Wien for hospitality during a visit in which much of the work on 
this project was completed, and is also grateful to Dani Wise for numerous discussions about cubical small cancellation 
theory over the years. Both authors thank Kasia Jankiewicz for answering a question about short inner paths, and the referee for a large number of very helpful comments.

\section{Acylindrical hyperbolicity and WPD elements}\label{sec:background}
The notion of an \emph{acylindrically hyperbolic group} was defined in~\cite{Osin:acyl}, as follows:

\begin{defn}[Acylindrical action, acylindrical hyperbolicity]\label{defn:acyl_action}
Let $(X,\dist)$ be a metric space and let $G$ act on $X$ by isometries.  Then the action is \emph{acylindrical} if for each 
$\epsilon\geqslant0$, there exists $R\geqslant0$ and $N\in\naturals$ so that for all $x,y\in X$ for which 
$\dist(x,y)\geqslant R$, we have $$|\{g\in G:\dist(x,gx)\leqslant\epsilon,\dist(y,gy)\leqslant\epsilon\}|\leqslant N.$$
Let $X$ be Gromov-hyperbolic and let $G$ act by isometries on $X$.  The action of $G$ is \emph{elementary} if the limit set 
of $G$ in $\boundary X$ has at most two points.  If $G$ acts non-elementarily and acylindrically on a hyperbolic space, then 
$G$ is \emph{acylindrically hyperbolic}.
\end{defn}

When $G$ is a cubical small cancellation group, we will construct an explicit action of $G$ on a hyperbolic space $\hyp$, 
but this will not necessarily be the action that witnesses acylindrical hyperbolicity.  Instead, the action will be such 
that $G$ contains a \emph{WPD isometry of $\hyp$}.

\begin{defn}[WPD element~\cite{BestvinaFujiwara}]\label{defn:wpd}
Let $G$ act by isometries on the space $X$.  Then $h\in G$ is a \emph{WPD element} if for each $\epsilon>0$ and each $x\in 
X$, there exists $M\in\naturals$ so that $$|\{g\in 
G:\dist(x,gx)\leqslant\epsilon,\dist(h^Mx,gh^Mx)\leqslant\epsilon\}|<\infty.$$
\end{defn}

In~\cite{Osin:acyl}, Osin showed that if $G$ is not virtually cyclic and acts on a hyperbolic space $\hyp$, and some $g\in 
G$ acts on $\hyp$ as a loxodromic WPD element, then $G$ is acylindrically hyperbolic.  This is instrumental in the 
proof of Theorem~\ref{thmi:main}.

\section{Triangles in cubical small cancellation groups}\label{sec:cubical_pres}
In this section, $X$ denotes a connected nonpositively curved cube complex with universal cover $\widetilde X$.  When doing 
geometry in $\widetilde X$, we never use the CAT(0) metric and instead only use the usual graph metric on $\widetilde 
X^{(1)}$ in which each $1$--cube has length $1$ and a combinatorial path is geodesic if and only if it contains at most one 
edge intersecting each hyperplane of $\widetilde X$.  

\begin{notation}[Carriers and neighbourhoods]
Let $X$ be a nonpositively curved cube complex and let $H\to X$ be an 
immersed hyperplane.  The \emph{(abstract) carrier} of $H$ is $\neb(H)=H\times[-\frac12,\frac12]$, equipped with the product 
cubical structure, where $H$ has a nonpositively curved cubical structure coming from $X$ and $[-\frac12,\frac12]$ is a 
$1$--cube.  The map $H\to X$ extends to a cubical map $\neb(H)\to X$, see e.g.~\cite[Section 2.g]{Wise:QCH}.  When $X$ is 
CAT(0), this map is an embedding whose image is a convex subcomplex $\neb(H)$, the \emph{carrier} of the hyperplane $H$.  In 
this case, $\neb(H)$ is just the union of the closed cubes that intersect $H$.

Given an arbitrary metric space $M$ and a subspace $H$, we denote by $\neb_r(H)$ the closed $r$--neighbourhood of $H$ in 
$M$.  The similarity in notation is justified by the fact that, when $X$ is a CAT(0) cube complex and $H$ is a hyperplane, 
$\neb(H)^{(0)}=\neb_{\frac12}(H\cap \widetilde X^{(1)})^{(0)}$.
\end{notation}

\subsection{Cubical presentations}  We fix a (possibly infinite) index set $\mathcal I$, and for each 
$i\in\mathcal I$, let $Y_i\to X$ be a local isometry of connected nonpositively curved cube complexes.  Each $Y_i\to X$ is a $\pi_1$--injection~\cite[Corollary 2.12]{Wise:QCH}.

Following~\cite{Wise:QCH}, the associated \emph{cubical presentation} is $\langle X\mid\{Y_i\}_{i\in\mathcal 
I}\rangle$ and the corresponding \emph{cubical presentation complex} $X^*$ is formed as follows.  For each $i\in\mathcal I$, 
let $C(Y_i)$ be the \emph{relator on $Y_i$}, i.e. the space formed from $Y_i\times[0,1]$ by collapsing $Y_i\times\{1\}$ to a 
point.  This space has an obvious cell-structure so that $Y_i\stackrel{\sim}{\to}Y_i\times\{0\}\hookrightarrow C(Y_i)$ is a 
combinatorial embedding.  For each $i\in\mathcal I$, we attach $C(Y_i)$ to $X$ along $Y_i\times\{0\}$ using the above local 
isometry.  The resulting complex is $X^*$. The group of our interest is defined by $G=\pi_1X^*$.  

We say that $\langle 
X\mid\{Y_i\}_{i\in\mathcal I}\rangle$ is a \emph{cubical presentation for $G$}. We refer to the complexes $Y_i$, or to the local isometries $Y_i\to X$, as \emph{relators} of such a cubical presentation.

The universal cover $\widetilde X^*$ of $X^*$ is a nonpositively curved cube complex with cones attached.  Let $\cay(X^*)$ 
be the part of $\widetilde X^*$ consisting only of cubes (i.e. the complement of the open cones).  This is the 
\emph{generalized Cayley graph} of $G$ with the given cubical presentation.  Note that there are covering maps $\widetilde 
X\to\cay(X^*)\to X$; the generalized Cayley graph is the nonpositively curved cube complex obtained by taking the cover of 
$X$ corresponding to the kernel of $\pi_1X\to\pi_1X^*$.

\begin{rem}(Classical and graphical presentations)
If $X$ is a wedge of circles and each $Y_i$ is an immersed combinatorial circle, then $\langle X\mid\{Y_i\}_{i\in\mathcal 
I}\rangle$ is a
group presentation in the usual sense (each $C(Y_i)$ is a disc) and $\cay(X^*)$ is the associated Cayley graph of $G$.  As 
mentioned in~\cite[Examples 3.s]{Wise:QCH}, if 
$X$ is a graph and each $Y_i$ is an immersed graph, then the above cubical presentation is a graphical presentation in the 
sense of~\cite{RipsSegev,Gromov:random,Ollivier}.
\end{rem}

\begin{rem}[Elevations]\label{rem:elevations}
The local isometries $Y_i\to X$ lift to local isometries $Y_i\to\cay(X^*)$ (in fact, under the small cancellation 
conditions we shall soon be assuming, the latter maps are embeddings~\cite[Section~4]{Wise:QCH}).  We use the term 
\emph{elevation} to refer to a lift $\widetilde Y_i\to\widetilde X$ of the map $\widetilde Y_i\to Y_i\to X$, where 
$\widetilde Y_i\to Y_i$ is the universal covering map.  Since $Y_i\to X$ is a local isometry, it is $\pi_1$--injective and 
$\widetilde Y_i\to\widetilde X$ is a combinatorial embedding with convex image.  
%We regard elevations $\widetilde 
%Y_i\to\widetilde X$ and $g\widetilde Y_i\to\widetilde X$ as identical if $g\in\stabilizer_{\pi_1X}(\widetilde Y_i)$. 
\end{rem}

\subsection{Cubical small cancellation conditions}   We now review background about cubical small cancellation theory, 
following~\cite{Wise:QCH}.  By a \emph{trivial path} in a 
cube complex, we mean a combinatorial path that does not traverse any edges, i.e. one whose image is a single 
vertex. Throughout the paper, all paths are combinatorial and, if $P$ is a combinatorial path, then $|P|$ denotes its 
length; in particular, $|P|=0$ when $P$ is trivial.  

% 
% Given a CAT(0) cube complex $\widetilde X$, and convex subcomplexes $U,V$, let $\mathrm{Proj}(U\to V)$ be the subcomplex of 
% $V$ defined as follows.  First, a closed $1$--cube $e$ of $V$ is in $\mathrm{Proj}(U\to V)$ if $e$ is dual to a hyperplane 
% intersecting $U$. Then add any cube of $V$ whose $1$--skeleton appears in $\mathrm{Proj}(U\to V)$.  Lemma~3.6 of~\cite{Wise:QCH} ensures 
% $\mathrm{Proj}(U\to V)$ is convex.

\begin{defn}[Wall-projection]\label{defn:wall_proj}
Let $\widetilde X$ be a CAT(0) cube complex, and let $U,V$ be convex subcomplexes.  Let $\mathrm{Proj}(U\to V)$ be the subcomplex of $V$ defined as follows.  

First, two cubes $c,c'$ of $\widetilde X$ are \emph{parallel} if they intersect exactly the same hyperplanes. In particular, any two $0$--cubes are parallel.

A cube $\bar u$ of $V$ belongs to $\mathrm{Proj}(U\to V)$ if and only if there is a cube $u$ of $U$ such that $\bar u$ and $u$ are parallel, and $\bar u$ is the nearest cube of $V$ that is parallel 
to $u$~\cite[Definition 3.4]{Wise:QCH}.
\end{defn}

% We will often work with a characterisation of $\mathrm{Proj}(U\to V)$ using a more standard notion from the literature on cube complexes and median graphs, using the following definition:

\begin{rem}[Gate maps]
Throughout, we will often want to closest-point project to convex subcomplexes in a CAT(0) cube complex.  The way to do this 
is to use the \emph{gate map}.  Let $\widetilde X$ be a CAT(0) cube complex and let $C$ be a convex subcomplex.  Then there 
is a map $\gate_C\colon\widetilde X\to C$ with the following properties:
\begin{itemize}
     \item $\gate_C$ is the closest point projection with respect to the combinatorial metric, i.e. for all $x\in\widetilde 
X^{(0)}$, the vertex $\gate_C(x)$ is the unique closest vertex of $C$ to $x$.
     \item $\gate_C$ is $1$--lipschitz on $\widetilde X^{(0)}$, with respect to the combinatorial metric.
     \item If $x\in\widetilde X$ and $H$ is a hyperplane, then $H$ separates $x,\gate_C(x)$ if and only if $H$ separates 
$x,C$.
\item If $x,y\in\widetilde X$ and $H$ is a hyperplane, then $H$ separates $\gate_C(x),\gate_C(y)$ if and only if $H$ 
intersects $C$ and separates $x,y$.
\end{itemize}
More information about gate maps can be found in e.g.~\cite[Section 2.1]{BHS:HHS_I}.    
\end{rem}

\begin{rem}[Gate maps and wall-projections]
Lemma 3.6 in~\cite{Wise:QCH} implies that if $U,V\subseteq\widetilde X$ are convex subcomplexes, then $\mathrm{Proj}(U\to V)=\gate_V(U)$, which is itself convex.
\end{rem}

Now we can define \emph{pieces}.  Let $\langle X\mid\{Y_i\}_{i\in\mathcal I}\rangle$ be a cubical presentation.

\begin{defn}[Wall-piece]\label{defn:abstract_wall_piece}
     Let $Y_i\to X$ be a relator.  Let $\widetilde Y_i\to\widetilde X$ be an elevation.  Let $H$ be a hyperplane of 
$\widetilde X$ that is \textbf{disjoint} from $\widetilde Y_i$.  Suppose that the subcomplex $\mathcal 
P=\mathrm{Proj}(\mathcal N(H)\to\widetilde Y_i)$ is not a single point.  Then $\mathcal P$ is an \emph{abstract wall-piece} in 
$\widetilde Y_i$.  We have a map $\mathcal P\hookrightarrow\widetilde Y_i\to Y_i\to X$.  A path $P\to\mathcal P$ is a 
\emph{wall-piece} of $H$ in $Y_i$, and gives a path $P\to Y_i$ via the preceding map.  If $\mathcal N(H)\cap\widetilde 
Y_i\neq\emptyset$, then $\mathcal P=\mathcal N(H)\cap\widetilde Y_i$, and we say that $P$ is a \emph{contiguous wall-piece}. 
\end{defn}

\begin{defn}[Cone-piece]\label{defn:abstract_cone_piece}
Let $Y_i,Y_j\to X$ be relators.  Fix an elevation $\widetilde Y_i\to\widetilde X$.  An \emph{abstract cone-piece} of $Y_j$ 
in $Y_i$ is defined as follows.  Let $\widetilde Y_j\to\widetilde X$ be an elevation of $Y_j$, and let $\mathcal 
P=\mathrm{Proj}(\widetilde Y_j\to\widetilde Y_i)$.  

Then:
\begin{itemize}
     \item If $i\neq j$, the complex $\mathcal P$ is an abstract cone-piece of $Y_j$ in $Y_i$.
     \item If $i=j$ and $\widetilde Y_i\neq\widetilde Y_j$ (i.e. they are distinct $\pi_1X$--translates of $\widetilde 
Y_i$), then $\mathcal P$ is an abstract cone-piece of $Y_j$ in $Y_i$.
    \item If $i=j$ and $\widetilde Y_i=\widetilde Y_j$, then $\mathcal P$ is an abstract cone-piece  \emph{unless} the 
following holds:  Let $\mathcal P\to Y_i,Y_j$ be obtained by composing $\widetilde Y_i,\widetilde Y_j\to Y_i,Y_j$ with the 
inclusions of $\mathcal P$ into $\widetilde Y_i,\widetilde Y_j$ respectively.  Then there exists an automorphism $Y_i\to 
Y_j$ such that 
\begin{center}
     $
     \begin{diagram}
          \node{\mathcal P}\arrow{e}\arrow{s}\node{Y_i}\arrow{s}\arrow{sw}\\
          \node{Y_j}\arrow{e}\node{X}
     \end{diagram}
     $
\end{center}
commutes.
\end{itemize}
If $\mathcal P$ is an abstract cone-piece in $Y_i$, then we again have a map $\mathcal P\hookrightarrow\widetilde Y_i\to 
Y_i$.  A path $P\to\mathcal P$ is a \emph{cone-piece} in $Y_i$, and gives a path $P\to Y_i$, which we will refer to as a 
\emph{cone-piece} in $Y_i$ without reference to the elevations involved.  If $\widetilde Y_i,\widetilde Y_j$ intersect, then 
$P\to Y_i$ is a \emph{contiguous} cone-piece. 
\end{defn}

\begin{exmp}[Cubical proper powers]\label{exmp:proper_power}
The third case in Definition~\ref{defn:abstract_cone_piece} generalises to cubical small cancellation setting how relators that are proper powers are treated in classical small cancellation theory.  For instance, regard the presentation 
$\langle a,b\mid (ab)^2\rangle$ as a cubical presentation by taking $X$ to be a wedge of two oriented circles labelled $a,b$ and $Y$ to be a $4$--cycle mapping to $X$ according to the word $abab$.  
So, in the tree $\widetilde X$ (the Cayley graph of  the free group on $a,b$), we can take $\widetilde Y$ to be the axis of the element $ab$.  The lines $\widetilde Y$ and $ab\widetilde Y$ coincide --- the second is 
a translate of the first by $ab$.  Since $ab$ descends to an automorphism of $Y\to X$, the line $\widetilde Y=\mathrm{Proj}(ab\widetilde Y\to\widetilde Y)$ is \emph{not} an abstract 
cone-piece.   
\end{exmp}

\begin{rem}[Variant definition of a cone-piece]\label{rem:3.3}
As explained in Section~3 of~\cite{Wise:QCH}, cubical small cancellation theory works under a slightly different definition 
of a cone-piece.  Indeed, the definition can be changed in the following way.  Given $\widetilde Y_i$ as in 
Definition~\ref{defn:abstract_cone_piece}, we forbid $\mathrm{Proj}(\widetilde Y_j\to\widetilde Y_i)$ from being an abstract cone-piece whenever $i=j$ and $\widetilde Y_i=\widetilde Y_j$ (i.e. 
the elevations differ by an element of 
$\stabilizer_{\pi_1X}(\widetilde Y_i)$).  However, to use this definition, one needs to insist that each element of 
$\stabilizer_{\pi_1X}(\widetilde Y_i)$ projects to an automorphism of $Y_i\to X$.  Convention~3.3 in~\cite{Wise:QCH} 
arranges this by insisting that $\pi_1Y_i$ is normal in $\stabilizer_{\pi_1X}(\widetilde Y_i)$.  We have opted to avoid this 
convention, and use the definitions above, following Przytycki-Wise and Jankiewicz~\cite{PrzytyckiWise:mixed,Jankiewicz}.

Under our definition, we still have the following: if $\mathcal P=\mathrm{Proj}(\widetilde Y_j\to\widetilde Y_i)$, and 
$\mathcal P$ is not an abstract cone-piece in $Y_i$, then $i=j$ and $\widetilde Y_j=g\widetilde Y_i$ for some 
$g\in\stabilizer_{\pi_1X}(\widetilde Y_i)$.  Indeed, if $\mathcal P$ is not an abstract cone-piece in $Y_i$, then by definition, $i=j$ and there is an automorphism $\bar g:Y_i\to Y_i$ such that the diagram in 
Definition~\ref{defn:abstract_cone_piece} commutes.  Since the automorphism group of $Y_i\to X$ is $$\mathrm{Normaliser}_{\stabilizer_{\pi_1X}(\widetilde Y_i)}(\pi_1\widetilde Y_i)/\pi_1\widetilde 
Y_i,$$ we can lift $\bar g$ to some $g\in\stabilizer_{\pi_1X}(\widetilde Y_i)$.  

However, under our definition, the converse does not hold: if 
$g\in\stabilizer_{\pi_1X}(\widetilde Y_i)$ does not normalise $\pi_1\widetilde Y_i$, then let $\widetilde Y_j=g\widetilde Y_i$.  Let $\mathcal P=\mathrm{Proj}(\widetilde Y_j\to\widetilde 
Y_i)=\widetilde Y_i$.  Then 
since $g\colon\widetilde Y_i\to\widetilde Y_j$ does not descend to an automorphism of $Y_i\to X$, our definition says that the 
whole of $\widetilde Y_i$ is an abstract cone-piece.  

Under the metric small cancellation conditions we'll be using, this then implies that $Y_i$ is simply connected, and can be 
removed from the cubical presentation without changing the group being presented or weakening the small cancellation 
conditions on what remains.

So when working under any of the small cancellation conditions used in this paper, there would have been no loss of generality in assuming that $\pi_1Y_i$ is normal in 
$\stabilizer_{\pi_1X}(\widetilde Y_i)$ and using the definition of a cone-piece from~\cite{Wise:QCH}.  (Relatedly, very often we'll work with \emph{minimal} cubical presentations, see Definition~\ref{defn:minimal}, which have 
the built-in assumption that $\stabilizer_{\pi_1X}(\widetilde Y_i)$ is conjugate to $\pi_1Y_i$, or with cubical presentations that are obtained from the minimal ones by replacing the 
$Y_i$ with finite regular covers.)
\end{rem}

\begin{defn}[Piece]
A path $P\to Y_i$ which is either a cone-piece or a wall-piece is a \emph{piece}.
\end{defn}

In the case where $X$ is a wedge of circles and each $Y_i$ is an immersed circle, all wall-pieces are trivial, and cone-pieces 
correspond to pieces in the sense of classical small cancellation theory.

\begin{defn}[$C'(\alpha)$ condition, uniform 
$C''(\alpha)$ condition]\label{defn:metric_small_c}
The cubical presentation $\langle X\mid\{Y_i\}_{i\in\mathcal I}\rangle$ satisfies the \emph{cubical $C'(\alpha)$ 
small cancellation condition} if the following holds for all $i\in\mathcal I$: 
$\diam(\mathcal P)<\alpha\|Y_i\|$ for all abstract pieces $\mathcal P$ in 
$Y_i$, where $\|Y_i\|$ denotes the infimum of the lengths of essential closed 
paths in $Y_i$ (a closed path is \emph{essential} if it is not homotopic to a constant map).
In this case, we say that $\langle X\mid\{Y_i\}_{i\in\mathcal I}\rangle$ is a \emph{$C'(\alpha)$ presentation} and 
$G=\pi_1X^*$ is a \emph{$C'(\alpha)$ group}.

Note that if $|\mathcal I|<\infty$, then the $C'(\alpha)$ condition yields a 
uniform bound on the length of all pieces, namely $\alpha\max_{i\in\mathcal 
I}\|Y_i\|$.  

In Section~\ref{sec:finding_loxodromic}, we will use the stronger \emph{uniform 
$C''(\alpha)$ condition}.  The cubical presentation $\langle 
X\mid\{Y_i\}_{i\in\mathcal I}\rangle$ satisfies the \emph{uniform  
$C''(\alpha)$ small cancellation condition} if $\diam(\mathcal P)<
\alpha\|Y_i\|$ for all $i$, whenever $\mathcal P$ is an abstract piece (not 
necessarily in $Y_i$).  In this case, we say that $\langle X\mid\{Y_i\}_{i\in\mathcal I}\rangle$ is a \emph{$C''(\alpha)$ presentation} and 
$G=\pi_1X^*$ is a \emph{$C''(\alpha)$ group}. This condition is needed to maintain an upper bound on 
the sizes of pieces, needed, for example, in the proof of 
Lemma~\ref{lem:loxodromic_persist_1}.

Another way to phrase the uniform condition is to let 
$Y=\bigsqcup_{i\in\mathcal I}Y_i$, so that the local 
isometries $Y_i\to X$ induce a local isometry $Y\to X$.  Then the uniform 
$C''(\alpha)$ condition for $\langle X\mid\{Y_i\}_{i\in\mathcal I}\rangle$ is 
equivalent to the (non-uniform) cubical $C'(\alpha)$ condition for the presentation 
$\langle X\mid\{Y\}\rangle$ (except allowing disconnected relators).  This 
should be compared to the small cancellation conditions in~\cite[Section 
2.2]{GruberSisto:graphicalJ}: in both cases, infinitely many relations can be 
encoded in a single cube complex (a possibly infinite, disconnected graph in 
the graphical case, $Y$ here), and it's the systole of that complex that is used 
in the uniform small cancellation condition.
\end{defn}

\subsection{Disc diagrams}
The key objects in cubical small cancellation theory are \emph{disc diagrams}.

\begin{defn}[Disc diagram, boundary path]\label{defn:disc_diagram}
A \emph{disc diagram} is a compact, contractible $2$--dimensional cell complex $D$ equipped with a fixed embedding in 
$\reals^2$.  We regard $\mathbb S^2$ as $\reals^2\cup\{\infty\}$, so that $\mathbb S^2$ is obtained from $D$ by attaching a 
$2$--cell containing $\infty$.  The attaching map of this $2$--cell is the \emph{boundary path} $\boundary_pD$ of $D$.
\end{defn}

Given a cubical presentation $\langle X\mid\{Y_i\}_{i\in\mathcal I}\rangle$ and a closed path $P\to X$ that is nullhomotopic 
in $X^*$, van Kampen's lemma provides a disc diagram $(D,\boundary_pD)\to(X^*,X)$ whose boundary path $\boundary_pD=P$.  
Closed paths in $\cay(X^*)$ bound disc diagrams $(D,\boundary_pD)\to(\widetilde X^*,\cay(X^*))$. 

The $2$--cells of such a 
diagram are either squares (mapping to $2$--cubes of $X\subseteq X^*$) or $2$--simplices mapping to cones over the various 
$Y_i$.  Since $P$ avoids cone-points, the $2$--simplices of $D$ are partitioned into classes: for each vertex of $D$ mapping 
to a cone-point in $X^*$, the incident $2$--simplices are arranged cyclically around the vertex to form a subspace $C$ of 
$D$ which is equal to the cone on its boundary path (a path in $D$ mapping to $X$).  The subspace $C$ is a \emph{cone-cell}. 
 In practice, we ignore the subdivision of $C$ into $2$--simplices and regard $C$ as a $2$--cell of $D$.

The \emph{complexity} of $D$ is the pair $(c,s)$, where $c$ is the number of cone-cells and $s$ is the number of squares.  
Taking the complexity in lexicographic order, we always consider diagrams $(D,\boundary_pD)\to(X^*,X)$ which are 
\emph{minimal} in the sense that the complexity of $D$ is lexicographically minimal among all diagrams with boundary path 
$\boundary_pD$.  This implies that for each cone-cell $C$, the path $\boundary_pC\to D\to Y\to X$ is essential.

In general, the cone-cell $C$ might not be a subdiagram of $D$ (it's true that $C$ is a subspace of $D$, but $C$ itself might not be a disc diagram, in the situation where $\boundary C$ 
self-intersects).  However, under the small cancellation hypotheses used 
throughout this paper, Remark 3.10 (see also Theorem 4.1) of~\cite{Wise:QCH} ensures that in any minimal-complexity disc 
diagram, the cone-cells are actually homeomorphic to discs and, in particular, are subdiagrams.  We will assume this freely 
throughout.

\begin{rem}[Dual curves and hexagon moves]\label{defn:dual_curve_square_ladder}
Let $D\to X$ be a square diagram, i.e. a disc diagram whose 2--cells are squares.  A \emph{dual curve} in $D$ is a path which is the concatenation of midcubes of squares of 
$D$ that starts and ends on $\boundary_pD$, where a \emph{midcube} of a square $[-\frac12,\frac12]^2$ is obtained by 
restricting exactly one coordinate to $0$ and a midcube of a $1$--cube is its midpoint.  If $X$ is a nonpositively curved 
cube complex, then each dual curve maps to a hyperplane.  If $K$ is a dual curve in $D$, then the union of all closed cubes 
intersecting $K$ is its \emph{carrier} (in analogy to the definition for the carrier of a hyperplane).

More generally, if $D\to X^*$ is a disc diagram, then one can define dual curves as above, but any dual curve has its two 
ends either on $\boundary_pD$ or on the boundary path of a cone-cell of $X^*$.

A \emph{hexagon move} is a homotopy of the diagram $D\to X^*$ that fixes the boundary path and the cone-cells and their 
boundary paths, while modifying the square part of $D$.  Specifically, if $s_1,s_2,s_3$ are squares in $D$ arranged 
cyclically around a central vertex $v$, forming a hexagonal subdiagram $E$ of $D$, then $X$ must contain a $3$--cube $c$ 
with a corner at the image of $v$ formed by the images of $s_1,s_2,s_3$.  The (hexagonal) boundary path of $E$ maps to a 
combinatorial path in $c$, and we can replace $E$ by a diagram $E'$ formed from the other $3$ squares on the boundary of 
$c$; this yields a new diagram $D'\to X^*$, with the same boundary path as $D$, formed by replacing $E$ by $E'$.  This 
modification is a hexagon move.  Hexagon moves are used to reduce area (i.e. the number of squares) of diagrams and, hence, their complexity in various ways; detailed accounts can be found in 
e.g.~\cite{Wise:QCH,Wise:CBMS}.
\end{rem}

\begin{defn}[External cone-cell, internal cone-cell, internal path]\label{defn:external_internal}
The cone-cell $C$ of the disc diagram $D$ is \emph{external} if $\boundary_pC=QS$, where $Q$ is a non-trivial  subpath of 
$\boundary_pD$ (i.e. containing at least one 1--cell) and $S$ is an \emph{internal} path in the sense that no $1$--cell of 
$S$ lies on $\boundary_pD$.  The cone-cell $C$ is \emph{internal} if $\boundary_pC$ and $\boundary_pD$ have no common 
non-trivial subpath.  A cone-cell can be internal, external, or neither.
\end{defn}

\begin{rem}[Rectification and angling]\label{rem:rectification_and_angling}
Given a disc diagram $(D,\boundary_pD)\to(X^*,X)$, one can \emph{rectify} $D$, to produce a \emph{rectified diagram} $\bar 
D$, by removing some internal open $1$--cells, so that $D$ is subdivided into cone-cells, \emph{rectangles} which are 
obtained from square ladders by deleting the internal open $1$--cells, and complementary regions called \emph{shards}.  
See~\cite[Section 3.f]{Wise:QCH} for more discussion of rectified diagrams.  We will not require further details here.
 
After rectifying $D$, each corner in each of the resulting $2$--cells is assigned an \emph{angle} according to one of 
several possible schemes.  We follow the \emph{split-angling} defined in~\cite[Section 3.h]{Wise:QCH}.  Specifically, if $v$ 
is a vertex of the rectified diagram $D$, and $c$ is an edge in the link of $v$ (i.e. a corner of a $2$--cell at the vertex 
$v$), then we assign an angle $\angle(c)$ according to rules discussed in~\cite[Section 3.h]{Wise:QCH}.  Since we will just 
be using \emph{consequences} of these angle assignments, rather than the exact (long) definition, we refer the reader 
to~\cite[Section 3.h]{Wise:QCH}.  Suffice it to say that:
\begin{itemize}
     \item the angle $\angle(c)$ is always $\pi/2,\pi,2\pi/3,3\pi/4$, or $0$;
     \item the choice of angle is made in a way that guarantees nonpositive curvature at shards, in the sense 
described momentarily.
\end{itemize}
\end{rem}

\begin{rem}[Defects and curvature]\label{rem:curvature}
We now review some notions of curvature, from~\cite[Section 3.g]{Wise:QCH}, that we will require below.  
Given a 
rectified disc diagram $\bar D$, we assign an \emph{angle} $\angle(c)$ -- a real number -- to each corner $c$ of each 
$2$--cell (i.e. to each $1$--cell of each vertex-link).  In our setting, we always assume that this is done using the 
split-angling convention.

The \emph{defect} $\mathfrak d(c)$ at the corner $c$ is $\mathfrak 
d(c)=2\pi-\angle(c)$.  The \emph{curvature} $\kappa(v)$ at a vertex $v$ of $\bar D$ is 
$\kappa(v)=2\pi-\sum\angle(c)-\pi\chi(\link(v))$, where $\link(v)$ is the link of $v$, the notation $\chi$ means the Euler 
characteristic,  and the sum is taken over the 
$1$--cells $c$ of $\link(v)$.  The \emph{curvature} $\kappa(f)$ at a $2$--cell $f$ of $\bar D$ is 
$\kappa(f)=2\pi-\sum\mathfrak d(c)$, where $c$ varies over the corners of $f$. 
\end{rem}

We will need the following theorem, which follows immediately from the ``combinatorial Gauss-Bonnet Theorem'' as stated 
in~\cite[Theorem~4.6]{McCammondWise}; very similar statements can be found in~\cite{Brin,Gersten,BallmannBuyalo}.

\begin{thm}[Gauss-Bonnet for diagrams]\label{thm:CBGT}
Let $\bar D\to X^*$ be a rectified disc diagram.  Then $$\sum_f\kappa(f)+\sum_v\kappa(v)=2\pi,$$ where $f$ varies over the 
$2$--cells of $\bar D$ and $v$ varies over the $0$--cells of $\bar D$.
\end{thm}

In the case where $X$ is a wedge of circles and each $Y_i$ is an immersed circle, i.e. $X^*$ is an ordinary presentation 
complex, disc diagrams are ordinary van Kampen diagrams, all rectangles are single edge, and rectifying has no effect on the 
diagram.  In this case, all $2$--cells of $D$ are cone-cells, the split-angling continues to ensure that the curvature at 
each vertex is nonpositive, and the condition on the curvature of shards is vacuous.

\begin{defn}[Generalized corner, spur, shell]\label{defn:positive_curvature}
A \emph{(positively-curved) shell} $C$ in the disc diagram $D$ is an external cone-cell whose curvature is positive; the 
boundary path of a shell has the form $QS$, where the \emph{outer path} $Q$ is a subpath of the boundary path of $D$, and 
the \emph{inner path} $S$ has no open $1$--cell on $\boundary_pD$. 

A \emph{spur} in $D$ is a vertex $v$ in $\boundary_pD$ so 
that the incoming and outgoing $1$--cells of $\boundary_pD$ map to the same $1$--cell of $X$, i.e. $v$ is the second vertex 
in a subpath of $\boundary_pD$ of the form $ee^{-1}$, where $e\to X$ is a $1$--cell.  

A \emph{generalized corner} is a path 
$ef$ in $D$, where each of $e,f$ is an edge, so that the dual curves emanating from $e,f$ cross inside a square $s$ of $D$, 
as shown in Figure~\ref{fig:corners}, and the subdiagram of $D$ bounded by the carriers of $e$ and $f$ is a square 
diagram.  See Definition 2.5 in~\cite{Wise:QCH}.
\end{defn}

\begin{figure}[h]
 \begin{overpic}[width=0.4\textwidth]{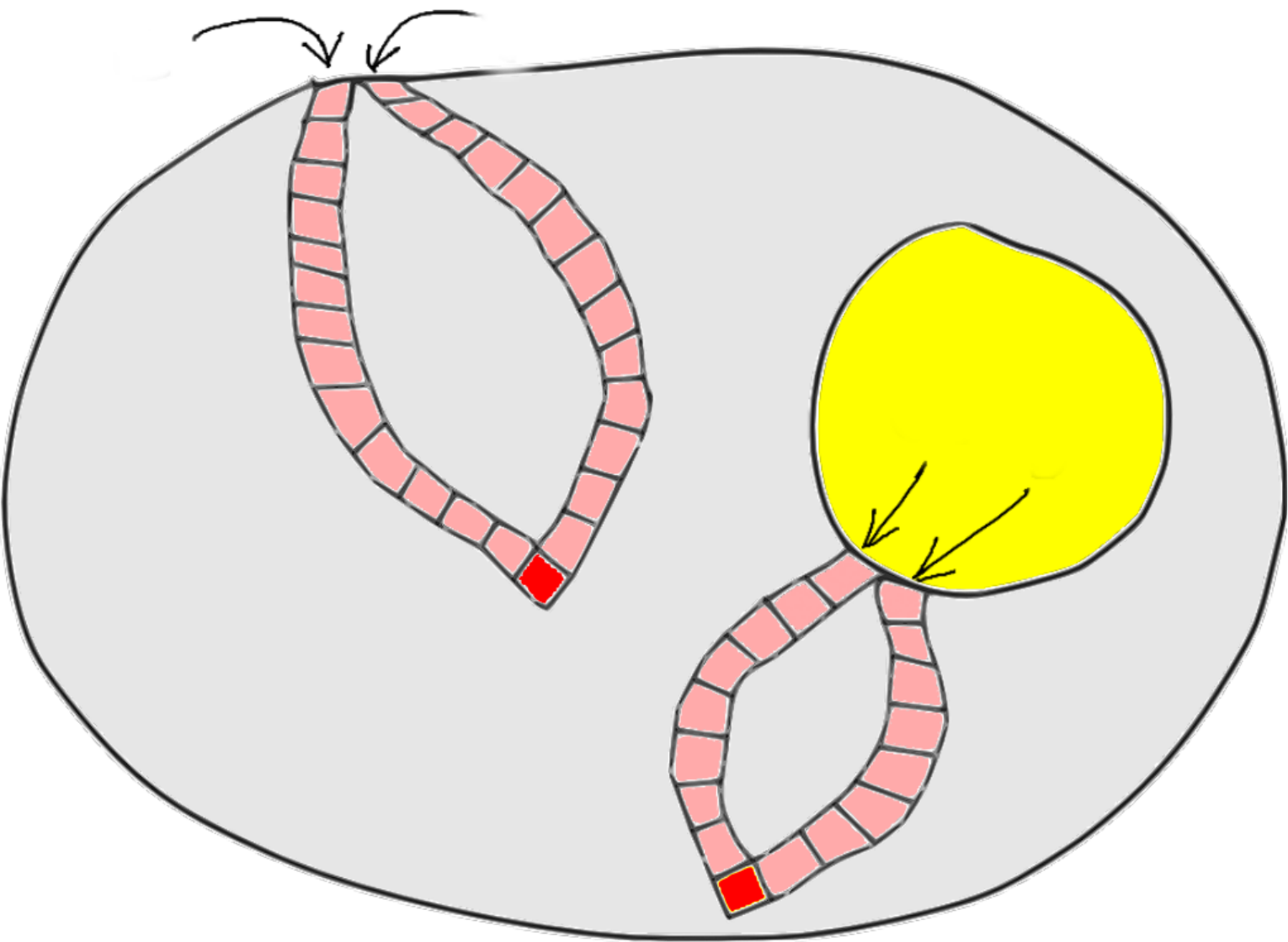}
      \put(11,69){$e$}
      \put(37,70){$f$}
      \put(72,38){$a$}
      \put(80,36){$b$}
 \end{overpic}

 \caption{$ef$ and $ab$ are generalized corners of the shaded squares.  $ef$ lies on the boundary of $D$, while $ab$ lies on 
the boundary of a cone-cell.}\label{fig:corners}
\end{figure}

\begin{rem}[Pushing generalised corners to the boundary by \textbf{\emph{shuffling}}]\label{defn:pushing_to_boundary}
If $ef$ is a generalized corner of a square $s$, and $ef$ lies along $\boundary_pD$, then we can perform a series of hexagon 
moves (see~\cite[Section~2.e]{Wise:QCH}) to homotop $D$, fixing its boundary path, so that there is a square with boundary 
path $efe'f'$, i.e. we can move squares to the boundary.  In~\cite{Wise:QCH}, this procedure is called ``shuffling''.  Later 
in the paper, we will occasionally invoke shuffling to modify disc diagrams $D$ so that a generalised corner $ef$ in 
$\boundary_pD$ is actually an \emph{exposed square}, i.e. there is a square $s$ in $D$ with boundary path containing $ef$.

If $ef$ is a generalized corner of a square $s$, and $ef$ lies on the boundary of some cone-cell $C$ mapping to a relator $Y$,  
then we can again shuffle until the square $s$ has two consecutive edges on $\boundary_pC$.  Convexity of $Y$ 
allows us to ``absorb'' the square $s$ into $C$, lowering complexity of $D$.    (In~\cite{Wise:QCH}, what we call a 
\emph{generalised corner} is called a \emph{cornsquare}, described in Definition 2.5 of \emph{loc. cit.})
\end{rem}

\begin{defn}[Padded ladder, ladder, cut-vertex]\label{defn:ladder}
A \emph{padded ladder} is a disc diagram $D\to X^*$ (or $\widetilde X^*$) with the following structure.  First, there is a 
sequence $C_1,\ldots,C_n$, where each $C_i$ is a cone-cell or vertex of $D$, so that $C_i,C_k$ lie in distinct components of 
$D-C_j$ whenever $i<j<k$.  The diagram $D$ is an alternating union of these vertices and cone-cells with a sequence of 
subdiagrams $R_0,\ldots,R_{n}$ called  \emph{pseudorectangles}, so that:
\begin{enumerate}
 \item The path $\boundary_pD$ is a concatenation $P_1P_2^{-1}$, where  $P_1,P_2$ start on $R_0$ and end on $R_{n}$.
 \item We have $P_1=\nu_0\rho_0\alpha_1\rho_1\cdots\alpha_n\rho_n$ and 
$P_2=\vartheta_0\gamma_1\vartheta_1\gamma_2\cdots\gamma_n\vartheta_n\mu_{n+1}$.
 \item We have $\boundary_pC_i=\mu_i\alpha_i\nu_i^{-1}\gamma_i^{-1}$.
 \item We have $\boundary_pR_i=\nu_i\rho_i\mu_{i+1}^{-1}\vartheta_i^{-1}$.
 \item Each $R_i$ is a \emph{square diagram}, i.e. contains no cone-cells.
 \item For each $i$, any dual curve in $R_i$ emanating from $\rho_i$ ends on $\vartheta_i$ and vice versa.  Hence, any dual 
curve emanating from $\nu_i$ ends on $\mu_{i+1}$ and vice versa.
 \item For each $i$, no two dual curves emanating from $\mu_i$ cross.
\end{enumerate}
See Figure~\ref{fig:better_ladder} for a picture illustrating the notation.  We say that $R_i$ is \emph{horizontally 
degenerate} if $|\mu_{i+1}|=|\nu_i|=0$ and \emph{vertically degenerate} if $|\rho_i|=|\vartheta_i|=0$.  When $C_i$ is a 
vertex, 
we call it a \emph{cut-vertex} of the padded ladder $D$.

If $R_0,R_{n}$ are vertically degenerate, then $D$ is a \emph{ladder}.  (A padded ladder is  a special case of what 
Jankiewicz calls a \emph{generalized ladder} in~\cite{Jankiewicz}; the definition of ladder  here is equivalent 
to that in~\cite[Definition 3.42]{Wise:QCH}). 
\end{defn}

We require the following three crucial facts, due to Wise.  These are tailored to our specific situation; 
the statements in~\cite{Wise:QCH} are more general.  

\begin{thm}[Ladder theorem]\label{thm:ladder_thm}
Let $\langle X\mid\{Y_i\}_{i\in\mathcal I}\rangle$ be a cubical $C'(\frac{1}{12})$  presentation.  Let $D\to X^*$ be a minimal disc diagram 
such that the corresponding rectified diagram has exactly two positively-curved cells along $\boundary_pD$.
Then $D$ is a ladder.
\end{thm}

\begin{proof}
This follows by combining~\cite[Theorem 3.43]{Wise:QCH} with~\cite[Theorem 3.31]{Wise:QCH}.  See also~\cite[Examples 
3.s.(3), Theorem 3.32]{Wise:QCH}.
\end{proof}

From Theorem 3.46 of~\cite{Wise:QCH}, we also get:

\begin{thm}[Greendlinger's lemma/diagram trichotomy]\label{thm:Greendlinger}
Under the hypotheses of Theorem~\ref{thm:ladder_thm}, if $D\to X^*$ is a minimal disc diagram, then either $D$ consists of a 
single vertex or cone-cell, or $D$ is a ladder, or $D$ contains at least three \emph{features of positive 
curvature} (i.e. shells, spurs, or generalized corners) along 
$\boundary_pD$.
\end{thm}

The next theorem follows directly from Lemma~3.70 of~\cite{Wise:QCH}.   In fact, it holds under weaker small-cancellation 
conditions (see~\cite[Lemma 3.70]{Wise:QCH} or~\cite{Jankiewicz}), but we will not require this.

\begin{thm}[Short inner paths]\label{thm:short_inner_paths}
Let $\langle X\mid\{Y_i\}_{i\in\mathcal I}\rangle$ be a cubical $C'(\frac{1}{14})$  presentation.  Let $D\to X^*$ be a disc diagram and let 
$C$ be a shell in $D$ with boundary path $QS$, with $Q$ a maximal common subpath of $\boundary_pC$ and $\boundary_pD$, and 
$S$ an internal path.  Suppose that $QS$ is essential in the relator $Y$ to which $C$ maps, and that $S$ is of minimal 
length among all paths $S'\to Y$ that are square-homotopic rel endpoints in $Y$ to $S$.  Finally, suppose that the total 
curvature 
contribution from $C$ is positive.  Then $|S|<|Q|$.
\end{thm}

\subsection{The classification theorem}\label{subsec:classification}
Using the above fundamental results, we produce a cubical small cancellation version of Strebel's classification of 
triangles in classical small cancellation groups~\cite{Strebel}.  An \emph{exposed square} in a disc diagram $D$ is a square 
with two consecutive edges on $\boundary_pD$.  A \emph{tripod} 
is a triangle diagram with no cone-cells or squares. We can now state our classification of triangles in cubical small cancellation groups:

\begin{thm}[Classification of triangles in cubical $C'(\frac{1}{144})$ groups]\label{thm:strebel_cubical_small_can}
 Let $X$ be a connected nonpositively curved cube complex, let $\mathcal I$ be a 
(possibly infinite) index set and let $\{Y_i\to X\}_{i\in\mathcal I}$ be a set of local isometries of connected complexes.

Let $\langle X\mid\{Y_i\}_{i\in\mathcal I}\rangle$ be a cubical presentation satisfying the $C'(\frac{1}{144})$ condition.  
Let $\alpha,\beta,\gamma\to\cay(X^*)$ be combinatorial geodesics so that $\alpha\beta\gamma$ is a geodesic triangle.  Then 
there exists a disc diagram $(D,\boundary_pD)\to(X^*,X)$ with boundary path $\alpha'\beta'\gamma'\to\cay(X^*)\to X^*$ lying 
in $X$, so that the following hold.  First, $\alpha\to X$ and $\alpha'\to X$ co-bound a bigon $B\to X$ (i.e. they are 
square-homotopic) and the same is true of $\beta,\beta'$ and $\gamma,\gamma'$.  Second, $D$ is of one of the following types.
\begin{enumerate}
 \item (\textbf{$3$--shell generic}:) $D$ has exactly three external cone-cells, $C_1,C_2,C_3$, respectively containing the 
points $\alpha'\cap\beta',\beta'\cap\gamma',\gamma'\cap\alpha'$.  There is exactly one cone-cell $M$ that intersects 
$\alpha',\beta'$, and $\gamma'$.  Moreover, $D$ is the union of three ladders, $L_1,L_2,L_3$ so that $L_i\cap L_j=M$ for all 
$i,j$.  
In particular, every cone-cell except $M$ intersects exactly two of the geodesics $\alpha,\beta,\gamma$.\label{item:3g}
 \item (\textbf{$3$--shell tripod}:) $D$ has exactly three external cone-cells, $C_1,C_2,C_3$, respectively containing the 
points $\alpha'\cap\beta',\beta'\cap\gamma',\gamma'\cap\alpha'$.  Every other cone-cell intersects exactly two of the 
geodesics 
$\alpha',\beta',\gamma'$. In this case, $D$ is the union of $3$ (possibly padded) ladders $L_1,L_2,L_3$ and a tripod 
triangle 
$P_1P_2P_3\to X$ so that $L_i$ intersects the other two ladders in the path $P_i$.\label{item:3t}
 \item (\textbf{$2$--shell generic}:)  Same as $3$--shell generic, except exactly one of $C_1,C_2,C_3$ is a spur or exposed 
square instead of a cone-cell.\label{item:2g}
 \item (\textbf{$2$--shell tripod}:) Same as $3$-shell tripod, except exactly one of $C_1,C_2,C_3$ is a spur or exposed 
square instead of a cone-cell.\label{item:2t}
 \item (\textbf{$1$--shell generic}:) Same as $2$--shell generic, except exactly two of $C_1,C_2,C_3$ are spurs or exposed 
squares.\label{item:1g}
 \item (\textbf{$1$--shell tripod}:) Same as $2$--shell tripod, except exactly two of $C_1,C_2,C_3$ are spurs or exposed 
squares.\label{item:1t}
 \item (\textbf{No-shell generic}:) Same as $3$--shell generic, except $C_1,C_2,C_3$ are all spurs or exposed 
squares.\label{item:0g}
 \item (\textbf{No-shell tripod}:) Same as $3$--shell tripod except $C_1,C_2,C_3$ are all spurs or exposed squares.  This 
includes the case where $\alpha'\beta'\gamma'$ is nullhomotopic in $X$, in which case $D$ is a tripod.\label{item:0t}
 \item (\textbf{Degenerate triangle}:) $D$ is a single vertex or cone-cell, or $D$ is a ladder.\label{item:ladder}
\end{enumerate}
The diagram $D\to X^*$ is a \emph{standard diagram} for the triangle $\alpha\beta\gamma$.  The eight non-degenerate cases 
are shown in Figure~\ref{fig:classification}.
\end{thm}

\begin{rem}[Media and small cancellation parameters]\label{rem:media_and_constants}
The standard diagram depends only on the endpoints of the geodesics $\alpha,\beta,\gamma$.  Just as it is usual in 
CAT(0) cube complexes to homotop geodesics, fixing their endpoints, in order to minimize the area of diagrams, here we are 
not married to particular geodesics, just to square-homotopy classes relative to their endpoints.  In particular, if $\alpha\beta\gamma$ 
bounds a disc diagram in $X$, then $D$ is a tripod.  When $\mathcal I=\emptyset$, 
Theorem~\ref{thm:strebel_cubical_small_can} just says: any three $0$--cubes in a CAT(0) cube complex determine a geodesic 
tripod.

More generally, as illustrated by Figure~\ref{fig:classification}, Theorem~\ref{thm:strebel_cubical_small_can} should be 
interpreted as saying that the vertices of the triangle have a ``median'' which is either a vertex or a cone-cell, and there 
is a geodesic triangle connecting the given three points, each of whose sides passes within a wall-piece of the ``median''.  
In other words, given $0$-cells $a,b,c\in\cay(X^*)$, the ``convex hulls'' of the three possible pairs mutually coarsely 
intersect.

At the other extreme, when $X$ is a wedge of circles and each $Y_i$ is an immersed circle, Theorem~\ref{thm:strebel_cubical_small_can} 
generalizes a weak version of Strebel's classification of triangles~\cite[Theorem~43]{Strebel}; specifically, 
Theorem~\ref{thm:strebel_cubical_small_can} provides the same classification as Strebel's result, but, because the proof 
must work in the more general context of cubical presentations, we require stronger metric small cancellation conditions 
than Strebel needs in the classical setting. 

\end{rem}

\begin{proof}[Proof of Theorem~\ref{thm:strebel_cubical_small_can}]
We consider a disc diagram bounded by the geodesic triangle.  The proof is then essentially a meticulous  application 
of Theorem~\ref{thm:Greendlinger}, Theorem~\ref{thm:ladder_thm}, and Theorem~\ref{thm:short_inner_paths}, following and 
followed by appropriately chosen square homotopies.
The main points are:
\begin{itemize}
     \item A curvature computation to eliminate the possibility of internal cone-cells (in the Strebel classification in the 
case of classical $C'(\frac{1}{6})$ condition, one of the primary features is that the disc diagrams do not have internal 
cells).  We also rule out some other types of cone-cells lying along the boundary of the diagram. This computation, which 
is a slightly modified version of a computation in the proof of Theorem~3.29 of~\cite{Wise:QCH} (same computation, 
except with different numbers reflecting our strong small-cancellation conditions), is why we need the $C'(\frac{1}{144})$ 
condition.

  \item Applications of Theorem~\ref{thm:short_inner_paths} to rule out positively-curved shells along the boundary of a 
disc diagram bounded by a geodesic triangle.

\item An application of Theorem~\ref{thm:ladder_thm} to decompose our diagram as the union of at most three ladders meeting 
along a central cell.
\end{itemize}

\textbf{Choosing $\alpha',\beta',\gamma'$ and constructing $D$:}  Given a geodesic $P\to\cay(X^*)$, let $[P]$ be the set of 
geodesics $Q$ that have the same endpoints as $P$ and the additional property that $PQ^{-1}$ bounds a disc diagram 
containing no cone-cells, i.e. there is a disc diagram $E\to X$ whose boundary path is $PQ^{-1}\to\cay(X^*)\to X^*$.

Choose a disc diagram $D\to\widetilde X^*$ so that $\boundary_pD=\alpha'\beta'\gamma'$, where 
$\alpha'\in[\alpha],\beta'\in[\beta],\gamma'\in[\gamma]$.  Choose $D$ so that the complexity is minimal among all disc 
diagrams with boundary path of the preceding form.

Abusing notation slightly, we now temporarily regard $D$ as the \emph{rectified diagram} from 
Remark~\ref{rem:rectification_and_angling}.  This means that certain square ladders are regarded as single (rectangular) 
$2$--cells, 
cone-cells are regarded as single 2--cells, and the remaining parts of the diagram are $2$--cells (formed by ignoring 
non-boundary $1$--cells in certain square subdiagrams) called \emph{shards}.  Angles are assigned to corners according to 
the \emph{split-angling} discussed above.
% 
% \textbf{Applying the ladder theorem:}  By Theorem~\ref{thm:ladder_thm}, either $D$ is a ladder, so 
% assertion~\eqref{item:ladder} holds, and we are done, or $D$ has at least $3$ ``features of positive curvature''   We assume the latter.

\textbf{Applying the Greendlinger lemma:}  By Theorem~\ref{thm:Greendlinger}, either $D$ is a single vertex, a single cone-cell, a ladder, or $D$ has at least $3$ features of positive curvature -- 
spurs, 
generalised corners, or shells -- along $\boundary_pD$.  In either of the first three cases, assertion~\eqref{item:ladder} holds.  So, assume that $D$ has at least $3$ features of positive 
curvature along the boundary, \emph{each of which} is a shell, a spur, or a generalized corner.

We may assume that all generalised corners along $\boundary_pD$ are actually squares with corners on $\boundary_pD$.  
Indeed, let $s$ be a square in $D$ with a generalised corner on $\boundary_pD$, so the dual curves $K_1,K_2$ 
intersecting $s$ end at consecutive $1$--cubes $e_1,e_2$ on $\boundary_pD$.  By shuffling, we modify 
$D$ -- without changing $\boundary_pD$ or increasing complexity -- so that $s$ lies along the boundary, i.e. $e_1,e_2$ 
are consecutive $1$--cubes of $s$.

\textbf{Square homotopies:}  Let $s$ be a square of $D$ so that $\boundary_ps$ 
and $\boundary_pD$ have a common subpath $e_1e_2$.  For $i\in\{1,2\}$, let $e'_i$ be the $1$--cube of $s$ opposite $e_i$.  
If $e_1e_2$ is a subpath of one of the three constituent geodesics of $\boundary_pD$ (say, $\alpha'$), then we can modify 
$\alpha'$ in its square-homotopy class by replacing $e_1e_2$ by $e_2'e_1'$, resulting in a new diagram with the same number 
of cone-cells and fewer squares.  This contradicts our minimality assumption. 

Hence we can assume the following: any square $s$ with a corner on the boundary 
lies at the transition from $\alpha'$ to $\beta'$, or $\beta'$ to $\gamma'$, or $\gamma'$ to $\alpha'$.  
Also, since $\alpha', \beta', \gamma'$ are geodesic, a spur of the form $ee^{-1}$ cannot occur along any of $\alpha', 
\beta',\gamma'$ so the only spurs consist of overlaps between $\alpha',\beta'$ or $\beta',\gamma'$ or $\gamma',\alpha'$.

We now rule out cone-cells in $D$ of a few types. First, we rule out positively curved shells with outer path on one of our 
three geodesics.  This puts us in a position where we have at most $3$ features of positive curvature (we already removed 
generalised corners along our geodesics, we will shortly remove shells, and spurs are impossible, so positive curvature can 
only occur at transitions between successive geodesics).  Using this, we rule out internal cone-cells and (nonpositively 
curved) shells.  From this, we deduce that each cone-cell has nonempty connected intersection with each of 
$\alpha',\beta',\gamma'$.  This is summarised in Figure~\ref{fig:illegal_cells}.

\begin{figure}[h]
     \begin{overpic}[width=0.35\textwidth]{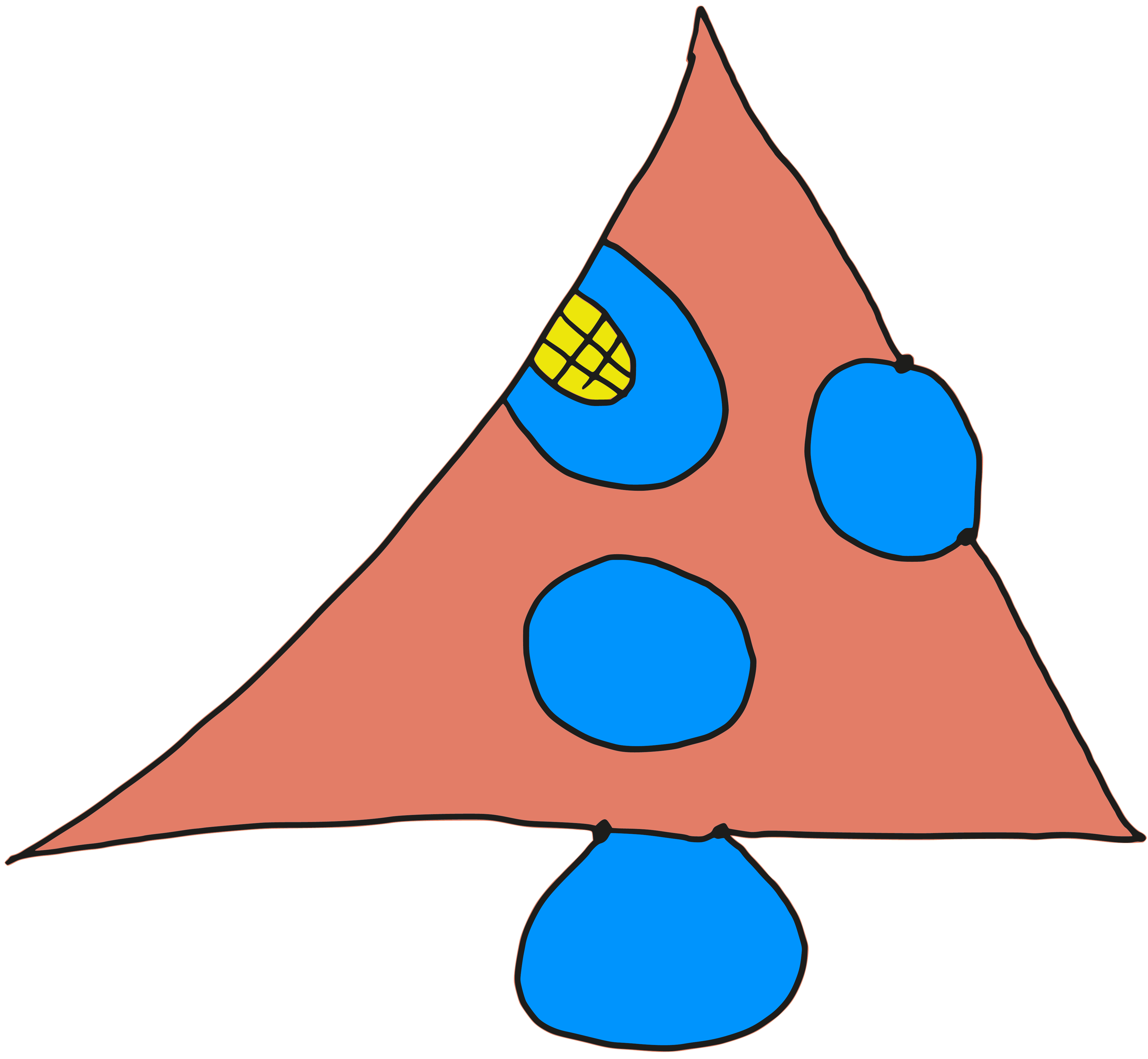}
          \put(53,8){$C_1$}
          \put(53,33){$C_2$}
          \put(73,53){$C_3$}
          \put(52,52){$C_4$}
     \end{overpic}
\caption{The positively-curved shell $C_1$ is ruled out using Theorem~\ref{thm:short_inner_paths}.  
The internal cone-cell $C_2$ and the shell $C_3$ are then ruled out using our small-cancellation 
conditions and Theorem~\ref{thm:CBGT}.  A cone-cell like $C_4$, whose intersection with $\alpha'$ 
is nonempty and disconnected, can then be ruled out: an innermost such $C_4$ would have to bound a 
square subdiagram with the subtended part of $\alpha'$, because otherwise we would get one of the 
other $3$ types of illegal cone-cell.  The square subdiagram could then be absorbed into $C_4$, 
lowering the complexity.}\label{fig:illegal_cells}
\end{figure}

\begin{claim}[No positive curvature along geodesics (after square 
homotopy)]\label{claim:no_positive_curvature_along_boundary}
Let $C$ be a positively-curved shell in $D$ with outer path $O$ and inner path $I$.  Then $O$ cannot be a subpath of 
$\alpha',\beta',$ or $\gamma'$.     Thus, $O=AB$, where $A$ is a nontrivial terminal subpath of $\alpha',\beta',$ or 
$\gamma'$ and $B$ is a nontrivial initial 
subpath of $\beta',\gamma',$ or $\alpha'$, respectively.
\end{claim}
\renewcommand{\qedsymbol}{$\blacksquare$}
\begin{proof}[Proof of Claim~\ref{claim:no_positive_curvature_along_boundary}]
Suppose $O$ is a subpath of $\alpha'$.  (The other cases are handled identically.)  Theorem~\ref{thm:short_inner_paths} 
shows that we can replace $O$ by a geodesic in $Y$ joining the 
endpoints of $I$ to get a strictly shorter path joining the 
endpoints of $\alpha'$, a contradiction.
\end{proof}

Claim~\ref{claim:no_positive_curvature_along_boundary}, and the discussion preceding it, put us in the following situation: $D$ has exactly \textbf{three features of 
positive curvature along the boundary}, which are subdiagrams $C_1,C_2,C_3$.  For each $i$, $\boundary_pC_i=OI$, where $O$ is 
a subpath of $\boundary_pD$ and $I$ is an internal path, and $O$ has at least one $1$-cube on each of two distinct subpaths among 
$\alpha',\beta',\gamma'$ of the boundary.

Let $C$ be a cone-cell of $D$.  Recall that $C$ is \emph{internal} if its boundary path 
intersects $\boundary_pD$ in a set containing no $1$--cube.  To rule out internal cone-cells, we need an auxiliary claim.

\begin{claim}[Internal cone-cell curvature contribution]\label{claim:internal_cone_cell_curvature_contribution}
     Let $C$ be an internal cone-cell of $D$.  Then the curvature at $C$ is strictly less than $-4\pi$.
\end{claim}

\begin{proof}[Proof of Claim~\ref{claim:internal_cone_cell_curvature_contribution}]
We argue almost exactly as in the proof of~\cite[Theorem 3.32]{Wise:QCH}, except exploiting our stronger small-cancellation 
assumption.  Since $D$  has minimal complexity, it is \emph{reduced} (in the sense of \cite[Definition 3.11]{Wise:QCH}), then $\boundary_pD$ is not the concatenation of fewer than $145$ pieces.  Now apply 
exactly the proof of~\cite[Theorem 3.31]{Wise:QCH}, except that the number $13$ in that proof is replaced by $145$, yielding 
a curvature contribution from $C$ of at most $2\pi-\frac{145\pi}{6}<-4\pi$.
\end{proof}

\begin{claim}[No internal cone-cells]\label{claim:no_internal_cone_cells}
The diagram $D$ does not contain an internal cone-cell.     
\end{claim}

\begin{proof}[Proof of Claim~\ref{claim:no_internal_cone_cells}]
Suppose that there are $n\geqslant0$ internal cone-cells.  

Let $v$ be a $0$--cube of (the rectified) $D$.  Then, by Theorem 3.32 of~\cite{Wise:QCH}, the curvature contribution 
from $v$ is:
\begin{enumerate}
 \item at most $0$ if $v$ is internal or not contained in a $2$--cell and not a spur;
 \item exactly $\pi$ if $v$ is a spur;
 \item exactly $\frac{\pi}{2}$ if $v$ is the corner of a square along $\boundary_pD$.
\end{enumerate}

Let $f$ be a $2$--cell of $D$ (a cone-cell, rectangle, or a \emph{shard} of the corresponding \emph{rectified 
diagram}~\cite{Wise:QCH}).  Then the curvature contribution is:
\begin{enumerate}
 \item at most $0$ if $f$ is a rectangle or shard;
 \item less than $-4\pi$ if $f$ is an internal cone-cell, by Claim~\ref{claim:internal_cone_cell_curvature_contribution};
 \item at most $2\pi$ if $f$ is a shell.
\end{enumerate}

Hence, our three features of positive curvature contribute a total of at most $6\pi$ of curvature, while the sum of the 
remaining curvatures is $<-4n\pi$.  This contradicts Theorem~\ref{thm:CBGT} unless $n=0$.  Thus, there are no internal 
cone-cells.
\end{proof}

We have to rule out another type of cone-cell.  A cone-cell $C$ in $D$ is \emph{shortly external} if its boundary path has 
the form $OI$, where $I$ is internal and $O$ is a subpath of $\alpha'$, $\beta'$, or $\gamma'$.  

\begin{claim}\label{claim:no_shortly_external}
     $D$ has no shortly-external cone-cell.
\end{claim}

\begin{proof}[Proof of Claim~\ref{claim:no_shortly_external}]
Note that $|O|\leqslant|I|$ since $\alpha',\beta',\gamma'$ are geodesics.  Following the proof of Theorem~3.29 
of~\cite{Wise:QCH}, we write $I$ as a concatenation of $k$ paths, each of which is a concatenation of at most $3$ pieces, 
such that each path contributes an angle defect of at least $\pi/4$.  Our small-cancellation condition guarantees that there 
are more than $144/(2\cdot 3)=24$ such paths, so the total curvature is at most $2\pi-k\pi/4<-4\pi$.  Exactly as in the 
proof 
of Claim~\ref{claim:no_internal_cone_cells}, we obtain a contradiction with Theorem~\ref{thm:CBGT} unless there are no 
shortly-external cone-cells.
\end{proof}

At this point, we have completed the curvature computations in the proof, and now regard $D$ as an ordinary (not 
rectified) diagram.

\textbf{Analysis of the cone-cells:}  Let $C$ be a cone-cell in $D$.  We would like to show that $C$ has connected 
intersection with each of $\alpha',\beta',\gamma'$.  

Suppose that for some 
$\delta\in\{\alpha',\beta',\gamma'\}$, there is a subpath $\delta'=PQR$ of $\delta$, where:
\begin{itemize}
     \item the paths $P,R$ are subpaths of the boundary path of $C$,
\item the terminal vertex of $P$ and initial vertex of $R$ subtend a subpath $Q'$ of $\boundary_pC$, such that
\item the path $QQ'$ bounds a subdiagram $E$ of $D$ between $C$ and $\boundary_pD$.
\item $Q$ and $\boundary_pC$ have no common $1$--cell.
\end{itemize}

 If $E$ contains no cone-cell, then $E$ is a square diagram between the relator $Y_i$ to which $C$ maps and the geodesic 
$Q$, so by local convexity of $Y_i$ in $X$, we have that $E\to X$ factors through $Y_i\to X$.  Hence $E$ could have been 
absorbed into the cone-cell $C$, whence minimality of the complexity of $D$ ensures that $E$ is trivial, i.e. $Q=(Q')^{-1}$. 
 This is a contradiction, so $E$ is not a square diagram.

We can assume $C$ is innermost, in the sense that any cone-cell $C_0$ in $E$ has connected intersection (possibly empty) with 
$Q$.

Such a cone-cell $C_0$ is not internal in $E$, for then it would be internal in $D$, violating 
Claim~\ref{claim:no_internal_cone_cells}.  So $C_0$ is shortly-external in $D$, or a positively-curved shell along $Q$, 
violating Claim~\ref{claim:no_shortly_external} or Claim~\ref{claim:no_positive_curvature_along_boundary}.

So, we have proven:

\begin{claim}\label{claim:connected}
For each cone-cell $C$ of $D$, the path $\boundary_pC$ has connected intersection with each of 
$\alpha',\beta',$ and $\gamma'$.  Moreover, $\boundary_pC$ intersects at least two of the paths $\alpha',\beta',\gamma'$.   
\end{claim}

\textbf{The current situation:}  Thus far, we have reduced to the following situation:
\begin{itemize}
     \item $D$ is not a ladder.
     \item $D$ has precisely $3$ features of positive curvature along its boundary path, which are subdiagrams $C_1,C_2,C_3$.
     \item The subdiagram $C_1$ has boundary path $ABI$, where $A$ is a nontrivial terminal subpath of $\alpha'$, $B$ is a 
nontrivial initial subpath of $\beta'$, and $I$ is a (possibly trivial) path.
\item Either $C_1$ is a single cone-cell (a shell) or $C_1$ is a spur, $|I|=0$, and $A$ is an edge and $B=A^{-1}$, or $C_1$ 
is a square, $|A|,|B|\geqslant1$, and $|I|\leqslant2$.  The same description holds for $C_2$ (with $\beta',\gamma'$ replacing 
$\alpha',\beta'$) and $C_3$ (with $\gamma',\alpha'$ replacing $\alpha',\beta'$).
\item Every cone-cell $C$ of $D$ not in $\{C_1,C_2,C_3\}$ has connected intersection with each of 
$\alpha',\beta',\gamma'$ and intersects at least $2$ of these paths.  We call $C$ a \emph{median-cell} if $C$ intersects all 
three of these paths, and a \emph{tail-cell} otherwise.
\end{itemize}

  We emphasise that if $C$ is a median cell, then it has nonempty, \emph{connected} 
intersection with each of $\alpha',\beta',\gamma'$.  Hence, if $C$ is a median-cell, then $C$ separates $D$ into three 
complementary components, each disjoint from one of the paths $\alpha',\beta',\gamma'$.  Hence, $C$ is the unique cone-cell 
of 
$D$ intersecting each of $\alpha',\beta',\gamma'$.  (We note that there may be other disc diagrams with the same boundary 
path, 
containing a different median-cell.)

We now divide into cases.  First, if $D$ contains a median-cell, then we are in one of the \emph{generic} cases, i.e. we 
will show that one of~\eqref{item:0g},\eqref{item:1g},\eqref{item:2g},\eqref{item:3g} holds, according to how many of 
$\{C_1,C_2,C_3\}$ are spurs or shells.  Otherwise, we will show that one 
of~\eqref{item:0t},\eqref{item:1t},\eqref{item:2t},\eqref{item:3t} holds. 

\textbf{The generic cases:}  Suppose that $D$ has a (unique) median-cell $M$ and let $i\in\{1,2,3\}$.  Let 
$\delta,\delta'\in\{\alpha',\beta',\gamma'\}$ be the parts of the boundary path of $D$ that intersect $C_i$.  Let 
$\boundary_pM=AP_1BP_2CP_3$, where $A,B,C$ are respectively subpaths of $\alpha',\beta',\gamma'$ and $P_1,P_2,P_3$ are 
internal 
paths.  Write $\alpha'=\bar\alpha' A\check\alpha',\beta'=\bar\beta' B\check\beta',\gamma'=\bar\gamma' C\check\gamma'$.  
Consider the 
subdiagram $L_1$ bounded by $A\check\alpha'\bar\beta' BP_2CP_3$.  The ladder theorem, Theorem~\ref{thm:ladder_thm}, and our 
above analysis of the possible features of positive curvature in (the rectification of) $D$ shows that $L_1$ is a ladder.  
The ladders $L_2,L_3$ are constructed analogously. 

\textbf{The tripod cases:}  Suppose there is no median-cell.  Then we have a subdiagram $T$ of $D$ with boundary path 
$AP_1BP_2CP_3$, where $A$ is a subpath of $\alpha'$, $B$ a subpath of $\beta'$, $C$ a subpath of $\gamma'$, and 
$P_1,P_2,P_3$ 
internal subpaths that lie on innermost cone-cells in $D$ or, if they do not exist, spurs or exposed squares in 
$\{C_1,C_2,C_3\}$.  By construction, $T$ is a possibly degenerate square diagram, and by convexity of relators and 
minimality, for each path $Q\in\{A,B,C,P_1,P_2,P_3\}$, no two dual curves in $T$ emanating from $Q$ can cross.  Moreover, no 
dual curve travels from $Q$ to $Q$ or to the next named subpath, for otherwise we could reduce complexity.  Some 
possibilities are shown in Figure~\ref{fig:tripod_cases}.  

\begin{figure}[h]
\begin{overpic}[width=\textwidth]{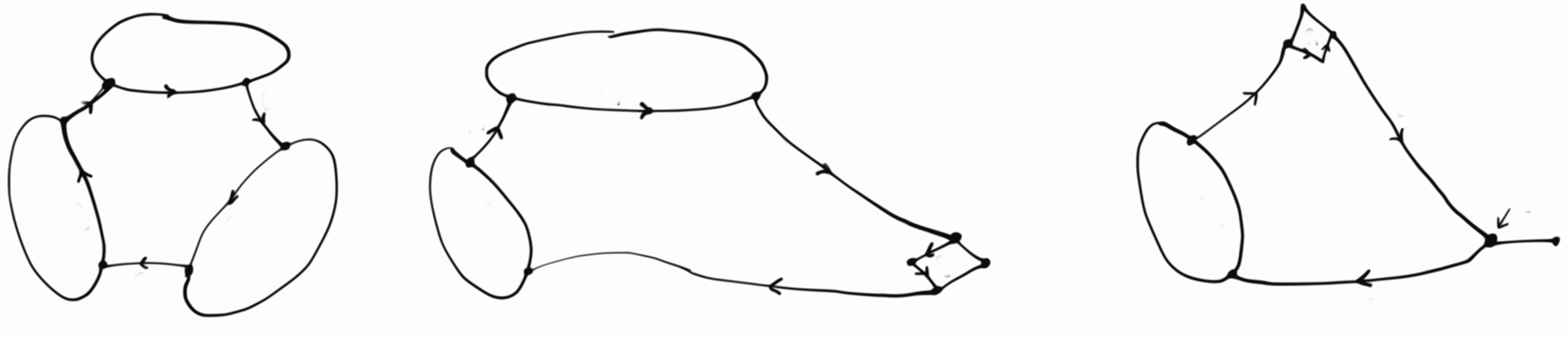}
     \put(9,2){$C$}
     \put(3,7){$P_3$}
     \put(6,13){$A$}
     \put(10,17){$P_1$}
     \put(15,7){$P_2$}
     \put(17,14){$B$}
     \put(55,4){$P_2$}
     \put(83,15){$P_1$}
     \put(96,9){$P_2$}
\end{overpic}

\caption{Some possibilities for the internal square subdiagram in the tripod cases.}\label{fig:tripod_cases}
\end{figure}

For convenience, we lift $T$ to a diagram $T\to\widetilde X$ (the CAT(0) cube complex $\widetilde X$, not the generalized 
Cayley graph).  Here, an analysis of the dual curves shows that $T$ decomposes as required; the analysis is indicated in 
Figure~\ref{fig:final_tripod}.
\begin{figure}[h]
 \begin{overpic}[width=0.8\textwidth]{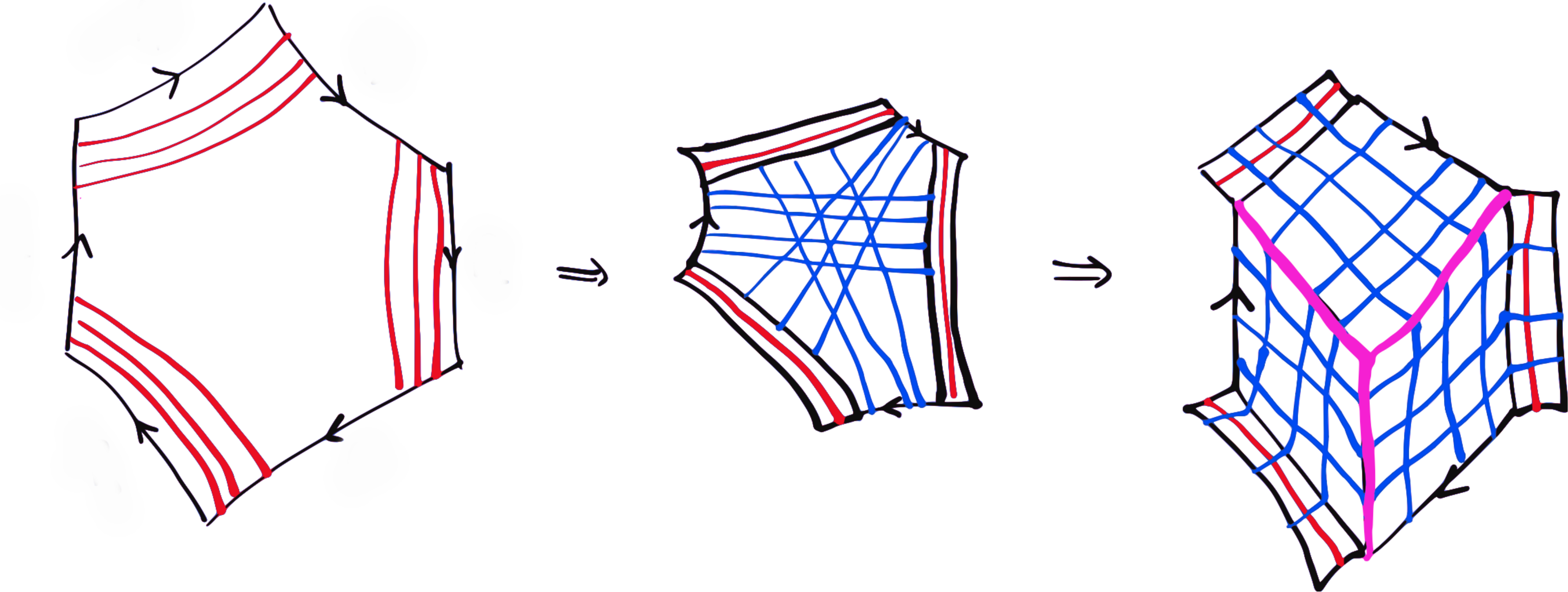}
      \put(2,22){$A$}
      \put(10,35){$P_1$}
      \put(24,31){$B$}
      \put(30,20){$P_2$}
      \put(24,9){$C$}
      \put(5,9){$P_3$}
 \end{overpic}

 \caption{The final square diagram analysis in the tripod cases.}\label{fig:final_tripod}
\end{figure}
First, consider dual curves in $T$ traveling from $A$ to $B$, $B$ to $C$, or $C$ to $A$.  Taking the union of all carriers 
of such dual curves yields rectangles attached to $P_1,P_2,P_3$.  Now consider the subdiagram that remains.  It is a hexagon 
bounded by subpaths of $A,B,C$ and parts of carriers of dual curves.  Dual curves in the subdiagram must travel from a 
subpath of $A,B,$ or $C$ to the antipodal dual-curve carrier.  Dual curves emanating from the same ``syllable'' of the 
boundary path do not cross, and we conclude, as at right in Figure~\ref{fig:final_tripod}, that this subdiagram is a 
``corner of a subdivided cube''.  It is now easy to deduce the padded ladder decomposition of $D$, with an application of 
Theorem~\ref{thm:ladder_thm}.  (Various parts of the 
picture may be degenerate, as suggested in Figure~\ref{fig:tripod_cases}.)
\renewcommand{\qedsymbol}{$\Box$}
\end{proof}

\section{Detecting WPD elements using the classification of triangles}\label{sec:metalemma}
In this section, we adopt the following assumptions and conventions:
\begin{enumerate}
 \item $\langle X\mid\{Y_i\}_{i\in\mathcal I}\rangle$ is a cubical presentation satisfying the $C'(\frac{1}{144})$ 
condition,   $X^*$ is the presentation complex, and $\widetilde X^*$ is the universal cover.  We assume that $X$ is locally 
finite, but we do \textbf{not} assume $X$ is \textbf{uniformly} locally finite.  

\item Since $\pi_1X$ is torsion-free, every $g\in\pi_1X-\{1\}$ has a combinatorial geodesic axis in the first cubical subdivision of $\widetilde X$ (see~\cite{Haglund:semisimple} for the notion of 
cubical subdivision and a proof of the statement about axes).  Since we can replace $X$ and each $Y_i$ by their first cubical subdivisions without changing the group of the cubical presentation or 
the small-cancellation conditions, \textbf{we assume in the rest of the paper} that each $g\in\pi_1X-\{1\}$ admits a combinatorial geodesic axis.

 \item Denote by $\dist$ the graph metric on $\cay(X^*)^{(1)}$, and by $\dist_{\widetilde X}$ the graph metric on 
$\widetilde X^{(1)}$.

\item Let there be a $\delta$--hyperbolic graph $\hyp$ and a coarsely surjective map $\Pi\colon\cay(X^*)\to \hyp$ so that:
\begin{enumerate}
%  \item \label{item:coarse_points} $\diam(\Pi(x))\leqslant \delta$ for all 
% $x\in\cay(X^*)$.
 \item \label{item:coarsely_lipschitz} We have $\dist_{\hyp}(\Pi(x),\Pi(y))\leqslant\dist(x,y)$ whenever 
$x,y\in\cay(X^*)^{(0)}$.
 \item \label{item:cone_projection} If $Y_i\subseteq\cay(X^*)$ is any relator, then $\diam(\Pi(Y_i))\leqslant \delta$.
 \item \label{item:equivariant} The group $\pi_1X^*$ acts by isometries on $\hyp$ in such a way that $\Pi$ is 
$\pi_1X^*$--equivariant.
 \item \label{item:median_cell} Let $H$ be a hyperplane in $\cay(X^*)$.  Then
$\diam(\Pi(\neb(H)))\leqslant \delta$.

\end{enumerate}
\end{enumerate}

Under these conditions, we will prove a lemma --- Lemma~\ref{lem:acyl_version_2} --- showing that $\pi_1X^*$ contains a WPD 
isometry of $\hyp$ provided it contains a \emph{fast} loxodromic one (defined below).  Later, we choose 
specific $\hyp$ and $\Pi$.  The proof of Lemma~\ref{lem:acyl_version_2} will require us to show that certain paths produced 
by an application of Theorem~\ref{thm:strebel_cubical_small_can} fellow-travel.  We need some preliminary lemmas:

\begin{lem}[Ladders are thin between cone-cells]\label{lem:thin_ladders}
Let $L\to\widetilde X^*$ be a minimal complexity padded ladder with boundary path $\alpha\beta^{-1}\gamma$, where 
$\alpha,\beta\colon [0,\ell]\to\cay(X^*)$ are 
geodesics with $\alpha(\ell)=\beta(\ell)$ and $\gamma$ is a piece.  Let $\Delta$ be the maximum length of a subpath of 
$\alpha$ or 
$\beta$ that 
lies on a single cone-cell of $L$. Then there exists $\kappa_0=\kappa_0(\Delta)$ so that for 
all $t\leqslant \ell$, $\dist(\alpha(t),\beta(t))\leqslant\kappa_0$.
\end{lem}

\begin{proof}
Write $\alpha=\alpha_0\eta_1\alpha_1\cdots\eta_n\alpha_n$ and $\beta=\beta_0\eta'_1\cdots\eta'_n\beta_n$, where each 
$\alpha_i,\beta_i$ lies on the top or 
bottom boundary path of one of the constituent pseudorectangles of $L$ and each $\eta_i,\eta'_i$ lies on the boundary path 
of 
a cone-cell or cut-vertex (and hence has length at most $\Delta$).  The boundary path of the $i^{th}$ pseudorectangle has 
the form $\alpha_ip_i\beta_i^{-1}q_i$, as in Figure~\ref{fig:thin_ladder_setup}.  When $i=0$, the path $q_0=\gamma$, which 
is a piece.  

\begin{figure}[h]
\begin{overpic}[width=0.75\textwidth]{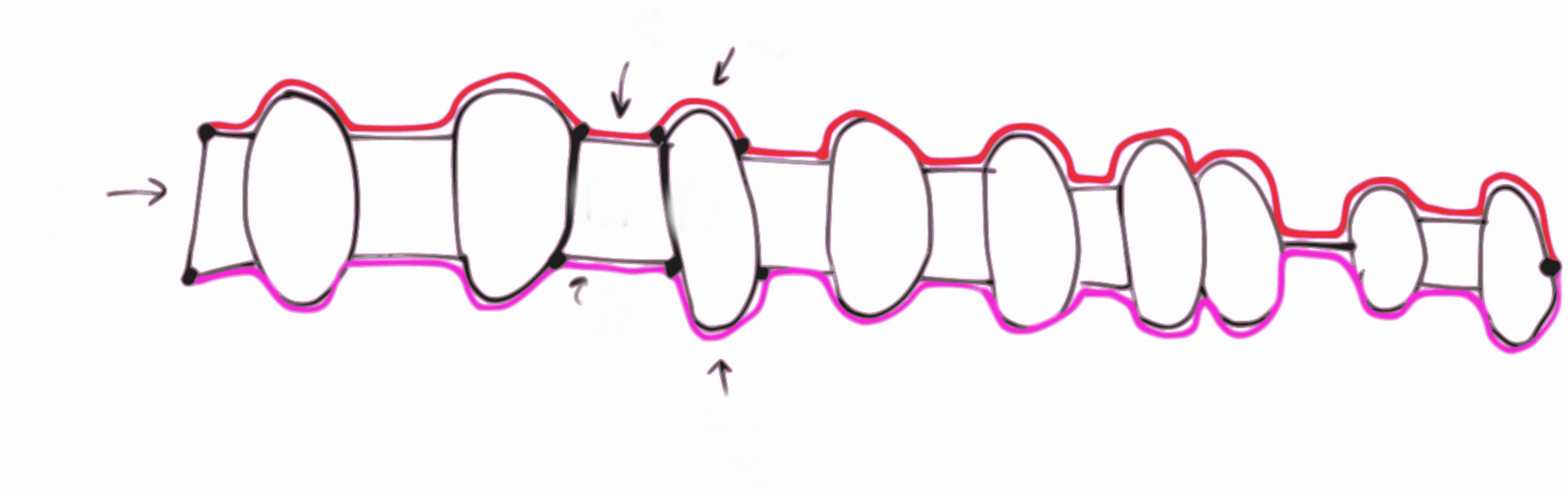}
     \put(5,19){$\gamma$}
     \put(37,10){$\beta_i$}
     \put(37,19){$q_i$}
     \put(40,29){$\alpha_i$}
     \put(47,29){$\eta_i$}
     \put(43,18){$p_i$}
     \put(46,4){$\eta_i'$}
\end{overpic}

\caption{Ladders are thin relative to cone-cells.}\label{fig:thin_ladder_setup}
\end{figure}

We observe that each $p_i,q_i$ is actually a piece.  Indeed, let $R_i$ be the pseudorectangle containing $p_i$, and 
let $C_i,C_{i+1}$ be the cone-cells adjacent to $R_i$.  Let $Y_i,Y_{i+1}$ be the relators in $\widetilde X^*$ to which 
$C_i,C_{i+1}$ map.  Since it is a square diagram, $R_i$ lifts to a disc diagram $R_i\to\widetilde X$ and the paths 
$q_i,p_i$ lie on elevations $\widetilde Y_i,\widetilde Y_{i+1}$ of $Y_i,Y_{i+1}$.  If $\widetilde Y_i=\widetilde 
Y_{i+1}$, then convexity of $\widetilde Y_i$ and the fact that $\alpha_i,\beta_i$ are geodesics implies that $R_i$ factors 
through $\widetilde Y_i$.  So, we could replace $C_i\cup R_i\cup C_{i+1}$ in $L$ by a single cone-cell, reducing complexity. 
 Hence $\widetilde Y_i,\widetilde Y_{i+1}$ are different, so $p_i,q_i$ are cone-pieces.

Hence, by the small cancellation conditions, there exists $M=M(\Delta)$ so that 
$|p_i|\leqslant M$ for all $i$.  Indeed, the boundary path of the cone-cell 
containing $p_i$ consists of one or two pieces, together with two subpaths of 
length at most $\Delta$.  Letting $\tau$ be the length of this boundary path, the 
small-cancellation conditions imply that $\tau<288\Delta/142$, and another 
application of the small-cancellation conditions gives $|p_i|<\Delta/71$.   

Since $\alpha,\beta$ 
are geodesics, 
we have $||\eta_i|-|\eta'_i||\leqslant 2M$ for all $i$, for otherwise we could construct shortcuts.  The lemma now follows 
easily.
\end{proof}

\begin{rem}[Rank one elements]\label{rem:rank_one_hull}
As usual (see e.g.~\cite{CapraceSageev:rank_rigidity}), $\tilde g\in\pi_1X$ is \emph{rank one} if it is hyperbolic on 
$\widetilde X$ and none of its axes lies in an isometrically embedded Euclidean half-plane.  Let $\alpha$ be a combinatorial 
geodesic axis in $\widetilde X$ for $\tilde g$.  Let $\mathcal W(\alpha)$ be the set of 
hyperplanes intersecting $\alpha$. Let 
$\mathcal C\alpha$ be the graph with vertex set $\mathcal W(\alpha)$, with vertices $H,V$ adjacent if the corresponding 
hyperplanes have intersecting carriers.  

Let $\widetilde B$ be the cubical convex hull of $\alpha$, which is a CAT(0) cube complex whose hyperplanes are 
exactly those in $\mathcal W(\alpha)$.  The graph $\mathcal 
C\alpha$ is exactly the \emph{contact graph} of $\widetilde B$, i.e. the intersection graph of its set of hyperplane 
carriers.  Considering the action of $\langle  \tilde g\rangle$ on $\widetilde B$, we 
see that there are finitely many $\langle  \tilde g\rangle$ orbits of hyperplanes in $\widetilde B$, and each has uniformly bounded 
coarse intersection with $\alpha$.  Hence, by~\cite[Theorem 2.4, Proposition 2.5]{Hagen:boundary}, $\langle  \tilde g\rangle$ has 
unbounded orbits in $\mathcal C\alpha$.  Therefore, since $\langle  \tilde g\rangle$ acts on $\mathcal C\alpha$ with finitely many 
orbits of vertices (each hyperplane of $\widetilde B$ is dual to one of $\langle  \tilde g\rangle$--finitely many $1$--cubes in 
$\alpha$), there exists $N$ such that if $H,V$ are hyperplanes intersecting 
$\alpha$ in $1$--cubes lying at distance more than $N$, then $H$ and $V$ 
cannot cross.
\end{rem}

\begin{lem}\label{lem:thin_bigons}
Let $\tilde g\in\pi_1X$ act hyperbolically on $\widetilde X$, and suppose that $\tilde g$ is rank one.  Then for each 
$\tilde x\in\widetilde X^{(0)}$, there exists $\kappa_1$ so that the following holds: if $n\geqslant0$ and 
$P,Q\colon [0,d]\to\widetilde X$ are combinatorial geodesics joining $\tilde x, \tilde g^n\tilde x$, then $\dist_{\widetilde 
X}(P(t),Q(t))\leqslant\kappa_1$ for $0\leqslant t\leqslant d$.
\end{lem}

\begin{proof}
Let $\alpha\to\widetilde X$ be a combinatorial geodesic axis for $\tilde g$ and let $\tilde a\in\alpha$ be a $0$--cube.  
Given $n\geqslant0$, let $P,Q\colon [0,d_n]\to\widetilde X^{(1)}$ (where $d_n=\dist_{\widetilde X}(\tilde x, \tilde g^n\tilde x)$) be 
combinatorial geodesics joining $\tilde x, \tilde g^n\tilde x$ and let $D\to\widetilde X$ be a minimal-area disc diagram with 
boundary path $PQ^{-1}$.  Note that $D\to\widetilde X$ is actually an isometric embedding on the $1$--skeleton.  Indeed, 
every dual curve in $D$ travels from $P$ to $Q$ since $P,Q$ are geodesics.  Hence each dual curve maps to a distinct 
hyperplane, so that for any vertices $v,v'\in D$, the number of dual curves of $D$ separating $v,v'$ is equal to the number 
of hyperplanes in $\widetilde X$ separating their images.

Fix $t\in\{0,1,\ldots,d_n\}$.  The above discussion shows that $\dist_{\widetilde X}(P(t),Q(t))$ is bounded by the number of 
dual curves in $D$ that travel from $P([0,t))$ to $Q((t,d_n])$, plus the number of dual curves from $Q([0,t))$ to 
$P((t,d_n])$.  Each dual curve of the former type crosses each dual curve of the 
latter type.  

Let $\mathfrak H$ be the set of dual curves of the former type, and let 
$\mathfrak V$ be the set of dual curves of the latter type.  Let 
$N_1=|\mathfrak H|$ and $N_2=|\mathfrak V|$.  Note that since  all but at most 
$2\dist_{\widetilde X}(\tilde a,\tilde x)$ hyperplanes that cross $P$ 
cross $\alpha$, at least $N_1+N_2-2\dist_{\widetilde X}(\tilde a,\tilde x)$ 
hyperplanes in $\mathfrak V\cup\mathfrak H$ cross $\alpha$.  

Let $\mathfrak 
H'\subseteq\mathfrak H,\mathfrak V'\subseteq\mathfrak V$ be the subsets consisting 
of hyperplanes/dual curves that cross $\alpha$.  Then by 
Remark~\ref{rem:rank_one_hull}, we have $\min\{N_1-2\dist_{\widetilde X}(\tilde 
a,\tilde x),N_2-2\dist_{\widetilde X}(\tilde a,\tilde x)\}\leqslant N$, where $N$ 
depends only on $\tilde g$.   But $|\mathfrak H|=|\mathfrak V|$, since for each 
dual curve crossing $Q([0,t])$ and $P([t,d_n])$, there must be a dual curve 
crossing $P([0,t])$ and $Q([t,d_n])$.  Hence $\mathfrak H\cup\mathfrak V$ has 
bounded cardinality, and we conclude that $\dist_{\widetilde X}(P(t),Q(t))$ is 
bounded by some $\kappa_1$ depending only on $\tilde g$ and $\tilde x$.
\end{proof}

\begin{defn}[$\Delta$--fast]\label{defn:fast_loxodromic}
Let $g\in\pi_1X^*$ act on $\hyp$ as a loxodromic isometry and let $\Delta\geqslant 0$.  Then $g$ is $\Delta$--\emph{fast} if the 
following holds.  Let $\widetilde A$ be a combinatorial geodesic axis in $\widetilde X$ for some $\tilde g\in\pi_1X$ mapping 
to $g$.  Let $A$ be the image of $\widetilde A$ in $\cay(X^*)$ and let $x\in A$ be a $0$--cube.  Let $R\geqslant 0$ and let 
$\alpha$ be a geodesic in $\cay(X^*)$ from $x$ to $g^Rx$.  Then any subpath of $\alpha$ lying in a hyperplane carrier or 
relator has length at most $\Delta$.
\end{defn}

We are now ready for the main lemma:

\begin{lem}[Fast loxodromic implies WPD]\label{lem:acyl_version_2}
Suppose $g\in\pi_1X^*$ acts loxodromically on $\hyp$ and that $g$ is $\Delta$--fast for some $\Delta$.  Then for all 
$\epsilon>0,\bar x\in\hyp$, there exists $R\in\mathbb N$ 
so 
that $$|\{h\in G\mid\dist_\hyp(h\bar x,\bar x)\leqslant\epsilon,\dist_\hyp(hg^R\bar 
x,g^R\bar x)\leqslant\epsilon\}|<\infty,$$ i.e. $g$ is a WPD 
element. 
\end{lem}

\begin{proof}
Fix $\epsilon>0$ and let $\bar x\in\hyp$; since $\Pi$ is coarsely surjective, we can assume $\bar x=\Pi(x)$ for some 
vertex $x$ of $\cay(X^*)$.  We will actually show that the claim holds provided $R$ is chosen sufficiently 
large (in terms of $g,x,\epsilon$).

Let $\tau\geqslant1$ be the 
translation length of $g$ on the graph $\hyp$.  

It suffices to prove the claim for a specific 
$x$, so we can assume that $x$ lies on the image in $\cay(X^*)$ of the combinatorial geodesic axis of some lift of $g$ to 
$\pi_1X$.  Hence, since $g$ is $\Delta$--fast, for any $R$ and any geodesic $\alpha$ from $x$ to $g^Rx$, any 
subpath of $\alpha$ lying in a hyperplane carrier or a relator has length at most $\Delta$.  

Fix an integer $R$ satisfying $R\geqslant10^9(\epsilon+2\delta+\Delta+\kappa_1')/\tau$, where $\kappa_1'$ is a constant 
depending only on $g$ and $x$ and chosen below.  

\textbf{Rank one lift:} Let $\tilde g\in\pi_1X$ be any lift of $g$ and let 
$\widetilde A$ be a 
combinatorial geodesic axis for $\tilde g$.  If $\tilde g$ is not rank one, then 
the image  $A$ of $\widetilde A$ in $\hyp$ 
has diameter at most $3\delta$, by property~\eqref{item:median_cell} of the 
map $\Pi$ together with~\cite[Proposition 
5.1]{Hagen:boundary}.  This contradicts that $g$ is loxodromic.  Hence $\tilde 
g$ is rank one.

\textbf{Setup for verifying WPD condition:}  Let $y=g^Rx$.  Fix a combinatorial 
geodesic $\alpha$ of $\widetilde X^*$ from $x$ to $y$.  Suppose that 
$h\in\pi_1X^*$ satisfies $\dist_{\hyp}(\Pi(x),\Pi(hx))<\epsilon$ and $\dist_{\hyp}(\Pi(y), 
\Pi(hy))<\epsilon$.

\textbf{The triangle:}  Let $\beta$ be a 
$\cay(X^*)$--geodesic from $x$ to $hx$, let $\eta$ be a geodesic from $hy$ to 
$y$, and let $\gamma$ be a geodesic from $hx$ to $y$, so that we have geodesic 
triangles $\alpha^{-1}\beta\gamma$ and $\eta^{-1}(h\alpha)^{-1}\gamma$ with common side $\gamma$, as in 
Figure~\ref{fig:orientation_guide}.

\begin{figure}[h]
 \begin{overpic}[width=0.75\textwidth]{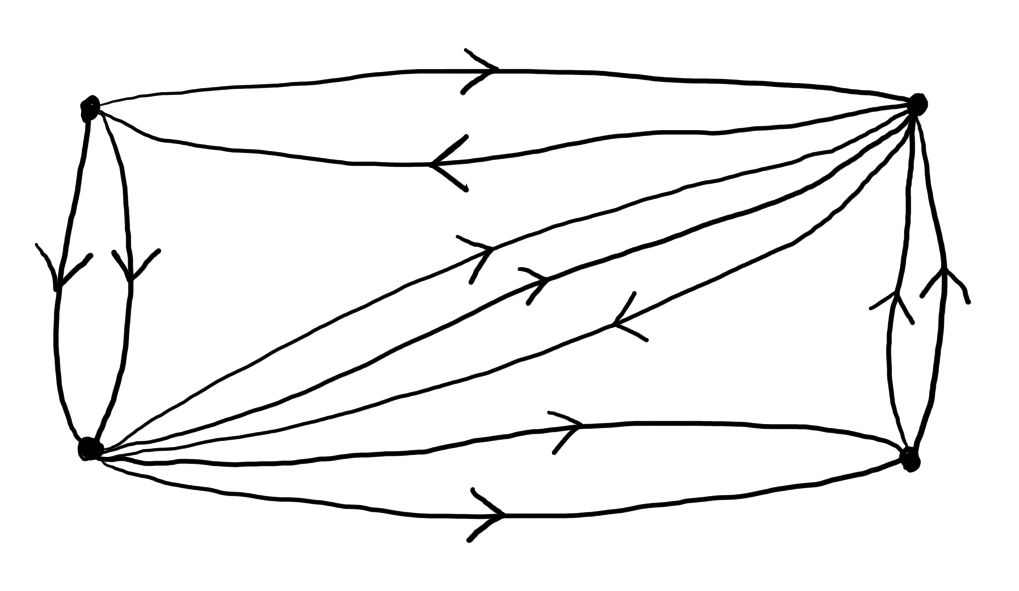}
  \put(55,52){$\alpha$}
  \put(55,3){$h\alpha$}
  \put(8,50){$x$}
  \put(91,50){$y$}
  \put(8,10){$hx$}
  \put(91,10){$hy$}
  \put(2,30){$\beta$}
  \put(95,30){$\eta$}
  \put(55,29){$\gamma$}
  \put(35,38){$A_1$}
  \put(13,25){$B_1$}
  \put(26,26){$C_1$}
 \put(70,18){$A_2$}
  \put(83,25){$B_2$}
  \put(70,27){$C_2$}
 \end{overpic}
\caption{The geodesic triangles, showing the orientations of the various paths.}\label{fig:orientation_guide}
\end{figure}

\begin{figure}[h]
\begin{overpic}[width=\textwidth]{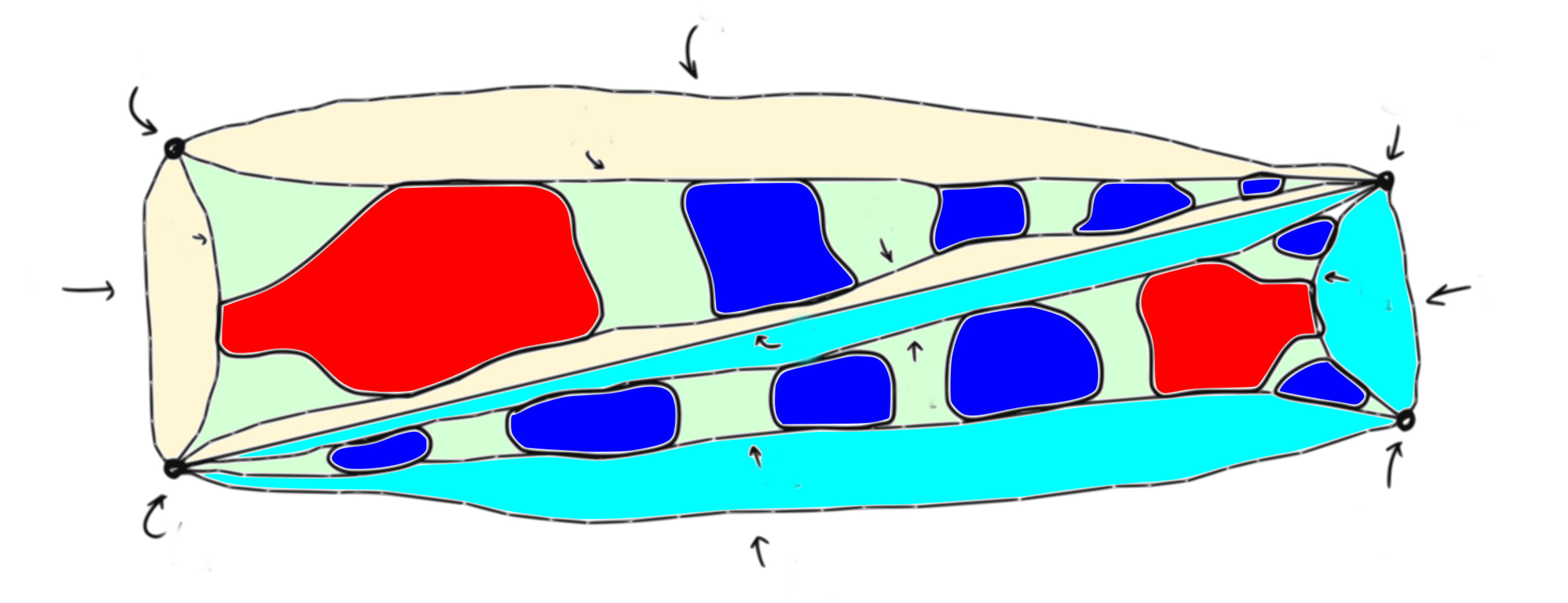}
     \put(45,37){$\alpha$}
     \put(8,33){$x$}
     \put(88,31){$y=g^Rx$}
     \put(2,19){$\beta$}
     \put(10,22){$B_1$}
     \put(11,3){$hx$}
     \put(36,30){$A_1$}
     \put(49,0){$h\alpha$}
     \put(48,7){$A_2$}
     \put(50,16){$\gamma$}
     \put(57,13){$C_2$}
     \put(55,24){$C_1$}
     \put(88,5){$hy$}
     \put(86,20){$B_2$}
     \put(95,20){$\eta$}
     
\end{overpic}

\caption{The two triangles in the proof of Lemma~\ref{lem:acyl_version_2}.  The padded ladder $L_1$ is the subdiagram of 
$S_1$ between the red subdiagram
(cone-cell or union of $3$ square grids) intersecting $A_1,B_1,C_1$ and the point $y$.  The ladder $L_2$ is the subdiagram 
of 
$S_2$ between the analogous red subdiagram (on the right) and $hx$.}\label{fig:initial_decomposition}
\end{figure}

\textbf{Applying the classification of triangles:}  By Theorem~\ref{thm:strebel_cubical_small_can}, we have a disc 
diagram $D=D_1\cup_{\gamma}D_2\to\widetilde X^*$, shown in Figure~\ref{fig:initial_decomposition}, with boundary path $\alpha^{-1}\beta(h\alpha)\eta$, with the following structure:
\begin{itemize}
 \item The diagram $D_1$ has boundary path $\alpha^{-1}\beta\gamma$ and $D_2$ has boundary path $\eta^{-1}(h\alpha)^{-1}\gamma$.
 \item The diagrams $D_1$ and $D_2$ are minimal for the given boundary paths.
 \item For $i\in\{1,2\}$, the diagram $D_i$ decomposes as $B_i^1\cup B_i^2\cup B_i^3\cup S_i$, where $S_i$ is a standard 
diagram in the sense of Theorem~\ref{thm:strebel_cubical_small_can} and each $B_i^j$ is a bigon diagram in $X$ (i.e., no 
cone cells).  The boundary path of $S_i$ is a geodesic triangle $A_iB_iC_i$, where $A_1\alpha,$ 
$B_1\beta^{-1},C_1\gamma^{-1}$ are the boundary paths of $B_1^1,B_1^2,B_1^3$ respectively, and $A_2(h\alpha)^{-1},B_2\eta^{-1}$, 
and $C_2\gamma$ are the boundary paths of $B_2^1,B_2^2,B_2^3$ respectively.
 \item The diagram $S_i$ contains a constituent padded ladder $L_i$ whose image in $\widetilde X^*$ projects under $\Pi$ to 
a set of diameter at least $R-2(\epsilon+2\delta)$, along with two ladders projecting to sets of diameter 
$\leqslant10(\epsilon+2\delta)$.  Specifically, the 
padded 
ladder $L_1$ is the subdiagram of $D_1$ obtained as follows: either $D_1$ is a ladder, in which case $L_1=D_1$, or there is 
a 
cone-cell or tripod with $3$ complementary components, all of whose closures are padded ladders; $L_1$ is the padded ladder 
among these that contains $y$.  The padded ladder $L_2$ is defined analogously.
\end{itemize}

The diagram $D$ is formed by gluing the minimal diagrams $D_1,D_2$ along $\gamma$.  Minimality of $D_1,D_2$ was used in order to apply Theorem~\ref{thm:strebel_cubical_small_can} to extract the 
padded ladders and bigons.  (It does not matter whether or not the entire diagram $D$ is minimal for its boundary path.)

The above notation is summarized in Figure~\ref{fig:initial_decomposition}.

\textbf{Bounds on cones and pseudorectangles:}  
Let $A'_1$ be the part of $A_1$ on the boundary path of the ladder $L_1$.  (Note that $(A_1')^{-1}$ starts somewhere on 
$A_1$, and ends at $y$.)

Then there is a decomposition 
$(A_1')^{-1}=\rho_0\sigma_1\rho_1\cdots\sigma_s\rho_s$, where each $\rho_i$ lies on a 
pseudorectangle and each $\sigma_i$ lies on the boundary path of a cone-cell.  
Our choice of $\Delta$ ensures that $|\sigma_i|,|\rho_i|\leqslant\Delta$, with 
the following exception: we may have $|\rho_i|>\Delta$ if the pseudorectangle 
carrying $\rho_i$ is horizontally degenerate. 

The maximal subpath $A_2'$ of $A_2$ (starting at $hx$) lying on the ladder $L_2$ decomposes as 
$\vartheta_0\varsigma_1\cdots\varsigma_\ell\vartheta_\ell$, where each $\vartheta_i$ lies on a pseudorectangle, each 
$\varsigma_i$ 
lies on a cone-cell, and each $|\varsigma_i|,|\vartheta_i|\leqslant\Delta$, except that we may have $|\vartheta_i|>\Delta$ 
if 
$\vartheta_i$ is carried on a horizontally degenerate pseudorectangle.  See Figure~\ref{fig:long_ladders}.    

\begin{figure}[h]
\begin{overpic}[width=0.9\textwidth]{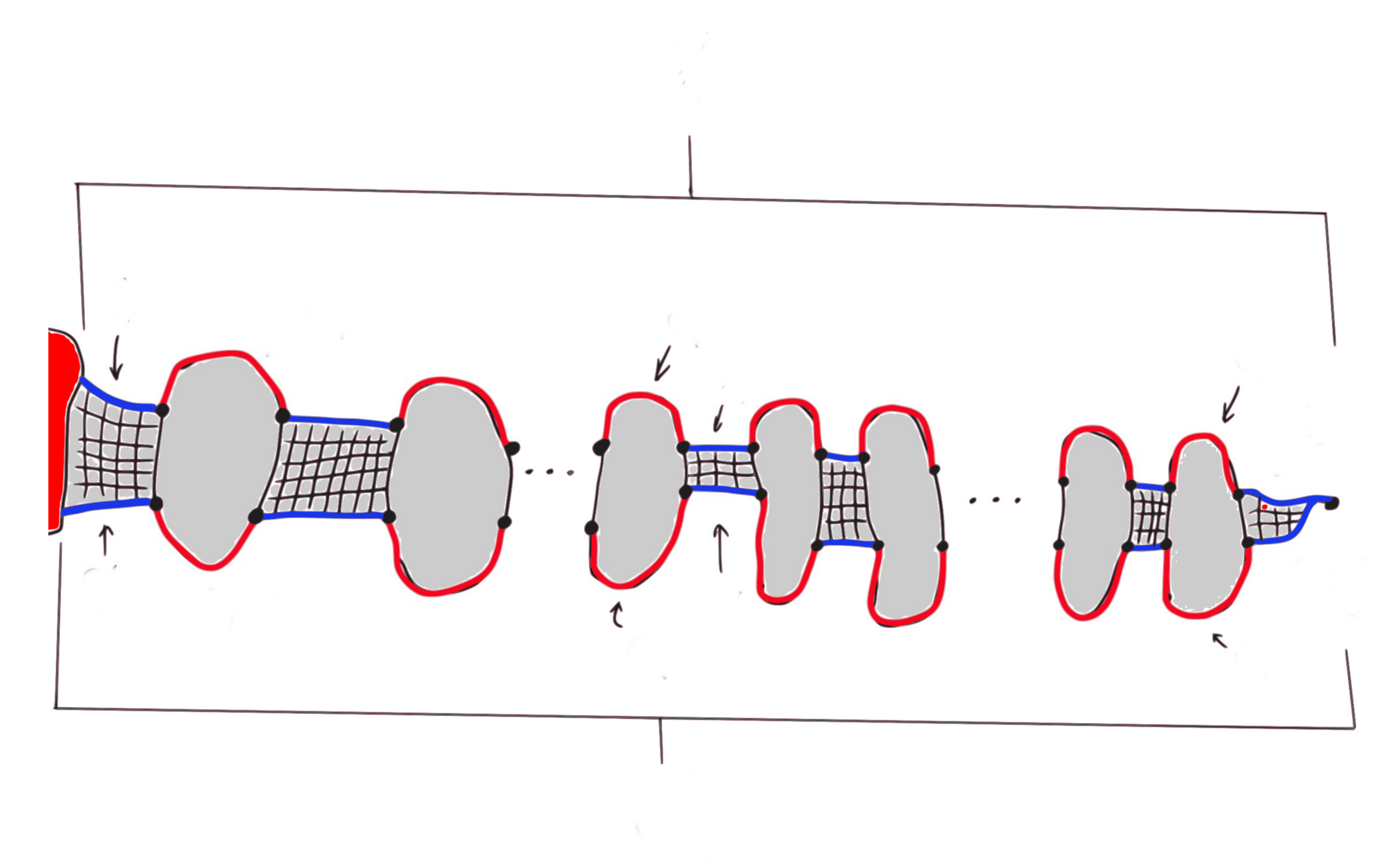}
     \put(49,55){$(A_1')^{-1}$}
     \put(47,5){$C_1'$}
     \put(7,20){$\rho_0'$}
     \put(8,40){$\rho_0$}
     \put(49,39){$\sigma_i$}
     \put(45,16){$\sigma_i'$}
     \put(52,34){$\rho_i$}
     \put(52,19){$\rho_i'$}
     \put(90,36){$\sigma_s$}
     \put(90,15){$\sigma_s'$}
\end{overpic}

\caption{The ladder $L_1$.  The paths $(A'_1)^{-1}$ and $C'_1$ travel from left to right.}\label{fig:long_ladders}
\end{figure}

\textbf{The paths $C_1',C_2'$:}  For each $i$, let $R_i$ be the pseudorectangle carrying $\rho_i$ and let $\rho'_i$ be the 
part of the boundary path of $R_i$ 
parallel to (i.e. crossing the same dual curves as) $\rho_i$.  Let $K_i$ be the cone-cell carrying $\sigma_i$ and let 
$\sigma'_i$ be the part of $\boundary_pK_i$ between $\rho_{i-1}'$ and $\rho_i'$, as shown in Figure~\ref{fig:long_ladders}.  
Let $C_1'=\rho_0'\sigma_1'\rho_1'\cdots\sigma_s'\rho_s'$ be the part of $C_1$ formed by concatenating these paths.  Define 
$\varsigma_i',\vartheta_i'$, and the resulting subpath $C_2'$ of $C_2^{-1}$ analogously.  

\begin{claim}[Fellow-traveling of $\alpha,A'_1$]\label{claim:alpha_A'}
There exist $\kappa_1,s_0\geqslant 0$, depending only on $g,x,R$, such that the following hold:
\begin{itemize}
 \item Parametrising $h\alpha$ so that $(h\alpha)(t)=h\cdot\alpha(t)$, and parametrising $A_2$ so that it starts at 
$hx$, we have $\dist(h\alpha(t),A_2'(t))\leqslant\kappa_1$ for $0\leqslant t\leqslant |A'_2|$.

\item For $0\leqslant t\leqslant |A'_1|$, we have $\dist(\alpha(t+s),(A'_1)^{-1}(t))\leqslant\kappa_1$ for some $s\leqslant s_0$.  Hence, after enlarging $\kappa_1$ by an amount depending only on $g,x,R$, we have 
$\dist(\alpha(t),(A'_1)^{-1}(t))\leqslant\kappa_1$.  
\end{itemize}
\end{claim}
\renewcommand{\qedsymbol}{$\blacksquare$}
\begin{proof}[Proof of Claim~\ref{claim:alpha_A'}]
Recall that $\tilde g$ is rank one.  Now, $\alpha A_1$ lifts to a geodesic bigon in $\widetilde X$ 
(because the diagram between them is a square diagram and hence lifts), and 
Lemma~\ref{lem:thin_bigons} shows that $\alpha$ and $A_1$ lie at Hausdorff distance $\kappa_1'$ bounded in terms of $g$ and 
$x$. 

Choose $s$ so that $\dist(\alpha(s),(A'_1)^{-1}(0))\leqslant \kappa_1'$.  For the given $\alpha$ and choice of diagram $D$, we have such an $s$.  Since $D$ is one of finitely many possible diagrams 
constructed in the given way for the finitely many choices of $\alpha,\beta,\gamma,\eta$ ($\cay(X^*)$ is locally finite and $gx^R$ is fixed), there is an upper bound $s_0$ on such $s$ depending 
only on $g,x,R$.  A computation supplies $\kappa_1$ in terms of 
$\kappa_1'$, proving the second assertion.

The first assertion follows similarly, except one does not need $s$ since $h\alpha$ and $A'_2$ have the same initial point.  (The asymmetry is because $\alpha$ and $(A'_1)^{-1}$ have the same 
\emph{terminal} point; refer to Figure~\ref{fig:orientation_guide}.)
% 
% (In particular, we choose a lift $\tilde h\tilde g\tilde h^{-1}$ of $hgh^{-1}$ and 
% note that the convex hull of its axis is a translate of that of $\tilde g$, 
% which is why applying Lemma~\ref{lem:thin_bigons} for $hgh^{-1}$ yields the 
% same constant that it did for $g$.)
\end{proof}

(Note that $s$ is small in the following sense.  Since $\Pi\circ\alpha$ has length at least $\tau R$ in $\hyp$, and 
$\dist_{\hyp}(\Pi((A'_1)^{-1}(0)),\Pi(x))\leqslant 
10(\epsilon+2\delta)+2\delta$, the path $\alpha|_{[s,|\alpha|]}$ projects to a path in $\hyp$ of length at least $\tau 
R-(10\epsilon+22\delta+\kappa_1')$.  Since $\Pi$ is lipschitz, the same number bounds from below the length of 
$\alpha|_{[s,|\alpha|]}$.  We won't need this fact, but we will use 
a similar fact about the paths $C_1',C_2'$ established below.)

\textbf{Fellow-traveling of $C_1',C_2'$:}  Next, consider the subdiagram 
$E=B_1^3\cup_\gamma B_2^3$ of $D$ bounded by $C_1C_2$.  Since $C_1,C_2$ 
are geodesics, and $E$ is a square diagram, every dual curve starting on $C_1$ 
ends on $C_2$ and every dual curve starting on $C_2$ ends on $C_1$.  

Let $K,K'$ be dual 
curves that emanate from $C'_1$ and cross each other.  Let $C''_1$ 
be the subpath of $C'_1$ between and including the $1$--cubes dual to $K$ and 
$K'$.  These $1$--cubes $e,f$ respectively lie on common cone-cells or 
pseudorectangles with points $a_e,a_f$ on $A'_1$.  We can choose these so that 
$\dist(e,a_e),\dist(f,a_f)\leqslant2\Delta$.

Now, $a_e,a_f$ respectively lie at distance 
at most $\kappa_1$ from points on $b_e,b_f\in\alpha$.  Hence $\dist(e,b_e)\leqslant2\Delta+\kappa_1$ and 
$\dist(f,b_f)\leqslant 2\Delta+\kappa_1$.  Thus $\dist(b_e,b_f)\geqslant|C''_1|-4\Delta-2\kappa_1$.  On the 
other hand, 
$\dist_\hyp(\Pi(b_e),\Pi(b_f))\leqslant2\delta+4\Delta+2\kappa_1$, by 
property~\eqref{item:median_cell} of $\Pi$ and the 
fact that $K,K'$ cross.  (Indeed, there is a path in $\neb(K)\cup\neb(K')$ from 
$e$ to $f$, so $\dist_\hyp(\Pi(e),\Pi(f))\leqslant2\delta$, and 
$\neb(K),\neb(K')$ map to hyperplane carriers in $\cay(X^*)$, whose images in 
$\hyp$ have diameter at most $\delta$.)

Let $\upsilon_\alpha=\max\frac{\dist(p,q)}{\dist_\hyp(\Pi(p),\Pi(q))}$, where $p,q$ 
vary over vertices 
of $\alpha$ with distinct images in $\hyp$, and let $\upsilon$ be the maximal $\upsilon_\alpha$ over the (finitely many) 
choices of 
$\alpha$ with the given endpoints (here we are using that $X$, and hence $\cay(X^*)$, is locally finite).    
Let $\zeta_\alpha$ be the maximum of 
$\dist(p,q)$ as $p,q$ vary over vertices of $\alpha$ with $\Pi(p)=\Pi(q)$, and 
let $\zeta$ be the maximum of the $\zeta_\alpha$ over all choices of $\alpha$.  
Note that $\upsilon,\zeta$ depend on $g,x$ and $R$, but not $h$ or $\alpha$.

So, $|C''_1|\leqslant ( 2\delta+10\Delta+2\kappa_1)\upsilon + 
\zeta+10\Delta+2\kappa_1$.  In other words, there exists $N$ depending only on 
$g,x,R$ such that any two dual curves that emanate from $C_1'$ at distance more 
than $N$ cannot cross.

Parametrise $C_1,C_2$ so that $C_1(0)=C_2^{-1}(0)=hx$.  Our choice of $R$ ensures 
that there exist $t_0,t_0'$, depending only on $R,\epsilon,\Delta,\delta,\tau$ so that $C_1(t),C_2^{-1}(t)$ lie on $C'_1,C'_2$ respectively when $t_0\leqslant
t\leqslant t_0'$.  

Let $t=\frac{t_0'-t_0}{2}$ and let $z=C_1(t)$.  Let 
$z'=C_2^{-1}(t)$.  Since $E$ is a square diagram, it lifts to a square diagram in $\widetilde X$ bounded by geodesics lifting 
$C_1,C_2^{-1}$.  So $\dist(z,z')$ is bounded by the number of dual curves $K$ in $E$ that cross $C_1$ before $z$ and $C_2^{-1}$ after 
$z'$, or vice versa.  Our choice of $R$ ensures that any such $K$ cannot cross $C_1-C_1'$ or $C_2-C_2'$, since 
property~\eqref{item:median_cell} would 
then provide a shortcut in $\hyp$ from $\Pi(hx)$ to $\Pi(gx)$.  Let $K$ be a dual curve crossing $C_1'$ before $z$ and 
$C_2'$ after $z$ and let $L$ be a dual curve crossing $C_2'$ before $z'$ and $C_1'$ after $z$.  Then the distance along 
$C_1'$ between $K$ and $L$ is at most $N$, so $K$ and $L$ are $N$--close to $z$.  Thus $\dist(z,z')$ is bounded in terms of 
$N$, 
say by some $\kappa_2$.

\textbf{Fellow-traveling of $(A'_1)^{-1},C_1'$ and $A'_2,C'_2$:}  By 
Lemma~\ref{lem:thin_ladders}, there exists $\kappa_3$, depending on $\Delta$ 
and the small-cancellation assumption, so that 
$\dist((A'_1)^{-1}(t),C_1'(t))\leqslant\kappa_3$.  The same is true 
for $A'_2,C'_2$.

\textbf{Conclusion:}  Let $z,z'$ be as above.  Then $z$ is uniformly close to 
$\alpha(t)$ (the distance is bounded by $\kappa_1+\kappa_3$), and the same is 
true for $h\alpha(t),z'$.  Hence $\dist(\alpha(t),h\alpha(t))\leqslant 
2(\kappa_1+\kappa_3)+\kappa_2$, which does not depend on $h$. Since $\cay(X^*)$ 
is locally finite and $\pi_1X^*$ acts freely on $\cay(X^*)$, the action of 
$\pi_1X^*$ on $\cay(X^*)$ is metrically proper and hence there are 
finitely many such $h$.
\renewcommand{\qedsymbol}{$\Box$}
\end{proof}

\section{The hyperbolic space $\hyp$ and the projection $\Pi:\cay(X^*)\to\hyp$}\label{sec:hyp}
Let $\langle X\mid\{Y_i\}_{i\in\mathcal I}\rangle$ be a cubical $C'(\frac{1}{144})$ presentation and define a space $\hyp$ as follows.  
First, let $\hyp'$ be the $1$--skeleton of $\widetilde X^*$.  This consists of the $1$--skeleton of $\cay(X^*)$, together 
with a combinatorial cone on each lift of each $Y_i$.  We form $\hyp$ from $\hyp'$ by adding a combinatorial cone on the 
carrier of each hyperplane.

We also have a projection $\Pi\colon\cay(X^*)\to\hyp$, defined as follows.  On the $1$--skeleton of $\cay(X^*)$, we declare 
$\Pi$ to be the inclusion.  If $c$ is a cube of $\cay(X^*)$ with $\dimension(c)\geqslant2$, we send $c$ arbitrarily to a 
point in the image of its $1$--skeleton.  However, we require this choice to be made $\pi_1X^*$--equivariantly, so that 
$\Pi$ is $\pi_1X^*$--equivariant.  Obviously $\Pi$ is coarsely surjective and $1$--Lipschitz on the $1$--skeleton of 
$\cay(X^*)$.  By construction, 
$\Pi$ sends each cone to a set of diameter $\leqslant2$, while hyperplane carriers in $\cay(X^*)$ are sent to 
subsets of $\hyp$ with diameter at most $2$.  Hence, to see that $\hyp$ and $\Pi$ satisfy the conditions required in 
Section~\ref{sec:metalemma}, we need only to prove that $\hyp$ is hyperbolic.

\begin{lem}[Square bigons have thin projection]\label{lem:square_bigon_project}
Let $\alpha,\beta\to\cay(X^*)$ be geodesics with common endpoints, and suppose that $\alpha\beta$ bounds a disc diagram 
$D\to\widetilde X^*$ that does not contain any cone-cells.  Then $\Pi(\alpha),\Pi(\beta)$ lie at uniformly bounded Hausdorff 
distance in $\hyp$.
\end{lem}

\begin{proof}
Let $e$ be a $1$--cube of $\alpha$ and let $K$ be the dual curve in $D$ emanating from $e$ and mapping to a hyperplane $H$ 
of $\cay(X^*)$.  Since $\alpha$ is a geodesic, $K$ terminates at a $1$--cube $f$ of $\beta$, whence 
$\dist_\hyp(\Pi(e),\Pi(f))\leqslant2$.  Hence $\Pi(\alpha)\subseteq\neb_2(\Pi(\beta))$ and the proof is complete by 
symmetry.
\end{proof}

\begin{prop}\label{prop:hyperbolicity}
The graph $\hyp$ is hyperbolic.
\end{prop}

\begin{proof}
It suffices to prove that the $0$--skeleton of $\cay(X^*)$, with the subspace metric inherited from $\hyp$, is hyperbolic.  
First, suppose that $\alpha\beta\gamma$ is a geodesic triangle in $\cay(X^*)$.  Then Lemma~\ref{lem:square_bigon_project} 
combines with Theorem~\ref{thm:strebel_cubical_small_can} and the fact that $\Pi$ sends cones to uniformly bounded sets to 
show that each of $\Pi(\alpha),\Pi(\beta),\Pi(\gamma)$ is contained in the $\delta'$--neighborhood in $\hyp$ of the union of 
the other two, for some uniform $\delta'$.  The Guessing Geodesics Lemma~(see 
e.g.~\cite[Proposition~3.5]{Hamenstadt:gg}\cite[Proposition~3.1]{Bowditch:uniform}) now implies that $\hyp$ is hyperbolic.
\end{proof}

\begin{thm}\label{thm:general_case}
Let $\langle X\mid\{Y_i\}_{i\in\mathcal I}\rangle$ be a $C'(\frac{1}{144})$ presentation with $X$ locally finite, and let 
$G=\pi_1X^*$.  Then any $g\in G$ acting on the space $\hyp$ constructed above as a fast loxodromic element acts on $\hyp$ as 
a WPD element, whence either $G$ is virtually cyclic or acylindrically hyperbolic.
\end{thm}

\begin{proof}
The assertion that $g$ is a WPD element follows from Lemma~\ref{lem:acyl_version_2}; hyperbolicity of $\hyp$ comes from 
Proposition~\ref{prop:hyperbolicity}.  Applying~\cite[Theorem~1.2.($AH_3\Rightarrow AH_2$)]{Osin:acyl} completes the proof.
\end{proof}

\section{Proof of Theorem~\ref{thmi:main}}\label{sec:finding_loxodromic}
We now study the question of when $\pi_1X^*$ contains a loxodromic isometry of $\hyp$, 
using knowledge of which 
elements of $\pi_1X$ act loxodromically on the \emph{contact graph} $\contact\widetilde X$ of $\widetilde X$, which is the 
intersection graph of the set of hyperplane carriers in $\widetilde X$.

Let $p\colon\widetilde X\to\cay(X^*)$ be the universal covering map (regarding $\cay(X^*)$ as the cover of $X$ corresponding to 
the 
subgroup $K=\langle\langle\{\pi_1Y_i\}_{i\in\mathcal I}\rangle\rangle$ of $\pi_1X$).  Let $\widetilde{\hyp}$ be the graph 
obtained from $\widetilde X^{(1)}$ by coning off the $1$--skeleton of each hyperplane carrier.

Form a new graph $\widehat{\hyp}$ from $\widetilde{\hyp}$ by coning off every subgraph of $\widetilde 
X^{(1)}\subseteq\widetilde{\hyp}$ which is the $1$--skeleton of an elevation $\widetilde Y_i\hookrightarrow\widetilde X$ of 
some $Y_i\to X$.  
Observe that $p$ induces a quotient map $\hat p\colon\widehat{\hyp}\to\hyp$, which restricts to $p$ on $\widetilde X^{(1)}$ and 
which sends the cone-point $v_H$ over the hyperplane carrier $\neb(H)$ to the cone-point $v_{p(H)}$ over the 
hyperplane carrier $p(\neb(H))$.  The map $\hat p$ also sends the cone over each elevation $\widetilde Y_i$ to the cone over the 
corresponding lift $Y_i\to\widetilde X^*$ of $Y_i\to X$.

\begin{rem}[Standing assumptions]\label{rem:additional_assumptions}
 In this section, we introduce extra hypotheses on the $C'(\frac{1}{144})$ 
cubical presentation $\langle X\mid\{Y_i\}_{i\in\mathcal I}\rangle$.  First, we 
assume that $X$ is compact and that each $Y_i$ is compact, as in Theorem~\ref{thmi:main}.  Second, we 
assume that $\langle X\mid\{Y_i\}_{i\in\mathcal I}\rangle$ satisfies the 
uniform $C''(\frac{1}{144})$ condition from 
Definition~\ref{defn:metric_small_c}, again as in Theorem~\ref{thmi:main}.  
Recall that this implies that each geodesic $P$ in any piece satisfies 
$|P|<\frac{1}{144}\inf_{i\in\mathcal I}\|Y_i\|$.  In particular, there is a 
uniform bound on the lengths of such $P$.  Later, we will introduce additional conditions.

Specifically, we will replace each $Y_i$ by a finite cover $\widehat Y_i\to Y_i\to X$.  Note that the 
collection of elevations $\widetilde Y_i\to\widetilde X$, and their stabilisers in $\pi_1X$, do not depend on 
the cover $\widehat Y_i$.
\end{rem}

\subsection{Coned-off spaces}  We first relate the various coned-off spaces to the corresponding graphs in a 
standard way.

\begin{lem}\label{lem:qi_to_int_graph}
Let $\Gamma$ be the graph with a vertex for each hyperplane carrier in $\widetilde X$, and a vertex for each elevation $\widetilde Y_i\hookrightarrow\widetilde X$ of each 
$Y_i\to X$, with adjacency corresponding to intersection.  Then $\Gamma$ is $\pi_1X$--equivariantly 
quasi-isometric to $\widehat{\hyp}$ and $\widetilde{\hyp}$ is $\pi_1X$--equivariantly quasi-isometric to $\contact\widetilde 
X$.
\end{lem}

\begin{proof}
Define a map $f\colon\Gamma^{(0)}\to\widehat{\hyp}$ by sending each vertex (corresponding to a hyperplane carrier or an elevation $\widetilde 
Y_i$) to the corresponding cone-point.  Since each 
point of $\widetilde X$ lies in a hyperplane carrier, the map $f$ is quasi-surjective.  If $v,w$ are vertices of $\Gamma$, 
corresponding to subcomplexes $\widetilde C_v,\widetilde C_w$, and $v,w$ are adjacent, then $\widetilde C_v\cap\widetilde 
C_w\neq\emptyset$, so $\dist_{\widehat{\hyp}}(f(v),f(w))\leqslant 2$.  Hence $f$ is coarsely Lipschitz.  By sending 
each cone-point $v$ in $\widehat{\hyp}$ to the vertex of $\Gamma$ corresponding to the subcomplex over which $v$ is the 
cone, we obtain a coarsely Lipschitz quasi-inverse for $f$, so $f$ is a quasi-isometry.  That $f$ is $\pi_1X$--equivariant 
follows immediately from the definition of $f$.  This shows that $\Gamma$ is $\pi_1X$--equivariantly 
quasi-isometric to $\widehat{\hyp}$.  The other assertion is proved in a similar manner in~\cite[Section 
5]{Hagen:quasi_arb}. 
\end{proof}

\begin{rem}[Summary of the graphs]\label{rem:graph_summary}
The various graphs are summarised below for reference:
\begin{itemize}
 \item $\contact\widetilde X$ is the contact graph of $\widetilde X$, i.e. the intersection graph of the hyperplane carriers.
 \item $\Gamma$ is the intersection graph of the family of subcomplexes of $\widetilde X$ that are either hyperplane carriers or elevations $\widetilde Y_i$ of relators.
 \item $\widetilde{\hyp}$ is the graph formed from $\widetilde X^{(1)}$ by coning off the $1$--skeleta of the various hyperplane carriers.
 \item $\widehat{\hyp}$ is the graph formed from $\widetilde{\hyp}$ by coning off the subgraphs of $\widetilde X^{(1)}\subseteq\widetilde{\hyp}$ that are $1$--skeleta of subcomplexes of the form 
$\widetilde Y_i$.
\item $\hyp$ is formed from the $1$--skeleton of $\widetilde X^*$ by coning off the $1$--skeleton of each hyperplane carrier, as in Section~\ref{sec:hyp}. 
\end{itemize}
The first and third graphs are $\pi_1X$--equivariantly quasi-isometric, and the second and fourth graphs are $\pi_1X$--equivariantly isometric.

The graph $\hyp$ is the ``main object of study'', and was chosen because it generalises the hyperbolic space used in the graphical case~\cite{GruberSisto:graphical}.  It is therefore simplest in many 
settings to work with $\widetilde{\hyp}$ and $\widehat{\hyp}$, because of the map $p$.

But in other settings, we invoke results in the literature about actions on $\contact\widetilde X$.  Moreover, for statements like Lemma~\ref{lem:loxodromic_persist_1}, it is a bit simpler to work 
with $\contact\widetilde X$ and $\Gamma$ rather than $\widetilde{\hyp}$ and $\widehat{\hyp}$.  This explains why we use all of the different graphs.  The reader is encouraged to refer to the above 
list as needed. 
\end{rem}

\subsection{When do loxodromics on $\widetilde{\hyp}$ stay loxodromic on coning?}  Since our ultimate goal is 
to understand when $\pi_1X^*$ contains a loxodromic isometry of $\hyp$, and existing tools 
(mainly from~\cite{CapraceSageev:rank_rigidity,Hagen:boundary}) tell us when elements of $\pi_1X$ act loxodromically on 
$\contact\widetilde X$, we need to relate these phenomena.

\begin{lem}[Loxodromics persist upstairs]\label{lem:loxodromic_persist_1}
Let $\tilde g\in\pi_1X$ act loxodromically on $\widetilde{\hyp}$.  Then either $\tilde g$ acts loxodromically on 
$\widehat{\hyp}$ or $\tilde g$  stabilizes some elevation $\widetilde Y_i\subseteq\widetilde X$ of 
some $Y_i\to X$.
\end{lem}

\begin{proof}
 By Lemma~\ref{lem:qi_to_int_graph}, $\widehat{\hyp}$ is quasi-isometric to the intersection graph $\Gamma$ of the set of 
hyperplane carriers and various elevations $\widetilde Y_i$  in $\widetilde X$.  Hence it suffices 
to show that $\langle \tilde g\rangle$ acts loxodromically on $\Gamma$ provided that $\tilde g$ doesn't stabilise any 
$\widetilde Y_i$.    

\textbf{The axis:}  Let $\widetilde A\subseteq\widetilde X$ be a combinatorial geodesic axis for $\tilde g$ (by replacing 
$\widetilde X$ by its first cubical subdivision, we may assume that such an axis exists~\cite{Haglund:semisimple}).  Fix a 
$0$--cube $\tilde a\in\widetilde A$ and fix $n>0$.  Let $P$ be the subpath of $\widetilde A$ joining $\tilde a$ to $g^n\tilde 
a$.  

\textbf{The corresponding geodesic in $\Gamma$:}  Let $\widehat Q$ be a geodesic of $\Gamma$ joining vertices corresponding 
to subcomplexes containing $\tilde a$ and $g^n\tilde a$.  Let the vertex-sequence of $\widehat Q$ be $C_0,\ldots,C_N$, where 
each $C_i$ is either a hyperplane carrier in $\widetilde X$ or an elevation $\widetilde Y_i\subseteq \widetilde X$.  Then we 
have a combinatorial path $Q=\alpha_0\cdots\alpha_N$ joining $\tilde a$ to $g^n\tilde a$, where each $\alpha_i$ is a geodesic 
in $C_i$.  The closed path $QP^{-1}$ bounds a disc diagram $D\to\widetilde X$.

\begin{claim}\label{claim:good_Q}
The geodesic $\widehat Q$, and the corresponding path $Q$, can be chosen so that $Q$ is a geodesic of $\widetilde X$.     
\end{claim}
\renewcommand{\qedsymbol}{$\blacksquare$}
\begin{proof}[Proof of Claim~\ref{claim:good_Q}]
This argument is entirely about CAT(0) cube complexes (i.e. it takes place in $\widetilde X$). The argument is 
almost exactly the same as the proof of Proposition~3.1 of~\cite{BHS:HHS_I}, with one change: the argument in 
\cite{BHS:HHS_I} corresponds to the case $\mathcal I=\emptyset$, but it only uses convexity of carriers; since the 
$\widetilde Y_i$ are convex, the same argument works here.

Suppose that $\widehat Q$ and $Q$ have been chosen as above in such a way that any minimal-area disc diagram $D\to\widetilde 
X$ with boundary path $QP^{-1}$ has area that is minimal for all such choices of $\widehat Q,Q$, and $Q$ has 
no backtracks. (If there are backtracks, we can always remove them.  This does not quite make $Q$ a geodesic, 
since it can still have distinct non-consecutive edges dual to the same hyperplane.)  

Let $K$ be a dual curve in $D$ emanating from $P$.  Then $K$ ends on $Q$.  Indeed, since $P$ is a 
geodesic of $\widetilde X$, the hyperplane to which $K$ maps is dual to at most one $1$--cell of $P$.

Now let $K$ be a dual curve in $D$ emanating from $Q$.  Then $K$ emanates from some $\alpha_i$.  Suppose 
that $K$ ends on $\alpha_j$.  Since $\alpha_i$ is a geodesic, we have $j\neq i$.  If $|i-j|>2$, then since the hyperplane to 
which $K$ maps has carrier $\Gamma$--adjacent to $C_i,C_j$, we have contradicted that $\widehat Q$ is a $\Gamma$--geodesic.  

If $j=i\pm2$, then we can replace $C_{i\pm1}$ with the carrier of the hyperplane to which $K$ maps (i.e., modify $\widehat 
Q$ and $Q$), providing a new choice of $\widehat Q$ leading to a lower-area choice of $D$.

Finally, if $j=i\pm1$, then we can apply hexagon moves to show that $\alpha_i$ has a terminal segment coinciding with an 
initial segment of $\alpha_{i+1}$ (say), contradicting that $Q$ has no backtrack.

We conclude that $D$ can be chosen so that all dual curves travel from $Q$ to $P$.  Hence, $Q$ and $P$ have the same 
length, so $Q$ is a geodesic of $\widetilde X$.  
\end{proof}

Fix $\widehat Q$ and $Q$ as in Claim~\ref{claim:good_Q}.

Lemma~\ref{lem:qi_to_int_graph} provides $\lambda,\mu\geqslant 1$ such that $\Gamma$ and $\widehat{\mathcal H}$ are 
$(\lambda,\mu)$--quasi-isometric, so we have $\dist_{\widehat{\mathcal H}}(a,\tilde g^na)\geqslant |\widehat Q|/10\lambda$ for 
all sufficiently large $n$.  Hence, it suffices to bound $|\widehat Q|$ from below by a linear function of $n$ under the 
assumption that $\tilde g$ does not have a power stabilising any $\widetilde Y_i$.

Now, $|\widehat Q|=N+1\geqslant \frac{|Q|}{\max_i|\alpha_i|}$.  We now make an auxiliary 
claim: 

\begin{claim}\label{claim:stabiliseY}
Suppose that $\tilde g$ does not have a positive power that stabilises any elevation $\widetilde Y_i$ of any $Y_i$. Then 
there exists $L\geqslant 1$, depending only on $\tilde g$, such that the following holds.  Let $\widetilde B$ be either a 
hyperplane of $\widetilde X$ or an elevation of some $Y_i\to X$.  Then at most $L$ hyperplanes cross both $\widetilde B$ and 
the axis $\widetilde A$.
\end{claim}

\begin{proof}[Proof of Claim~\ref{claim:stabiliseY}]
First suppose that $\widetilde B$ is a hyperplane.  Since $\tilde g$ is loxodromic on $\widetilde{\mathcal H}$, 
Lemma~\ref{lem:qi_to_int_graph} implies that $\tilde g$ is loxodromic on $\mathcal C\widetilde X$, and the claim follows.

Suppose that for each $J\geqslant 0$, there exists $\widetilde B_J$, an elevation of some $Y_i\to X$, such that at least $J$ 
hyperplanes cross both $\widetilde B_J$ and $\widetilde A$.  

Consider the set of hyperplanes crossing $\widetilde A$.  We first observe that there exists $n_0\geqslant 1$ and  a
hyperplane $H$ crossing $\widetilde A$ such that $\langle \tilde g^{n_0}\rangle\cdot H$ is an infinite collection of disjoint 
hyperplanes.  Indeed, otherwise $\langle \tilde g\rangle$ has a bounded orbit in $\mathcal C\widetilde X$, contradicting the 
assumption that $\tilde g$ is loxodromic.

Let $J'\gg0$ be an integer to be determined.  Since $\tilde g$ is a rank-one element, Remark~\ref{rem:rank_one_hull} 
implies that there exists $J$ such that at least $J'$ elements of $\langle\tilde g^{n_0}\rangle\cdot H$ cross both 
$\widetilde A$ and $\widetilde B_J$.  By translating, we can assume that $H,\tilde g^{n_0}H,\cdots,\tilde g^{n_0(J'-1)}H$ 
cross $\widetilde B_J$.

Therefore, $\tilde g^{n_0}H,\ldots,\tilde g^{n_0(J'-1)}H,\tilde g^{n_0J'}H$ cross $\tilde g^{n_0}\widetilde B_J$.  Hence, there 
are $J'-1$ hyperplanes 
that cross both $\widetilde B_J$ and $\tilde g\widetilde B_J$ and cross the geodesic $\widetilde A$.  For $J'$ sufficiently 
large (in terms of the bound on the diameter of cone-pieces provided by the uniform small-cancellation condition, as in 
Remark~\ref{rem:additional_assumptions}), this gives a contradiction unless $\tilde g^{n_0}\widetilde B_J=\widetilde B_J$.  
This 
proves the claim.
\end{proof}

Suppose that no positive power of $\tilde g$ stabilises an elevation of a relator.

Letting $L$ be the constant from Claim~\ref{claim:stabiliseY}.  Our choice of $Q$ guarantees that each hyperplane crossing 
each $\alpha_i$ crosses both the subcomplex $C_i$ and the axis $\widetilde A$.  Hence Claim~\ref{claim:stabiliseY} implies 
that $|\alpha_i|\leqslant L$.  Thus, $|\widehat Q|\geqslant L^{-1}|Q|=L^{-1}\dist_{\widetilde X}(a,\tilde g^na)$ (and $L$ is 
independent of $n$).  Thus, $\dist_{\widehat{\mathcal H}}(a,\tilde g^na)\geqslant(10\lambda L)^{-1}\dist_{\widetilde X}(a,\tilde 
g^na)$.  Since the latter quantity is bounded below by a linear function of $n$ (because $\tilde g$ acts hyperbolically on 
$\widetilde X$), so is the former.  So, $\tilde g$ is loxodromic on $\widehat{\mathcal H}$.

Finally, if for some $n>0$ we have $\tilde g^n\in\stabilizer_{\pi_1X}(\widetilde Y_i)$, then let $\widetilde A$ 
be an axis for $\tilde g$.  So $\widetilde A$ lies in a regular neighbourhood of both $\widetilde Y_i$ and 
$\tilde g\widetilde Y_i$, leading to impossibly large pieces unless $\tilde 
g\in\stabilizer_{\pi_1X}(\widetilde Y_i)$.
\renewcommand{\qedsymbol}{$\Box$}
\end{proof}

\subsection{Asystolicity and embeddability}  We now consider two properties of lifts of an element of 
$\pi_1X^*$ to $\pi_1X$,  along the homomorphism $\pi_1X\to \pi_1X^*$.  \emph{Embeddability} of a lift guarantees 
that its image has infinite order.  \emph{Asystolicity} is stronger and more concrete.

\begin{defn}[Embeddable, asystolic]\label{defn:short_embeddable_asystolic}
Fix $g\in\pi_1X^*$.  Any $\tilde g\in\pi_1X$ mapping to $g$ is a \emph{lift} of $g$.  Let $\widetilde A$ be a combinatorial 
geodesic axis for $\tilde g$, which exists provided $\tilde g\ne1$, since $\pi_1X$ is torsion-free and isometries of 
$\widetilde X$ are combinatorially semisimple~\cite{Haglund:semisimple}. Recall that $p\colon \widetilde X\to\cay(X^*)$ denotes 
the universal covering map.  Let $A$ denote $p(\widetilde A)$.

\begin{enumerate}
 \item $\widetilde A$ is an \emph{embeddable axis} if $p$ restricts to a cubical isomorphism $\widetilde A\to A$, i.e. if $A$ is 
an embedded combinatorial line in $\cay(X^*)$.

\item We say that $\tilde g$ is an \emph{embeddable lift} of $g$ if $\tilde g$ has \textbf{at least one} 
embeddable axis.
\end{enumerate} 

Now fix $\lambda\in[0,1)$.  

\begin{enumerate}
\setcounter{enumi}{2}
\item We say that $\widetilde A$ is a \emph{$\lambda$--asystolic axis} 
if for each subpath $P$ of $\widetilde A$ such that $P\subseteq\widetilde Y_i$, where $\widetilde Y_i$ is an elevation of a 
relator $Y_i$, we have $|\widetilde P|<\lambda\|Y_i\|$.

\item The lift $\tilde g$ is \emph{$\lambda$--asystolic} if \textbf{every} combinatorial geodesic axis of $\tilde g$ is a 
$\lambda$--asystolic axis.
\end{enumerate}
\end{defn}

Note that if $\tilde g$ is an embeddable lift, then $g$ has infinite order, and $\tilde g^n$ is an embeddable lift of $g^n$ for all $n>0$.  Indeed, if $\widetilde A$ is an embeddable axis for 
$\tilde g$, it is 
also an embeddable axis for $\tilde g^n$.  Note that if $\tilde g$ 
is an $\lambda$--asystolic lift of $g$, then 
$\tilde g^n$ is an $\lambda$--asystolic lift of $g^n$ for all $n>0$.

\begin{lem}[Embeddability from $\frac{35}{72}$--asystolicity]\label{lem:characterising_a_shortest}
Let $\tilde g\in\pi_1X$ and let $g$ be its image in $\pi_1X^*$.  

Suppose that $\widetilde A$ is a $\frac{35}{72}$--asystolic axis for $\tilde g$.  Then $\widetilde A$ is an embeddable axis for $\tilde g$, so $\tilde g$ is an embeddable lift of $g^n$ for all 
$n\in\integers-\{0\}$, and $g$ has infinite order.
%In particular, if $\tilde g$ is a $\frac{35}{72}$--asystolic lift of $g$, then $\tilde g^n$ is an embeddable lift of $g^n$ for all $n$. 
\end{lem}

\begin{proof}
It suffices to prove the claim for $n=1$. Let $\widetilde A$ be a combinatorial geodesic axis for $\tilde g$, and suppose that $\widetilde A$ is $\frac{35}{72}$--asystolic. 

Suppose that $\widetilde A$ is not an embeddable axis.  Then $p\colon \widetilde A\to A$ is not injective, so there 
exist distinct 
$0$--cubes $\tilde y,\tilde y'\in\widetilde A$ such that $p(\tilde y)=p(\tilde y')$.  In other words, letting $\widetilde P$ 
be the subpath of $\widetilde A$ joining $\tilde y$ to $\tilde y'$, the path $P=p\circ \widetilde P$ in $\cay(X^*)$ is a 
nontrivial closed path.  Let $D\to\widetilde X^*$ be a minimal-complexity disc diagram with boundary path $P$.  

\textbf{No spurs:}  We claim that $P$ has no spurs.  Indeed, if $P$ has a spur $ee^{-1}$, then $\widetilde P$ 
contains a subpath $\tilde 
e_1\tilde e_2$, where $\tilde e_1,\tilde e_2$ are distinct $1$--cubes such that $\tilde e_1\cap\tilde e_2$ is a $0$--cube 
and $p(\tilde e_1)=p(\tilde e_2)^{-1}$.  
By choosing $\tilde y,\tilde y'$ as close as possible, we can assume that the endpoints of $\tilde e_1\tilde e_2$ are 
$\tilde y,\tilde y'$, and $\tilde e_1,\tilde e_2$ are lifts of the $1$--cube $e$.  Hence $\tilde y=h\tilde y'$ for some 
nontrivial $h\in\kernel(\pi_1X\to\pi_1X^*)$, so $h$ fixes the $0$--cube  $\tilde e_1\cap\tilde e_2$.  This contradicts that 
$\pi_1X$ acts on $\widetilde X$ freely.  

\textbf{Applying diagram trichotomy:}  Hence Theorem~\ref{thm:Greendlinger} implies that $D$ is one of the 
following:
\begin{itemize}
     \item A single vertex.  This is impossible since $P$ is nontrivial.
     \item A single cone-cell.  In this case, $P$ is an essential path in a relator $Y_i$, by minimality of the complexity.  
Hence, $|\widetilde P|\geqslant\|Y_i\|$, contradicting asystolicity.
     \item A ladder, or a diagram with at least three features of positive curvature (shells or generalised corners).
\end{itemize}

In either of the latter two cases, there are at least two features of positive curvature, and neither is a 
spur.  

\textbf{No shells:}  Suppose 
that $C$ is a positively-curved shell in $D$ with boundary path $OI$, with $O$ the outer path and $I$ the inner path.  Let 
$Y_i$ be the 
relator to which the path $\boundary_pC$ maps.  Let $I'$ be the shortest path in $Y_i$ that is square-homotopic in $Y_i$ (relative to its endpoints) to $I$.  By short inner paths 
(Theorem~\ref{thm:short_inner_paths}), $|I'|<|O|$.  Now, $O$ lifts to a subpath of 
$\widetilde P$ lying in an elevation $\widetilde Y_i$ of $Y_i$, so by asystolicity, $|O|<\|Y_i\|/2$.  So $|I'O|<\|Y_i\|$.  
Hence $I'O$ is inessential in $Y_i$.  Since $I,I'$ are square-homotopic in $Y_i$, we have that $IO=\boundary_pC$ is also 
inessential, so we can replace $C$ by a square diagram to reduce the complexity.  Thus, the shell $C$ cannot exist.

\textbf{Shuffling generalised corners to the boundary:} Hence every feature of positive curvature along 
$\boundary_pD$ is a generalised corner of a square.  By shuffling --- see 
Remark~\ref{defn:pushing_to_boundary} --- we can assume that these are exposed squares, i.e. each generalised corner of a 
square along $P$ is actually a length--$2$ subpath of the boundary path of a square.

\textbf{Square-homotoping $P$:}  Since there are at least two of these squares, at least one, denoted $s$, 
satisfies the 
following: $\boundary_ps=Ief$, where $ef$ is a subpath of 
$\boundary_pD$ and the vertex in which $e,f$ intersect is not $p(\tilde y)$.  Hence we can perform a square homotopy, 
removing $s$ from $D$, to obtain a new diagram $D'$ in which $ef$ is replaced by $I$ in the boundary path.  Note that 
$|\boundary_pD'|=|P|$, and $p(\tilde y)\in\boundary_pD'$.

Thus we can replace $\widetilde A$ by a $\langle \tilde g\rangle$--invariant geodesic $\widetilde A'$ as follows: lift $ef$ 
to a path $\tilde e\tilde f$ in $\widetilde P$, lift $s$ to a square $\tilde s$ meeting $\widetilde A$ in the subpath 
$\tilde e\tilde f$, and homotop $\widetilde A$ across $\tilde s$.  Do the same at each $\langle \tilde g\rangle$--translate 
of $\tilde s$.  Let $\widetilde P'$ be the subpath of $\widetilde A'$ from $\tilde y$ to $\tilde y'$.

Note that $\widetilde P'$ projects to a closed path $P$ in $\widetilde X^*$ bounding a proper subdiagram of 
$D$.  Repeating finitely many times, we find a subdiagram of $D$ that has at least two features of positive curvature (it 
has at least three, or is a ladder), none of which is a spur and at most one of which is a generalised corner (whose 
boundary path contains $y$).  So, there is some positively-curved shell in this subdiagram with outer path $O$.  By short 
inner paths (Theorem~\ref{thm:short_inner_paths}), $|O|>\|Y_i\|/2$.

Hence $\widetilde A$ is square-homotopic in $\widetilde X$ 
to a $\langle \tilde g\rangle$--invariant combinatorial geodesic $\widetilde B$ such that $\widetilde B$ 
contains a path $\widetilde O$ such that $\widetilde O$ lies in some $\widetilde Y_i$ and satisfies 
$|\widetilde O|>\|Y_i\|/2$.  Moreover, $\tilde y,\tilde 
y'\in\widetilde B$, and $\widetilde O$ is a subpath of the subpath $\widetilde Q$ of $\widetilde B$ joining $\tilde y$ to 
$\tilde y'$.  

Hence there exists a subpath $\widetilde Q_1$ of $\widetilde B$ with the following properties:
\begin{itemize}
     \item the path $\widetilde Q_1$ contains a subpath $\widetilde O_1$ that lies in $\widetilde Y_i$ and is maximal with 
that property;
\item we have $|\widetilde O_1|>\|Y_i\|/2$;
\item either $\widetilde O_1$ is unbounded, or $\widetilde Q_1$ starts and ends on $\widetilde A$.
\end{itemize}

We consider two cases.

\textbf{$\widetilde O_1$ bounded:}  First suppose that $\widetilde O_1$ is bounded and let $\widetilde P_1$ be the subpath 
of $\widetilde A$ subtended by the endpoints of $\widetilde Q_1$.

Consider the geodesic bigon $\widetilde Q_1\widetilde P^{-1}_1$ in $\widetilde X$.  Let $E\to\widetilde X$ be a minimal-area 
disc diagram with $\boundary_pE=\widetilde Q_1\widetilde P^{-1}_1$.  Moreover, since $\widetilde Y_i$ is convex, we make our 
choice allowing the geodesic $\widetilde O_1$ to vary, fixing the endpoints; any such geodesic lies in $\widetilde Y_i$.  In 
particular, if $E$ is chosen to be of minimal area among all disc diagrams with the given boundary path (with $\widetilde 
O_1$ allowed to vary as above), then no two dual curves emanating from $\widetilde O_1$ can cross.  

If $\widetilde Q_1=\widetilde O_1$, then $\widetilde P_1$ lies in $\widetilde Y_i$, and, since $|\widetilde P_1|=|\widetilde 
Q_1|>\|Y_i\|/2$, this contradicts our hypotheses.    

Now write $\widetilde Q_1=U\widetilde O_1V$, with at least one of $U,V$ a nontrivial path.  We now allow $U,V$ to vary, 
fixing their endpoints, and assume that $E$ had minimal area over all of these choices.  Hence no two dual curves emanating 
from $U$ can cross, and the same is true of $V$.

Hence consider $1$--cubes $r,s$ immediately preceding and succeeding $\widetilde O_1$ in $\widetilde Q_1$.  At least one of 
$r$ or $s$ exists; assume it is $r$.  If the hyperplane $H_r$ dual to $r$ crosses $\widetilde Y_i$, then convexity of 
$\widetilde Y_i$ implies $r\subseteq\widetilde Y_i$, contradicting maximality of $\widetilde O_1$.  (Indeed, considering the 
gate map to $\widetilde O_1$ shows that $r$ must project to a $1$--cube parallel to $r$, since $H_r$ crosses $\widetilde 
Y_i$.  On the other hand, the initial point of $r$ is already in $\widetilde Y_i$, and thus sent to itself.  Thus $r$ is 
equal to its image under the gate map, so $r\subseteq\widetilde Y_i$.)

Hence, $H_r$ does not cross $\widetilde Y_i$. 

Let $K_r$ be the dual curve in $E$ dual to $r$.  Let $K$ be a dual curve emanating from $\widetilde O_1$.  Let $L$ be a dual 
curve crossing $K$.  Then $L$ cannot cross $\widetilde O_1$.  If $L$ crosses $U$, then either $L=K_r$, or $L$ is separated 
in $E$ from $\widetilde O_1$ by $K_r$, so $K$ must cross $K_r$.  Hence, if $K$ is a dual curve emanating from $\widetilde 
O_1$ and having positive length, then $K$ crosses $K_r$ or (by a symmetric argument) the dual curve $K_s$ emanating from the 
$1$--cube of $V$ following $\widetilde O_1$.

Now, $H_r$ doesn't cross $\widetilde Y_i$, so each dual curve starting at $\widetilde O_1$ and crossing $K_r$ contributes to 
the length of a wall-piece in $\widetilde Y_i$.  Hence there is a constant $M$, depending only on the uniform 
$C''(\frac{1}{144})$ condition, such that $K_r$ can cross at most $M$ dual curves.  Similarly, $K_s$ crosses at 
most $M$ dual 
curves.  So, at most $2M$ of the dual curves $K$ emanating from $\widetilde O_1$ have positive length.  

Hence $\widetilde O_1,\widetilde P_1$ have a common subpath of length greater than $\|Y_i\|/2-2M$.  Since $M<\|Y_i\|/144$, 
we conclude that $\widetilde P_1$, and hence $\widetilde A$, has a subpath that lies in $\widetilde Y_i$ and has length more 
than $35\|Y_i\|/72$, a contradiction.

\textbf{$\widetilde O_1$ unbounded:}  The remaining case is where $\widetilde O_1$ is unbounded.  In other words, 
$\widetilde Y_i$ contains a sub-ray of the axis 
$\widetilde B$ of $\langle\tilde g\rangle$.  This implies that some power of $\tilde g$ stabilises 
$\widetilde Y_i$, so $\widetilde B\subseteq\widetilde Y_i$.

Let $\widetilde A_n$ be the subpath of $\widetilde A$ between $\tilde y$ and $\tilde g^n\tilde y$, let $U,U'$ be geodesics 
joining $\tilde y,\tilde g^n\tilde y$ to closest $0$--cubes of $\widetilde Y_i$, and let $V$ be a geodesic of $\widetilde 
Y_i$ joining the terminal points of $U,U'$.  Let $F\to\widetilde X$ be a minimal-area disc diagram bounded by the paths 
$\widetilde A_n, U,U',V$.  Then, by allowing $U,U',V$ to vary, fixing their endpoints, and assuming that $D$ is minimal over 
all such choices, we have that no two dual curves emanating from $V$ can cross.  Now, the number of dual curves intersecting 
$U,U'$ is bounded independently of $n$, since $\widetilde A$ lies in a uniform neighbourhood of $\widetilde Y_i$.  Hence, 
when $n$ is sufficiently large, we see that either some dual curve travelling from $\widetilde A_n$ to $V$ has length $0$, 
or some hyperplane $H$ crosses $U$ and $U'$.

In the former case, $\widetilde A$ contains a point of $\widetilde Y_i$.  It follows by convexity of $\widetilde Y_i$ 
that $\widetilde A\subseteq\widetilde Y_i$, contradicting asystolicity.    If the former case does not hold for any $n$, 
then there is a hyperplane $H$ separating $\widetilde A$ from $\widetilde Y_i$.  The hyperplane $H$ does not cross 
$\widetilde Y_i$, but every hyperplane crossing $\widetilde A$ crosses $\widetilde Y_i$ and $H$.  Hence $\widetilde Y_i$ 
contains arbitrarily large wall-pieces, a contradiction.

\textbf{Conclusion:}  We have shown that, if $\widetilde A$ contains no subpath of any $\widetilde Y_i$ of length more than 
$35\|Y_i\|/72$, then $\widetilde A$ is an embeddable axis and so $\tilde g$ is an embeddable lift of $g$.  In particular, $g$ has infinite order.
\end{proof}

\subsection{Finding fast loxodromics on $\hyp$}  We are now ready for our main technical lemma, which explains 
how to identify when an element of $\pi_1X$ that is loxodromic 
on the contact graph survives in $\pi_1X^*$ as an element that is loxodromic on $\hyp$.

\begin{lem}[Asystolic loxodromics 
persist]\label{lem:loxodromics_persist_2}
Let $g\in\pi_1X^*$.  Suppose that $\tilde g$ is a lift of $g$ admitting a $\frac{35}{72}$--asystolic axis.  Suppose that $\tilde g$ is 
loxodromic on $\widetilde{\hyp}$.  Then $g$ is loxodromic on $\hyp$.  

%In particular, if $\tilde g$ is a $\frac{35}{72}$--asystolic lift of $g$, then $g$ is loxodromic on $\hyp$.

Now suppose that $\tilde g$ is a lift of $g$ with a $\frac{17}{36}$--asystolic axis, and that $\tilde g$ is loxodromic on $\widetilde{\hyp}$.  Then $g$ acts on $\hyp$ as a loxodromic WPD 
element.  %This holds in particular if $g$ has a %$\frac{17}{36}$--asystolic lift.
\end{lem}

\begin{outline}
Since the proof of Lemma \ref{lem:loxodromics_persist_2} is rather long, we outline the steps.
\end{outline}
\begin{itemize}
     \item Using asystolicity and Lemma~\ref{lem:characterising_a_shortest}, $g$ 
has an embedded axis $A$ in $\cay(X^*)$ that is the image of the asystolic axis $\widetilde A$ of $\tilde g$.
     \item Fixing $a\in A$, we consider the subpath $A_n$ of $A$ from $a$ to $g^na$.  We also consider the following path 
$Q$ from $a$ to $g^na$: choose a sequence $C_0,\ldots,C_N$ of relators or hyperplane-carriers in $\widetilde X^*$ so that 
successive ones intersect, the first and last contain $a,g^na$, and $N$ is as small as possible.  We choose 
a path $Q\to\widetilde X^*$ from $a$ to $g^na$ that is a concatenation of paths in the 
various $C_i$.

    \item To prove that $g$ is loxodromic, we need to show that $N$ grows linearly in $n$.  This is Claim~\ref{claim:g_loxo}. 
 The idea is that for suitable disc diagrams $D\to\widetilde X^*$ bounded by $A_n$ and $Q$, there cannot be any cone-cells 
in $D$.  Hence $D$ factors through $\cay(X^*)$ and hence lifts to $\widetilde X$.  The diagram $D\to\widetilde X$ is 
bounded by a subpath of an axis $\widetilde A$ of $\tilde g$, and a lift $\widetilde Q$ of $Q$, which is a concatenation of 
lifts of the $\alpha_i$.  So, there is a sequence of vertices in $\Gamma$ joining the endpoints of $\widetilde A_n$ and 
having $n$ vertices.  This means that a linear function of $N$ bounds from above the length of $\widetilde A_n$, as 
measured in $\widehat{\hyp}$, which in turn grows linearly in $n$ by Lemma~\ref{lem:loxodromic_persist_1}. Hence $N$ must 
grow linearly in $n$.

\item Thus, it suffices to show that $Q$ can be chosen so that $D$ lifts.  To do this, we choose $C_0,\ldots,C_N$, the path 
$Q$, and $D$ so that the complexity of $D$ is minimal over all such choices.  We decompose $D$ into three subdiagrams 
$D',F,D''_1$ as follows.  First, we perform square homotopies in $D$ to replace $A_n$ by a path $P$, of the same length 
(hence also lifting to a geodesic in $\widetilde X$).  The part between $A_n$ and $P$ is $D'$.  In 
Claims~\ref{claim:D''_ij},\ref{claim:D''_no_shell},\ref{claim:D''_no_shell_2}, we show that the remaining subdiagram has no 
positively-curved shells.  Then we square-homotop $Q$ to a path $Q_1$; the subdiagram between these paths is a 
square diagram $F$, and the 
remaining part is $D''_1$.

\item We then invoke the ladder theorem, Theorem~\ref{thm:ladder_thm}, to show that $D''_1$ is a ladder, in Claim~\ref{claim:extracting_a_ladder}.  The 
idea is that minimality of $N$ means that any positively-curved shell in $D''_1$ is the concatenation of at most $5$ pieces 
between the shell and the various relators/hyperplanes $C_i$.  This contradicts our small-cancellation assumption, via 
the short inner paths property. 

\item Using asystolicity --- via the auxiliary Claim~\ref{claim:bounded_coarse_intersection_with_relators} --- we promote 
the statement that $D_1''$ is a ladder to the statement that $D_1''$ is a square diagram.  The idea is that any cone-cell in 
$D_1''$ has boundary path consisting of at most $5$ pieces on the $Q$ side (as above), $2$ pieces on the incident 
pseudorectangles, and a subpath on the $P$ side that must be short (this uses asystolicity).  Since $D_1'',F,D'$ are all 
square diagrams, so is $D$, which is what we needed.

\item It remains to show that $g$ is fast.  Here the argument is similar but easier: we form $D$ as above, except that $Q$ 
is replaced by a $\cay(\widetilde X^*)$--geodesic $S$ from $a$ to $g^na$.  The diagram $D$ is now bounded by $A_n$ and $S$.  
We again square-homotop $A_n$ and $S$ to geodesics, decomposing $D$ as the union of two square diagrams and a third 
``central'' diagram.  Much as before, the central diagram has to be a ladder.  This is seen more easily than in 
Claim~\ref{claim:extracting_a_ladder}, because $S$ is a geodesic and Theorem~\ref{thm:short_inner_paths} takes care of 
positively-curved shells for free.  Once we know this diagram is a ladder, two applications of asystolicity (again using 
Claim~\ref{claim:bounded_coarse_intersection_with_relators}) show that $g$ is fast.

\item We conclude using Lemma~\ref{lem:acyl_version_2}.
\end{itemize}

\begin{proof}[Proof of Lemma~\ref{lem:loxodromics_persist_2}]
The proof has several parts.  Since $\frac{17}{36}$--asystolicity implies $\frac{35}{72}$--asystolicity, we will assume   
$\frac{35}{72}$--asystolicity for the purpose of showing that $g$ is loxodromic, and $\frac{17}{36}$--asystolicity only for 
the purpose of showing that $g$ is fast (recall Definition~\ref{defn:fast_loxodromic}).

\textbf{Embeddability:}  Fix a $\frac{35}{72}$--asystolic axis $\widetilde A$ for $\tilde g$.  Lemma~\ref{lem:characterising_a_shortest} implies that $\widetilde A$ is an embeddable axis, so $\tilde 
g$ is an embeddable lift and hence $g$ has infinite order.  Hence, because each relator is compact, no 
lift of $g$ has a positive power that stabilises an elevation of a relator.  Since $\tilde g$ is loxodromic on 
$\widetilde{\hyp}$, Lemma~\ref{lem:loxodromic_persist_1} ensures that $\tilde g$ (which is necessarily rank one and has 
no power stabilising a hyperplane) is loxodromic on $\widehat{\hyp}$.  

\textbf{Bounded coarse intersections with hyperplanes:} Since $\tilde g$ is loxodromic on 
$\widetilde{\hyp}$, there exists $\mathfrak p_g<\infty$ such that for all hyperplanes $\widetilde H$, we have 
$\diam(\gate_{\neb{(\widetilde H)}}(\widetilde A))\leqslant \mathfrak p_g$.  In other words, at most $\mathfrak p_g$ of the hyperplanes crossing $\widetilde A$ can cross 
$\widetilde H$.

\textbf{Bounded coarse intersections with elevations of relators:}  Let $M$ be 
the upper bound on diameters of pieces, i.e. $M=\frac{1}{144}\inf_i\|Y_i\|$.  We need the following claim:

\begin{claim}\label{claim:bounded_coarse_intersection_with_relators}
There exists $\mathfrak q_g<\infty$ such that, for all subcomplexes $\widetilde Y_i$ that are lifts of 
relators to $\widetilde X$, $$\diam(\gate_{\widetilde Y_i}(\widetilde 
A))<\mathfrak q_g,$$ and $$\diam(\gate_{\widetilde Y_i}(\widetilde 
A))<\frac{1}{2}\|Y_i\|.$$     

Moreover, under the additional $\frac{17}{36}$--asystolicity assumption, $\diam(\gate_{\widetilde 
Y_i}(\widetilde A))<\frac{35}{72}\|Y_i\|.$
\end{claim}
\renewcommand{\qedsymbol}{$\blacksquare$}
\begin{proof}[Proof of Claim~\ref{claim:bounded_coarse_intersection_with_relators}]
Let $\tau_{\tilde g}$ be the combinatorial translation length of $\tilde g$.  Fix a vertex $\tilde a\in\widetilde A$.

\emph{Absolute bound on subpaths of $\widetilde A$ in relators:}  First, we will show that there exists $\mathfrak q_g'$ 
such that $|\widetilde P|\leqslant\mathfrak q'_g$ whenever $\widetilde P$ is a subpath of $\widetilde A$ that lies in some 
$\widetilde Y_i$.  

Suppose not.  Then for all $N\in\naturals$, there 
exists $\widetilde Y^N$ (an elevation of a relator $Y_{i_N}$) and a subpath $\widetilde P_N$ of $\widetilde A$ such that 
$\widetilde P_N$ lies in $\widetilde Y^N$ and has length more than $N$.  By applying powers of $\tilde g$, we can assume 
that $\widetilde P_N$ joins $\tilde a$ to $\tilde g^{k_N}\tilde a$, where $k_N\geqslant 
N/\tau_{\tilde g}-1$.

Hence the subpath $\widetilde Q_N$ of $\widetilde P_N$ joining $\tilde g\tilde a$ to $\tilde g^{k_N-1}\tilde a$ lies in 
$\widetilde Y^N\cap\tilde g\widetilde Y^N$.  Thus either $\widetilde Y^N=\tilde g\widetilde Y^N$, or $\widetilde Q_N\to Y^N$ 
is a piece. As mentioned above, $\tilde g$ cannot stabilise an elevation of a relator.  So $\widetilde Q_N$ is a piece of length at least 
$N-2\tau_g$.  For $N>M+2\tau_g$, this is a contradiction.  We conclude 
that there must exist $\mathfrak q'_g$ with the claimed property.

\emph{Relative bound on subpaths of $\widetilde A$ in relators:}  Second, just by asystolicity, if $\widetilde P$ is a 
subpath of $\widetilde A$ lying in some $\widetilde Y_i$, we have $|\widetilde P|<\frac{35}{72}\|Y_i\|$.  (Or $|\widetilde 
P|<\frac{17}{36}\|Y_i\|$ under the stronger of the two asystolicity assumptions.)

\emph{Absolute bound on projection of $\widetilde A$ to relators:}  Fix $\widetilde Y_i$.  Let $R$ be a geodesic in 
$\gate_{\widetilde Y_i}(\widetilde A)$; so, the hyperplanes intersecting $R$ all intersect $\widetilde A$ and $\widetilde 
Y_i$.  Initially, we choose $R$ to be the image of some geodesic $R'$ in $\widetilde A$ joining two $0$--cubes $y,y'$.  So, 
$R$ joins $\gate_{\widetilde Y_i}(y),\gate_{\widetilde Y_i}(y')$. 

Next, let $\beta,\beta'$ be geodesics joining $y,\gate_{\widetilde Y_i}(y)$ and 
$y',\gate_{\widetilde Y_i}(y')$ respectively.  Let $D\to\widetilde X$ be a 
minimal-area disc diagram bounded by $\beta,\beta',R,R'$.  We allow 
$\beta,\beta',R$ to vary among geodesics with the given endpoints; varying such 
a geodesic does not change the hyperplanes it crosses.  We make all such 
choices so that, among them, the resulting $D$ has minimal area.  (We emphasise 
that we do not allow $R'$ to vary, since we need it to be a subpath of 
$\widetilde A$.)  See Figure~\ref{fig:cat0_diagram}.

\begin{figure}[h]
\begin{overpic}[width=0.6\textwidth]{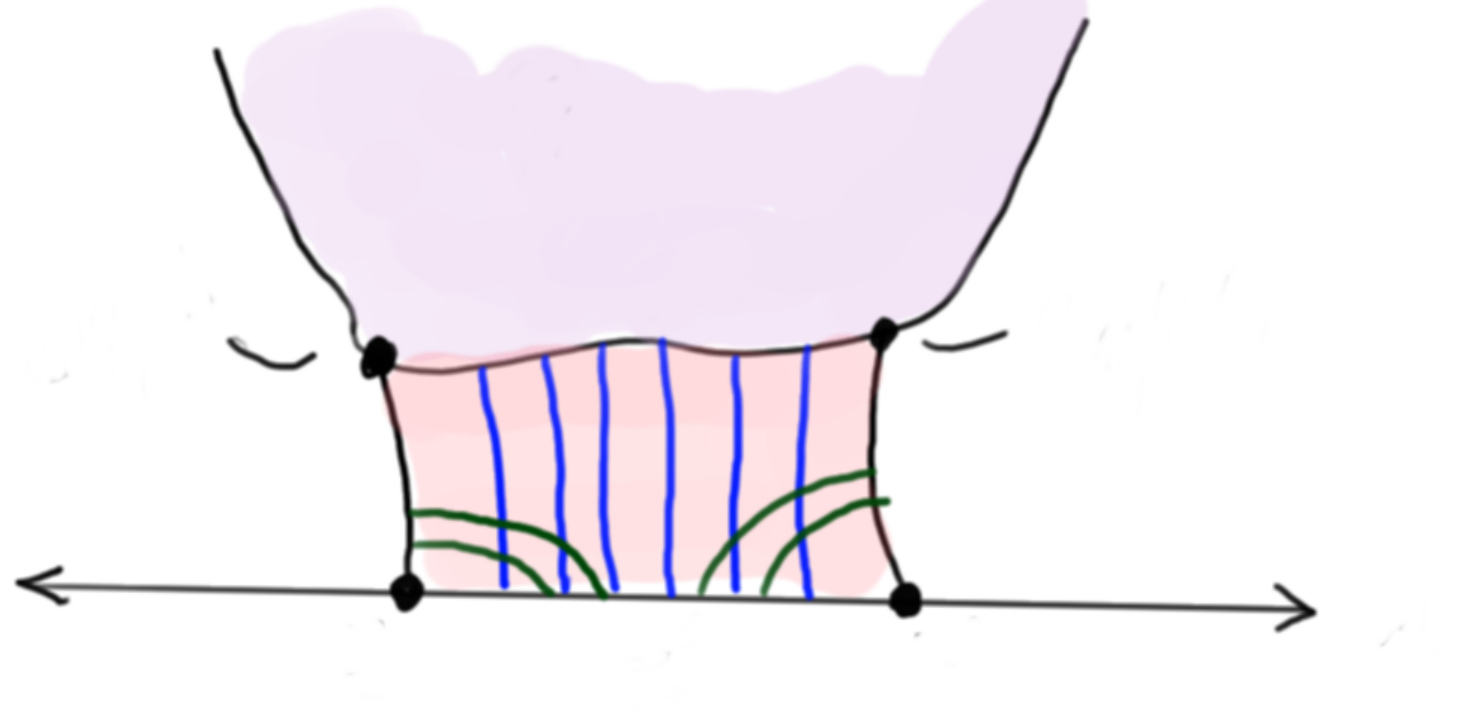}
     \put(46,26){$R$}
     \put(45,4){$R'$}
     \put(26,4){$y$}
     \put(63,3){$y'$}
     \put(45,40){$\widetilde Y_i$}
     \put(8,27){$\gate_{\widetilde Y_i}(y)$}
     \put(70,27){$\gate_{\widetilde Y_i}(y')$}
     \put(95,5){$\widetilde A$}
\end{overpic}

\caption{The diagram in $\widetilde X$ bounded by $R,R',\beta,\beta'$, along with some allowable dual 
curves.  At least $|R'|-2M$ of the vertical dual curves do not cross any other dual curve, and thus 
have length $0$.  Hence, for each $\widetilde Y_i$, at most $\|Y_i\|/2$ hyperplanes cross both 
$\widetilde A$ and $\widetilde Y_i$ in $\widetilde X$.}\label{fig:cat0_diagram}
\end{figure}

Let $K,K'$ be dual curves in $D$ emanating from $\beta$.  Then minimality of 
area ensures that $K,K'$ do not cross.  The same holds with $\beta$ replaced 
by $\beta'$ or $R$.  Next, let $K$ be a dual curve emanating from $R$.  The 
hyperplane to which $K$ maps must not separate $y,\gate_{\widetilde Y_i}(y)$, 
since it crosses $\widetilde Y_i$, so $K$ cannot cross $\beta$.  Similarly, $K$ 
cannot cross $\beta'$.  So, every dual curve emanating from $R$ terminates on 
$R'$.

Hence every dual curve $K$ travels from $R$ to $R'$ or from $R'$ to $\beta$ or 
$\beta'$, or from $\beta$ to $\beta'$.  

If some dual curve $K$ crosses $\beta$ and $\beta'$, then $K$ maps to a hyperplane $H$ that does not cross $\widetilde Y_i$ 
(since it separates $y,\gate_{\widetilde Y_i}(y)$) and crosses every hyperplane crossing both $R,R'$.  Hence $R'$ is a 
wall-piece, and we get $|R'|\leqslant M$.  So, either we are done, or no such $K$ exists.  Assume the latter.

Now, if $K$ travels from $\beta$ to 
$R'$, then $K$ maps to a hyperplane $H$ not crossing $\widetilde Y_i$ (since $H$ 
separates $\gate_{\widetilde Y_i}(y)$ from $y$), so it yields a wall-piece in 
$\widetilde Y_i$ and hence crosses at most $M$ of the dual curves traveling 
from $R$ to $R'$.  The same holds for dual curves traveling from $\beta'$ to 
$R'$  or $\beta$ to $\beta'$.  Hence there are at least $|R'|-2M$ dual curves that travel from $R$ to 
$R'$ and do not cross any other dual curves.  These dual curves thus have 
length $0$, so $R,R'$ have a common subpath of length at least $|R'|-2M$.  The preceding discussion shows that  $|R'|-2M < 
\mathfrak q'_g$, so 
$|R'|<\mathfrak q'_g+2M$.  Taking $\mathfrak q_g=\mathfrak q'_g+2M$ therefore suffices.

\emph{Relative bound:}  The preceding  argument also shows that $|R'|< \frac{35}{72}\|Y_i\|+2M$ (under 
$\frac{35}{72}$--asystolicity) or $|R'|<\frac{17}{36}\|Y_i\|+2M$ (under $\frac{17}{36}$--asystolicity).  Hence, 
under the weaker assumption, we get $|R'|< \|Y_i\|/2$, since $144M<\inf_j\|Y_j\|$.  Under the stronger assumption, we get 
$|R'|<\frac{35}{72}\|Y_i\|$.  This completes the proof of the claim.  \footnote{Since there is no uniform bound on the 
systoles of the relators, the second part of Claim~\ref{claim:bounded_coarse_intersection_with_relators} is insufficient to 
give the bound $\mathfrak q_g$, because $\widetilde A$ can pass through infinitely many orbits of elevations of 
relators.  We use the bound $\|Y_i\|/2$ to prove that $g$ is loxodromic, and the bound $\frac{35}{72}\|Y_i\|$ for 
proving that $g$ is fast.}
\end{proof}

Now we resume the proof of the lemma.  

\textbf{The path $A\to\cay(X^*)$:}  Let $\tilde a\in\widetilde A$ be a 
$0$--cube.   Let $a=p(\tilde a)$, so that 
$g^na=p(\tilde g^n\tilde a)$ for each $n>0$.  Let $A=p\circ\widetilde A$.  Since 
$\widetilde A$ is embeddable, the path $A$ is an embedded bi-infinite 
combinatorial path in $\cay(X^*)$.

\textbf{Paths in $\hyp$:}    Fix $n>0$. Choose a sequence $C_0,\ldots,C_N$ satisfying:
\begin{itemize}
 \item for each $j\leqslant N$, either $C_j$ is the carrier of a hyperplane in $\widetilde X^*$ or $C_j\to\widetilde X^*$ is 
a lift of some $Y_i\to X$ (abusing notation, we use the same name for $C_j$ as for its image);
 \item we have $a\in C_0$ and $g^na\in C_N$;
 \item we have $C_j\cap C_{j+1}\neq\emptyset$ for all $j\leqslant N$;
 \item the number $N$ is minimal with the above properties.
\end{itemize}

For each $j$, let $\alpha_i$ be a combinatorial 
 path in $C_j$, chosen so that $Q=\alpha_0\alpha_2\cdots\alpha_N$ is a path from $a$ to $g^na$ with no 
backtracks.  

For proving that $g$ is loxodromic on $\hyp$, our goal will be to prove that $N$ is bounded below by a linear function of 
$n$.

\textbf{Constructing a disc diagram:}  Let $A_n$ be the subpath of $A$ joining 
$a$ to $g^na$.  Let $D\to\widetilde X^*$ be a minimal disc diagram with boundary path 
$A_nQ^{-1}$.  Moreover, choose the $C_j,$ and $Q$, 
subject to the above constraints, so that $D$ has minimal complexity among all 
diagrams constructed in the preceding manner.  See Figure~\ref{fig:loxodromic_diagram}.

\begin{figure}[h]
\begin{overpic}[width=0.75\textwidth]{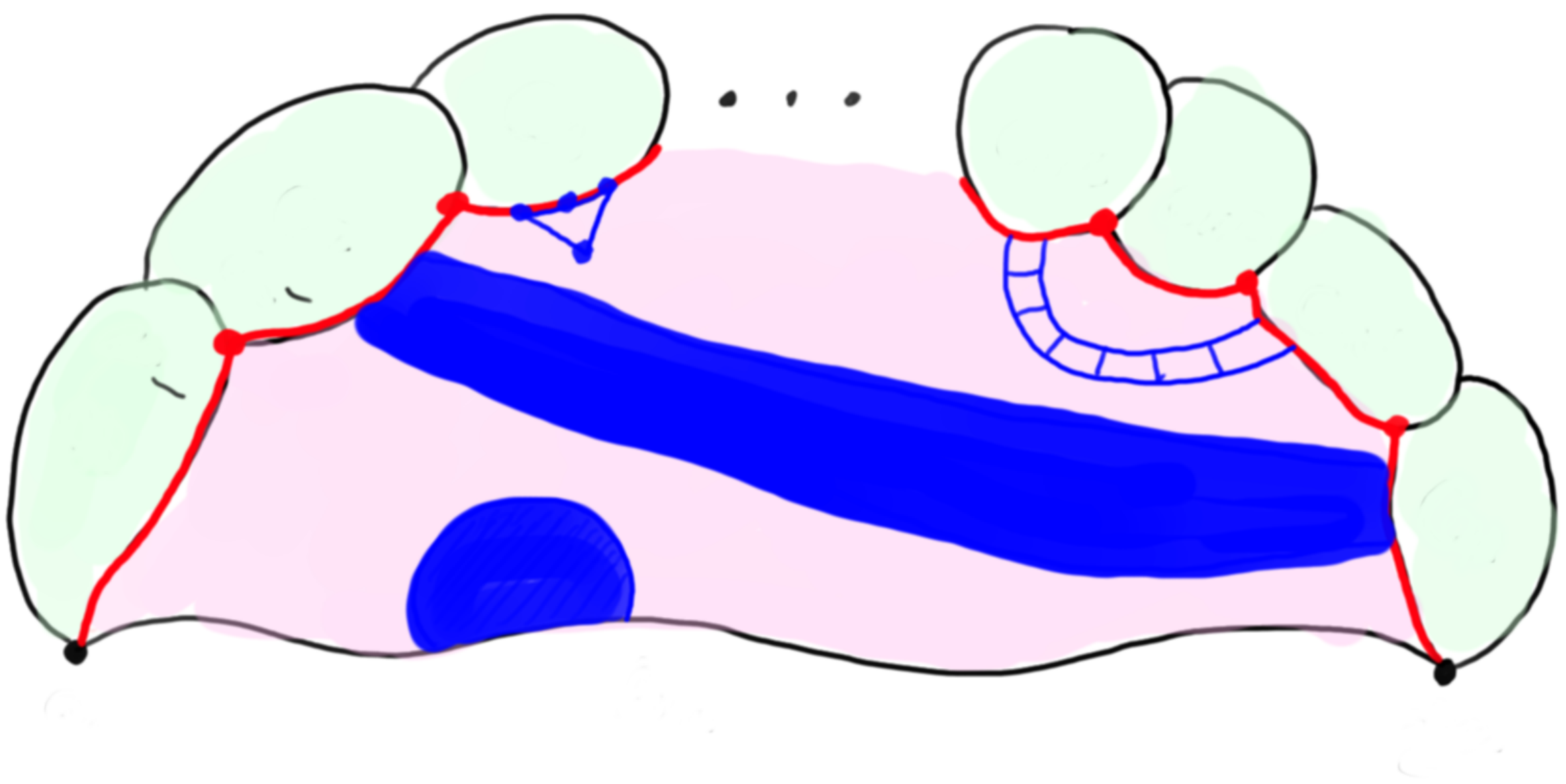}
     \put(5,22){$C_1$}
     \put(20,40){$C_2$}
     \put(35,44){$C_3$}
     \put(83,30){$C_{N-1}$}
     \put(74,35){$C_{N-2}$}
     \put(65,42){$C_{N-3}$}
     \put(93,15){$C_N$}
     \put(4,5){$a$}
     \put(93,5){$g^na$}
     \put(50,5){$A_n$}
     \put(7,27){$\alpha_1$}
     \put(15,32){$\alpha_2$}
\end{overpic}

\caption{The diagram $D$ between $A_n$ and $Q$, along with the surrounding relators and 
carriers $C_1,\ldots,C_N$ along $Q$.  Various illegal features are shown.  The exposed square 
along $\alpha_3$ belongs to $C_3$ by convexity, so we could have homotoped $\alpha_3$ across 
the square to obtain a lower-complexity diagram.  The dual curve carrier traveling from $\alpha_{N-3}$ to 
$\alpha_{N-1}$ maps to a hyperplane carrier intersecting $C_{N-3}$ and $C_{N-1}$, so we could 
have replaced $C_{N-2}$ by this carrier to yield a lower-complexity diagram.  The cone-cell 
intersecting $\alpha_2,\alpha_N$ contradicts minimality of $N$.  If the cone-cell with outer 
path on $A_n$ is a positively curved shell, then asystolicity of $g$ is contradicted.}\label{fig:loxodromic_diagram}
\end{figure}

Since $Q$ has no backtracks and $A_n$ is embedded, $D$ has no spurs along $Q$ or $A_n$.  (So there are at most 
two spurs.)

\textbf{The first square homotopy:}  If $D$ contains an exposed square (or 
generalised corner of a square) corresponding to a length--$2$ subpath $ef$ 
(with $e,f$ single $1$--cubes) that lies on $A_n$, then we can perform a square 
homotopy (shuffling first if necessary, see Remark~\ref{defn:pushing_to_boundary}) to replace $A_n$ by a square-homotopic 
embedded path $P$ with the same length and endpoints.  

Thus: $D=D'\cup_PD''$, where $D'$ is a 
square diagram bounded by the paths $P,A_n$ and $D''$ is a diagram bounded by 
$P$ and $Q$, and there are no generalised corners or spurs along 
$P$.  

\textbf{Bounding subpaths of $P$ in relators and carriers:}  Since $D'$ is a square diagram, it lifts to a square diagram 
$D'\to\widetilde X$ bounded by the subpath of $\widetilde A$ joining $\tilde a$ 
to $\tilde g^n\tilde a$ (since $\widetilde A$ is a lift of $A$) and a lift 
$\widetilde P$ of the embedded path $P$.  

The path $\widetilde A$ is a geodesic, so every dual curve in $D'$ starting on 
$A_n$ ends on $P$.  Now, since $|P|=|A_n|$, we also have that $\widetilde P$ is 
a geodesic of $\widetilde X$, so dual curves in $D'$ starting on $P$ end on 
$A_n$. Hence, if $P_1$ is a subpath of $P$ lying in a cone-cell $\Upsilon$, we 
have $|P_1|<\|Y_i\|/2$, where $Y_i$ is the relator 
to which $\Upsilon$ maps.  If $P_2$ is a subpath of $P$ lying in a 
hyperplane carrier, $|P_2|\leqslant\mathfrak p_g$.

\textbf{Analysis of $D''$:}  Now we prove some claims about $D''$.

\begin{claim}\label{claim:D''_ij}
Let $K$ be a dual curve or cone-cell in $D''$.  Suppose that $K$ 
intersects $\alpha_i,\alpha_j$.  Then $|i-j|\leqslant1$.
\end{claim}

\begin{proof}[Proof of Claim~\ref{claim:D''_ij}]
Minimality of $N$ implies that $|i-j|\leqslant 
2$, for otherwise the hyperplane or relator 
to which $K$ maps would provide a shortcut between $C_i$ and $C_j$, 
contradicting minimality of $N$.  

Now suppose that $i-j=2$.  Let $Y$ be the relator or hyperplane to which $K$ maps.  By replacing $C_{j+1}$ by $Y$, replacing 
$\alpha_i,\alpha_j$ by appropriate subpaths, and replacing $\alpha_{j+1}$ by part of the boundary path of $K$, we have found 
a new choice of $Q$ so that $Q,A_n$ bound a proper subdiagram of $D$ from which the subdiagram $K$ has been removed.  
Chopping off any spurs doesn't increase the complexity.  So we have a new choice of $Q$ leading to a lower-complexity 
diagram, contradicting minimality of $D$.  Hence, $|i-j|\leqslant 1$.
\end{proof}

\begin{claim}[Ruling out shells in $D''$, part I]\label{claim:D''_no_shell}
The diagram $D''$ contains no  positively-curved shell whose outer path is a subpath of $P$.
\end{claim}

\begin{proof}[Proof of Claim~\ref{claim:D''_no_shell}]
Suppose that $K$ is a positively-curved shell in $D''$ with boundary path $OI$, where the outer path $O$ is a 
subpath of $P$.  Let $Y$ be the relator to which $K$ maps, and let $I'$ be a shortest path in $Y$ that is square-homotopic 
in $Y$ rel endpoints to $I$.  Theorem~\ref{thm:short_inner_paths} implies $|I'|<|O|$.  By 
Claim~\ref{claim:bounded_coarse_intersection_with_relators} and the fact that $P$ and $A_n$ cross the same hyperplanes, 
$|O|<\|Y\|/2$, so $|I'O|<Y$.  Hence $I'O$, and thus $IO$, is inessential in $Y$ and thus $K$ can be replaced by a square 
diagram, contradicting minimal complexity of $D$.  Thus there are no positively-curved shells along $P$.
\end{proof}

Since $P$ is embedded, there are also no spurs, and there are no generalised corners in $D''$ along $P$.  
Hence $D''$, by Claim~\ref{claim:D''_no_shell}, $D''$ has no feature of positive curvature whose outer path is 
a subpath of $P$.

\begin{claim}[Ruling out shells in $D''$, part II]\label{claim:D''_no_shell_2}
     The diagram $D''$ contains no positively-curved shell whose outer path is a subpath of $Q$.
\end{claim}

\begin{proof}[Proof of Claim~\ref{claim:D''_no_shell_2}]
Suppose that $K$ is a positively-curved shell with boundary path $OI$, 
where the outer path $O$ is a subpath of $Q$.  Then the relator $Y$ to which $K$ 
maps must intersect $C_i,C_j$ for some $i,j$.  This can only 
happen if $|i-j|\leqslant 1$, by Claim~\ref{claim:D''_ij}.  

By minimal complexity of $D$, it follows that $O$ is square-homotopic to the concatenation of at most $2$ pieces.  Hence, 
letting $O'$ be a shortest path in $Y$ that is square-homotopic rel endpoints to $O$, we have $|O'|\leqslant 2M$.  By 
Theorem~\ref{thm:short_inner_paths}, we have $|O|>72M$, a contradiction. Hence there is no positively-curved shell whose 
outer path is a subpath of $Q$.
\end{proof}

\textbf{The second square homotopy:}  We are working toward an application of 
the ladder theorem --- in view of the diagram trichotomy, we now just need to 
remove generalised corners along $Q$ using square homotopies, as follows.

Let $K$ be a dual curve with one end on $Q$, i.e. one end on some $\alpha_i$.  
We have seen already that $K$ cannot have its other end on $\alpha_j$ for 
$|i-j|\geqslant 2$, because of Claim~\ref{claim:D''_ij}.  

We claim that $K$ cannot end on $\alpha_{i\pm1}$.  Indeed, 
otherwise $D$ would have a subdiagram $E$ bounded by the subpath of 
$\alpha_i\alpha_{i+1}$ (say) subtended by the $1$--cubes dual to $K$, along 
with a path on $\neb(K)$.  Minimal complexity of $D$ and 
Theorems~\ref{thm:Greendlinger},\ref{thm:short_inner_paths} imply that $E$ is a square diagram.  Choosing 
$K$ to be innermost, every dual curve in $E$ travels from $\neb(K)$ to 
$\alpha_i\alpha_{i+1}$.  Now, no two dual curves emanating from $\alpha_i$ 
(or $\alpha_{i+1}$) can cross, because an innermost such pair would give an 
exposed square in $D$ along $\alpha_i$; convexity of $C_i$ would then yield a 
contradiction with minimal complexity by absorbing the square into $C_i$.  By shuffling, we can also assume 
that no two dual curves in $E$ crossing $K$ can cross (if a square has two consecutive edges along the carrier of $K$, we 
can push it out of $E$ without changing $\boundary_pE$).  Hence either $\alpha_i$ has its 
terminal edge equal to the initial edge of $\alpha_{i+1}$ (contradicting that $Q$ has no backtracks), or $\boundary_pE$ has 
an exposed square along $\alpha_i$ (contradicting minimal complexity).  See 
Figure~\ref{fig:no_backtrack}.

\begin{figure}[h]
\begin{overpic}[width=0.6\textwidth]{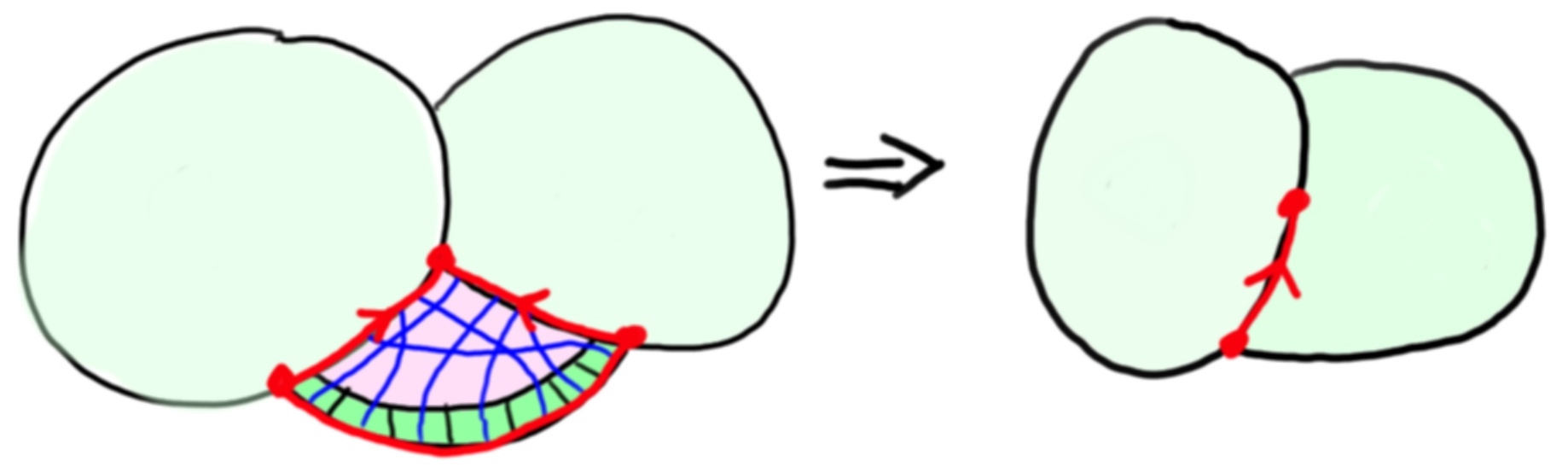}
     \put(15,15){$C_i$}
     \put(40,15){$C_{i+1}$}
     \put(75,15){$C_i$}
     \put(85,15){$C_{i+1}$}
\end{overpic}

\caption{If a dual curve $K$ in $D$ intersects $\alpha_i,\alpha_{i+1}$, then $D$ has a subdiagram $E$ as shown.  (Also shown 
are $C_i,C_{i+1}$).  This must be a square diagram, and convexity of $C_i,C_{i+1}$ then allow us to conclude, from 
minimality of complexity, that the $1$--cubes dual to the first and last points of $K$ coincide.  So no dual curve starts 
and ends on $Q$.}\label{fig:no_backtrack}  
\end{figure}

A simpler version of this argument also shows that $K$ cannot have two ends on $\alpha_i$.  Indeed, if both ends of $K$ are on $\alpha_i$, then $D$ would have a subdiagram $E$ bounded by the 
subpath of $\alpha_i$ between (and including) the $1$--cubes dual to $K$ and a path in $\neb(K)$.  As in the preceding argument, $E$ must be a square diagram, so its boundary path lifts to a closed 
path in $\widetilde X$, whence $\alpha_i$ contains two edges dual to the same hyperplane, contradicting that $\alpha_i$ is a geodesic.    

Let $Q_1$ be an embedded path in $D''$ such that:
\begin{itemize}
 \item $Q_1$ and $Q$ have the same endpoints;
 \item $|Q_1|\leqslant|Q|$;
 \item $Q_1$ has no spurs;
 \item the subdiagram $F$ of $D''$ bounded by $Q$ and $Q_1$ is a square 
diagram, and has as many squares as possible subject to the above constraints.
\end{itemize}

\begin{claim}\label{claim:F}
     Every dual curve in $F$ with an end on $Q$ has an end on $Q_1$, and every dual curve with an end on $Q_1$ has an end on 
$Q$.
\end{claim}

\begin{proof}[Proof of Claim~\ref{claim:F}]
The preceding discussion shows that no dual curve in $D''$ (of which $F$ is a subdiagram) starts and ends on $Q$.  So, since 
$F$ is a square diagram, any dual curve starting on $Q$ ends on $Q_1$.  Since $|Q_1|\leqslant|Q|$, it follows that every 
dual curve in $F$ with one end on $Q_1$ has an end on $Q$.
\end{proof}

\textbf{The subdiagram $D''_1$:}  Let $D''_1$ be the subdiagram of $D''$ bounded by $Q_1$ and $P$.  

\begin{claim}\label{claim:D''_1_spurs_cornsquares}
There are no generalised corners of squares of $D''_1$ lying along $Q_1$.     
\end{claim}

\begin{proof}[Proof of Claim~\ref{claim:D''_1_spurs_cornsquares}]
Suppose that 
$ef$ is a length--$2$ subpath of $Q_1$ corresponding to a generalised corner of 
a square in $D''_1$.  By shuffling, we can assume that $ef$ is a subpath of the 
boundary path of a square $s$ in $D''_1$.  Then $\boundary_ps=efe'f'$, and we 
can replace $ef$ by $e'f'$ (and remove spurs if necessary) in $Q_1$ to obtain a 
new path $Q_1'$, of length at most $|Q_1|$, such that the subdiagram bounded by 
$Q$ and $Q_1'$ has more squares than $F$, a contradiction.  See 
Figure~\ref{fig:big_F}. 
\end{proof}

\begin{figure}[h]
\begin{overpic}[width=0.75\textwidth]{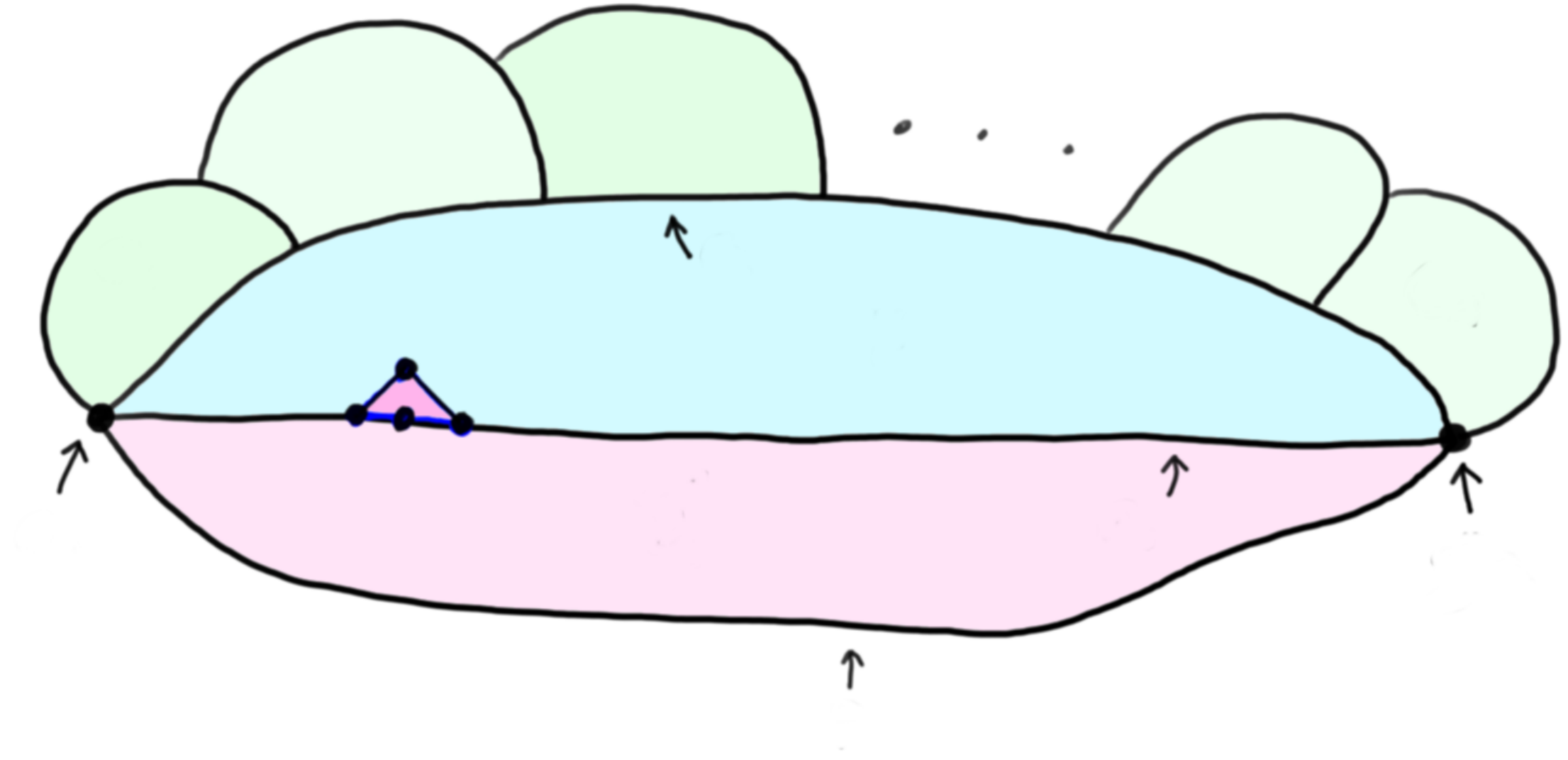}
     \put(2,15){$a$}
     \put(10,30){$C_1$}
     \put(44,30){$Q$}
     \put(90,30){$C_N$}
     \put(93,14){$g^na$}
     \put(55,25){$F$}
     \put(45,15){$D''_1$}
     \put(73,14){$Q_1$}
     \put(53,2){$P$}
\end{overpic}

\caption{The diagram $D''_1$ has no generalised corner of a square along $Q_1$, for otherwise we could enlarge the 
subdiagram $F$ with a square homotopy.}\label{fig:big_F}
\end{figure}

We found a ladder:

\begin{claim}\label{claim:extracting_a_ladder}
The diagram $D''_1$ is a ladder.
\end{claim}

\begin{proof}[Proof of Claim~\ref{claim:extracting_a_ladder}]
We have already seen that $D''$, and hence $D''_1$, has no positively-curved feature (shell, spur, generalised corner)
along $P$.  We have also seen already that $Q_1$ has no spur and $D''_1$ has no generalised corner on $Q_1$.

Suppose that $K$ is a positively-curved shell in $D''_1$ with 
boundary path $OI$, where the outer path $O$ is a subpath of $Q_1$.  Then $O$ is 
also a subpath of $\boundary_pF$, and every dual curve in $F$ emanating from $O$ 
ends on some $\alpha_i$.  Let $i_0,i_1$ be the minimal and maximal values of $i$ 
such that $F$ contains a dual curve $K'$ that travels from $\alpha_i$ to $O$.  
Let $Y$ be the relator to which $K$ maps.  Then there is a sequence 
$C_{i_0},H,Y,H',C_{i_1}$, where $H,H'$ are hyperplanes and consecutive terms 
intersect.  Hence, by minimality of $N$, we have $i_1-i_0\leqslant 4$.  See 
Figure~\ref{fig:shrink}.

\begin{figure}[h]
\begin{overpic}[width=0.75\textwidth]{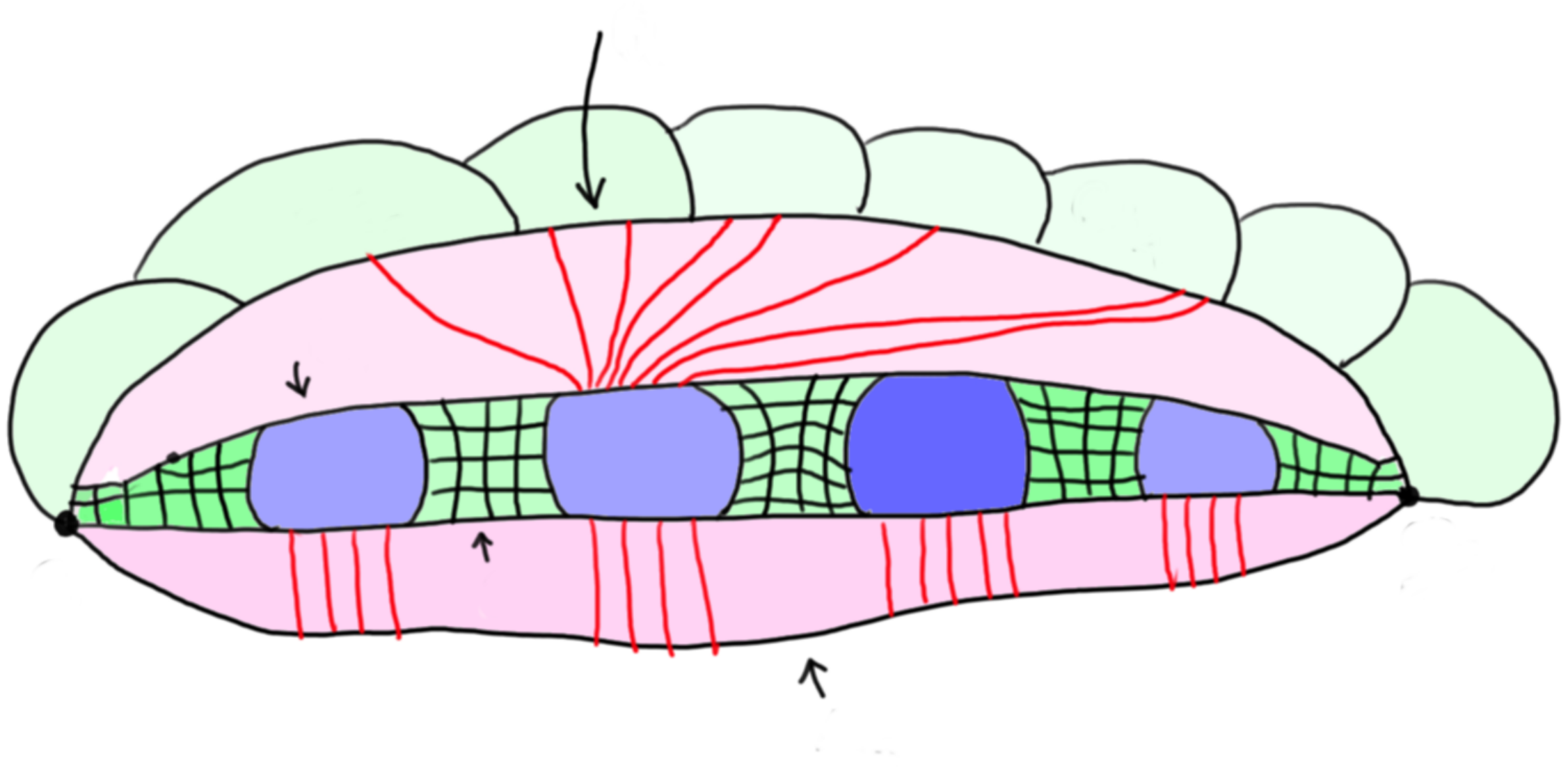}
     \put(3,12){$a$}
     \put(52,3){$A_n$}
     \put(90,14){$g^na$}
     \put(70,35){$C_{i_1}$}
     \put(37,49){$Q$}
     \put(20,35){$C_{i_0}$}
     \put(31,11){$p$}
\end{overpic}

\caption{The diagram $D$ is the union of square diagrams $D',F$ and a ladder $D''_1$.  Dual curves in $F$ travel from $Q$ to 
$Q_1$, and dual curves in $D'$ travel from $P$ to $A_n$.  The cone-cells in $D''_1$ have small projection to $A_n$ because 
of asystolicity of $g$.  They have bounded projection to $Q$ because of minimality of $N$: any dual curve in $F$ 
emanating from a common cone-cell of $D''_1$ must end on $\alpha_i$ for at most one of $5$ values of $i$.  
This is ultimately used to show that $D''_1$, and hence all of $D$, is a square diagram.}\label{fig:shrink}

\end{figure}

We'd now like to write $O$ as a concatenation of at most $5$ pieces.

The square diagram $F$ lifts to a disc diagram $F\to\widetilde X$ with boundary path $\widetilde Q\widetilde Q_1^{-1}$, 
where $\widetilde Q$ is a lift of $Q$ and $\widetilde Q_1$ is a lift of $Q_1$.  Moreover, $\widetilde Q_1$ has a subpath 
$\widetilde O$ lying on an elevation $\widetilde Y$ of $Y$, where $\widetilde O$ is a lift of $O$.  Also, $\widetilde 
Q=\tilde \alpha_0\cdots\tilde\alpha_N$, where each $\tilde\alpha_i$ is a lift of $\alpha_i$ to an elevation $\widetilde C_i$ 
of $C_i$.  (So, either $\widetilde C_i$ is a hyperplane-carrier or is the universal cover of a relator.)

Note that, if two 
dual curves emanating from $\widetilde O$ cross, then by shuffling, we find a square in $F$ with two consecutive boundary 
edges on $\widetilde O$.  This square maps to $\widetilde Y$, and its image in $D$ can thus be absorbed into $K$, 
contradicting minimality of the complexity.  So no two dual curves emanating from $\widetilde O$ can cross in $F$.  

Thus, we can write $\widetilde O=\theta_{i_0}\theta_{i_0+1}\cdots\theta_{i_1}$, where each $\theta_j$ has the property that 
all dual curves starting on $\theta_j$ end on $\tilde\alpha_j$.  For each $j$, if $\widetilde C_j\neq\widetilde Y$ and 
$\widetilde C_j$ is not the carrier of a hyperplane crossing $\widetilde Y$, we therefore have that $\theta_j$ is a path in 
an abstract piece. 

Suppose that $\widetilde C_j=\widetilde Y$ or $\widetilde C_j$ is the carrier of a hyperplane crossing $\widetilde Y$.  Let 
$\kappa$ be a dual curve from $\widetilde O$ to $\tilde\alpha_j$, and let $H$ be a hyperplane crossing the image of $\kappa$ 
in $\widetilde X$.  Then $F$ contains a dual curve $\omega$ mapping to $H$ and crossing $\kappa$.  The dual curve $\omega$ 
must cross all dual curves emanating from $\widetilde O$.  Hence either $H$ crosses $\widetilde Y$, or $\theta_j$ is a 
(wall) piece in $Y$.  

Take  $H$ to be the hyperplane crossing the image of $\kappa$ inside the square $s$ of $F$ containing the first 
edge of $\kappa$.  If $H$ crosses $\widetilde Y$, then convexity of $\widetilde Y$ implies that $s$ is in 
$\widetilde Y$.  Hence we could have absorbed $s$ into $K$ to 
reduce the complexity of $D$.  Thus either $\theta_j$ is a path in an abstract piece, or all such $\kappa$ have length $0$, 
so $\theta_j$ is a subpath of $\alpha_j$.

Let $j_0$ and $j_1$ be the largest and smallest $j$ for which the latter holds.  Let $U^{-1}$ be the subpath of $O^{-1}$ 
from the initial point of $\theta_{j_0}$ to the initial point of $O$, and let $V$ be the subpath of $O$ from the terminal 
point of $\theta_{j_1}$ to the terminal point of $O$.  (Here we regard $\theta_0,\theta_1$ as paths in $\widetilde X^*$.)  
Then $UIV\to\widetilde X^*$ is a path factoring through $Y$ and $D$, and joining $C_{j_0}$ to $C_{j_1}$.    So we can 
replace $C_{j_0+1},\ldots,C_{j_1-1}$ by $Y$ and modify $Q$ to obtain a new diagram which is a proper subdiagram of $D$ (the 
cone-cell $K$ is removed).  This contradicts minimal complexity.  The exception is when $j_0=j_1$ and $\theta_{j_0}$ is a 
trivial path (because then removing $K$ from $D$ would give a singular disc diagram, not a disc diagram).  However, in this 
case, each $\theta_j,j\neq j_0$ is a piece and $O$ is the concatenation of at most 4 pieces.

We have shown that $O$ is a concatenation of at most $5$ paths, each of which is square-homotopic to a piece.  
Thus $O$ is square-homotopic to a path of length at most $5M$.  By Theorem~\ref{thm:short_inner_paths}, 
$\boundary_pK$ is therefore square-homotopic to a path of length at most $10M$, so, by the small-cancellation 
condition, $\boundary_pK$ is not essential in $Y$, and thus $K$ could have been replaced by a square subdiagram, lowering 
complexity.

It then follows from  Theorem~\ref{thm:ladder_thm} that $D_1''$ is a ladder.
\end{proof}

\textbf{The ladder is a square diagram:}  
We now analyse the ladder $D''_1$ more carefully.

\begin{claim}\label{claim:D''_1_no_cone}
$D$ is a square diagram.
\end{claim}

\begin{proof}[Proof of Claim~\ref{claim:D''_1_no_cone}]
Let $\Upsilon$ be a cone-cell of 
$D_1''$.  Then $\boundary_p\Upsilon=e_{Q_1}f_1e_Pf_2$, where $f_1,f_2$ 
are trivial or are subpaths in the incident pseudorectangles, and $e_{Q_1}$ is a maximal subpath of $\boundary_p\Upsilon$ 
lying on $Q_1$, and $e_P$ is a maximal subpath on $P$.  Let $Y_\Upsilon$ be the relator to which $\Upsilon$ maps.  Then 
$|e_{P}|<\|Y_\Upsilon\|/2$, because of the bound on the diameter of projections 
of $\widetilde A$ to 
elevations of relators (Claim~\ref{claim:bounded_coarse_intersection_with_relators}).  On the other hand, 
$e_{Q_1}$ is square-homotopic to a path of length at most $5M$, by the exact same argument as was used in 
Claim~\ref{claim:extracting_a_ladder} to rule out shells with outer path on $Q_1$. Finally, $f_1,f_2$ are either trivial or 
pieces.

Thus, up to square-homotopy, $|\boundary_p\Upsilon|\leqslant 7M+\frac12\|Y_\Upsilon\|$.  Minimal complexity 
of $D$ requires that $|\boundary_p\Upsilon|\geqslant\|Y_\Upsilon\|$, so 
$7M\geqslant\frac12\|Y_\Upsilon\|$, contradicting the uniform $C''(\frac{1}{144})$ 
condition.  Thus $D_1''$ cannot contain any cone-cells.  In other words, $D''_1$ is a square diagram.  Since $D$ consists of 
$D''_1$ together with the square subdiagrams $F$ and $D'$, it follows that $D$ is a square diagram.
\end{proof}

At this point, we are ready to prove:

\begin{claim}\label{claim:g_loxo}
The element $g$ is loxodromic on $\hyp$.
\end{claim}

\begin{proof}[Proof of Claim~\ref{claim:g_loxo}]
Claim~\ref{claim:D''_1_no_cone} showed that $D$ is a square diagram.  So $D$ lifts to a diagram $D\to\widetilde X$, bounded 
by a lift $\widetilde Q$ of $Q$ and a lift $\widetilde A_n$ of $A_n$ joining $\tilde a$ to $\tilde g^n\tilde a$.  Now, each 
$\alpha_i$ lifts to a subpath $\tilde\alpha_i$ of $\widetilde Q$, and $\tilde\alpha_i$ lies in an elevation $\widetilde C_i$ 
of the hyperplane carrier or relator $C_i$.  Hence $\dist_{\widehat{\hyp}}(\tilde a,\tilde g^n\tilde a)$ is bounded above by 
a linear function of $N$, by Lemma~\ref{lem:qi_to_int_graph}.

If $N$ grows sublinearly in $n$, this means that $\tilde g$ is not loxodromic on $\widehat{\hyp}$.  But recall (from the very 
beginning of the proof of the lemma) that $\tilde g$ is loxodromic on $\widehat{\hyp}$ because it is embeddable and 
loxodromic on 
$\widetilde{\hyp}$.  Hence, $N$ grows linearly in $n$, so $g$ is loxodromic on $\hyp$.
\end{proof}

To conclude, we have to prove that $g$ is fast.  This is similar to the proof that $g$ is loxodromic, except we use the 
stronger estimates from Claim~\ref{claim:bounded_coarse_intersection_with_relators} that require 
$\frac{17}{36}$--asystolicity.

\begin{claim}\label{claim:fast}
There exists $\Delta$ such that $g$ is $\Delta$--fast.     
\end{claim}

\begin{proof}[Proof of Claim~\ref{claim:fast}]
Let $n\geqslant 0$.  Let $S$ be a geodesic in $\cay(X^*)$ from $a$ to $g^na$.  Write $S=S_0\cdots S_N$, with 
each $S_i$ lying in a relator or hyperplane carrier.  Let $D\to \widetilde X^*$ be a minimal-complexity diagram bounded by 
$A_n$ and $S$.  The proof of Claim~\ref{claim:D''_no_shell} implies that $D$ contains no positively curved shell with 
outer path along $A_n$.  Also, $D$ has no positively curved shell with outer path along $S$, by 
Theorem~\ref{thm:short_inner_paths}, since $S$ is a geodesic.

There are no spurs in $A_n$, since $A$ is an embedded path, and no spurs in $S$ since $S$ is a geodesic.  So, any feature of 
positive curvature in $D$ occurring along $A_n$ or $S$ is a generalised corner of a square.  Hence, performing square 
homotopies and applying Theorem~\ref{thm:ladder_thm}, we have the following.  The diagram $D$ is square-homotopic to a ladder 
$D''_2$ with boundary path $A'_nT^{-1}$, where $T$ is a geodesic square-homotopic to $S$ and $A'_n$ is 
square-homotopic to, and 
has the same length as, $A_n$.\footnote{It was much simpler to see that $D''_2$ is a ladder than it was for 
$D''_1$ above 
because $S$ is a $\cay(X^*)$--geodesic.}

Let $K$ be a cone-cell of $D''_2$.  Let $\gamma$ be the part of $\boundary_pK$ lying on $T$ and let $\sigma$ 
be the part of 
$\boundary_pK$ lying on $A'_n$.  All dual curves in $D$ starting on $A_n$ end on $A'_n$, because $A'_n$ was 
obtained from $A_n$ by square homotopies across generalised corners.  Thus, $\frac{17}{36}$--asystolicity and 
Claim~\ref{claim:bounded_coarse_intersection_with_relators} imply that $|\sigma|<\frac{35}{72}\|Y\|$, where $Y$ is the 
relator to which $K$ maps.

Now, since $T$ is a geodesic and $D''_2$ is a ladder, we have $|\gamma|\leqslant |\sigma|+2M$.  Hence 
$|\boundary_pK|<\|Y\|(\frac{1}{72}+\frac{35}{36})<\|Y\|$.  Thus $\boundary_pK$ bounds a square diagram in $X$, contradicting 
minimal complexity.  Hence $D$ contains no cone-cells.  Hence $D$ is a square diagram, and dual curves starting on $S$ end on 
$A_n$.  We thus have $|S_i|\leqslant\mathfrak q_g$ for all $i$, 
by Claim~\ref{claim:bounded_coarse_intersection_with_relators}.  Thus there exists $\Delta$ (depending on $g$) so that $g$ 
is $\Delta$--fast.
\end{proof}

To conclude, Claim~\ref{claim:g_loxo} and Claim~\ref{claim:fast} combine with Lemma~\ref{lem:acyl_version_2} to show that 
$g$ acts on $\hyp$ as a WPD element.
\renewcommand{\qedsymbol}{$\Box$}
\end{proof}

\subsection{Achieving asystolicity in the essential case}\label{subsec:loxodromic_construction}
In this subsection, we additionally assume that $\pi_1X$ acts \emph{essentially} on $\widetilde X$, in the sense 
of~\cite{CapraceSageev:rank_rigidity}.  This means that for each hyperplane $H$ of $\widetilde X$, each of the components of 
$\widetilde X-H$ contains points in any (fixed) $\pi_1X$--orbit arbitrarily far from $H$.  In particular, 
there is no proper $\pi_1X$--invariant convex subcomplex.  

By~\cite[Theorem 5.4]{Hagen:boundary} (which in turn relies on results in~\cite{CapraceSageev:rank_rigidity}), and 
Lemma~\ref{lem:qi_to_int_graph}, one of the following holds:
\begin{enumerate}
     \item \label{item:product}$\widetilde X\cong A\times B$, where $A,B$ are unbounded CAT(0) cube complexes.  In this 
case, either $\pi_1X^*$ is finite or $\pi_1X^*=\pi_1 X$.  Indeed, if $\widetilde Y_i=\widetilde X$, then $Y_i\to X$ is a 
covering map.  Since $Y_i$ is compact, this implies that $Y_i$ is a finite cover, whence $\pi_1X^*$ is finite.  

Otherwise, 
since $\widetilde Y_i\subseteq\widetilde X$ is convex, we have $\widetilde Y_i=A'\times B'$, where $A'\subset A$ and $B'\subset 
B$ are convex subcomplexes, and one of the two containments is proper.  Choose a hyperplane $H$ disjoint from $\widetilde 
Y_i$.  Without loss of generality, $H=H'\times B$, where $H'$ is a hyperplane of $A$ disjoint from $A'$.  Then $B'$ is an 
abstract wall-piece, so $B'$ has bounded diameter.  Hence there is a hyperplane $V'$ of $B$ disjoint from $B'$, so 
$\widetilde Y_i$ is disjoint from the hyperplane $A\times V'$.  Thus $A'$ is an abstract wall-piece and thus bounded.

Hence $\widetilde Y_i$ is bounded and therefore has trivial stabiliser, i.e. $Y_i$ is contractible.  

 \item \label{item:general_type}  There exists $\tilde g\in\pi_1X$ acting loxodromically on $\widetilde{\hyp}$.
\end{enumerate}

\textbf{We now restrict to case~\eqref{item:general_type}.}  So, we are assuming that some $\tilde g\in\pi_1X$ 
acts loxodromically on $\widetilde{\hyp}$.  By Lemma~\ref{lem:loxodromic_persist_1}, either $\tilde g$ acts 
loxodromically on $\widehat{\hyp}$, or stabilises some $\widetilde Y_i$.  The next lemma arranges for the 
former to hold, in a slightly more general setting.

\begin{lem}\label{lem:loxodromic_on_coned_contact_graph}
Let $X$ be a compact connected nonpositively curved cube complex such that $\pi_1X$ acts on $\widetilde X$ 
essentially.  Let $\{\widetilde Y_i\}_{i\in\mathcal I}$ be a $\pi_1X$--invariant collection of convex subcomplexes such that 
$\stabilizer_{\pi_1X}(\widetilde Y_i)$ acts cocompactly on $\widetilde Y_i$ for all $i$.  Suppose that there exists 
$M$ such that  $\diam(\gate_{\widetilde Y_i}(\widetilde Y_j))\leqslant M$ whenever $\widetilde Y_i\neq\widetilde 
Y_j$.  Suppose also that for all $\widetilde Y_i$ and hyperplanes $H$ disjoint from $\widetilde Y_i$, we have 
$\diam(\gate_{\widetilde Y_i}(\mathcal N(H)))\leqslant M$.

Let $\widehat{\hyp}$ be obtained from $\widetilde X$ by coning off each hyperplane carrier and each 
$\widetilde Y_i$.  Then one of the following holds:
\begin{itemize}
 \item There exists $i$ such that $\widetilde X=\widetilde Y_i$.
 \item There exists $\tilde g\in \pi_1X$ acting loxodromically on $\widehat{\hyp}$.
 \item $\widetilde X$ decomposes as a product with unbounded factors.\end{itemize}
\end{lem}

\begin{proof}
If $\widetilde X$ is a product with unbounded factors, we are done, so we can assume that $\pi_1X$ contains elements acting 
loxodromically on $\widetilde{\hyp}$.  Let $\tilde g\in\pi_1X$ be such an element.

If $\tilde g$ acts loxodromically on $\widehat{\hyp}$, we are done.  If not, then by (the proof of) 
Lemma~\ref{lem:loxodromic_persist_1}, $\tilde g$ stabilises some $\widetilde Y_i$.

Suppose that there exists a hyperplane $U$ of $\widetilde X$ that is disjoint from $\widetilde Y_i$.  Let 
$\ell=\dist_{\widetilde X}(\neb(U),\widetilde Y_i)$.  Let $n\gg 0$ be chosen sufficiently large in terms of 
$\ell$ and $M$. 

Then $\dist_{\widetilde X}(U,\tilde g^nU)>10^9(\ell+ M)$ and $\dist_{\widetilde{\hyp}}(U,\tilde g^nU)>100$.  
Moreover, $\tilde g^nU$ is also disjoint from $\widetilde Y_i$, since $\tilde g$ stabilises $\widetilde Y_i$.  
(So $\widetilde Y_i$ lies ``between'' the hyperplanes $U,\tilde g^nU$.)

Apply the Double-Skewering Lemma~\cite{CapraceSageev:rank_rigidity} to obtain $\tilde k\in\pi_1 X$ such that 
some (hence any) combinatorial geodesic axis $\widetilde B$ for $\tilde k$ is cut by both $U$ and $\tilde 
g^nU$.  Since $U,\tilde g^nU$ are $100$--far in $\widetilde{\hyp}$, Lemma~\ref{lem:qi_to_int_graph} 
and~\cite[Proposition 5.3]{Hagen:boundary} imply that $\tilde k$ acts loxodromically on $\widetilde{\hyp}$.

Now, any hyperplane $H$ separating $U$ from $\tilde g^nU$ must cross $\widetilde B$.  Such an $H$ either 
crosses $\widetilde Y_i$, or separates either $U$ or $\tilde g^nU$ from $\widetilde Y_i$.  So there are at 
least $10^9(\ell +M) - 2\ell$ hyperplanes $H$ that cross both $\widetilde B$ and $\widetilde Y_i$.

Suppose $\tilde k$ stabilises some $\widetilde Y_j$.  Then we could have chosen $\widetilde B$ to lie in $\widetilde Y_j$.  
So, there are $>M$ hyperplanes that cross $\widetilde Y_i,\widetilde Y_j$, and the geodesic $\widetilde B$.  It follows that 
$\diam(\gate_{\widetilde Y_i}(\widetilde Y_j))>M$, which is a contradiction unless $\widetilde Y_i=\widetilde Y_j$.

But $\tilde k$ cannot stabilise $\widetilde Y_i$.  Indeed, $\widetilde Y_i$ lies in the $U$--halfspace containing $\tilde 
g^nU$ and in the $\tilde g^nU$--halfspace containing $U$.  Since $\tilde k$ skewers $U,\tilde g^nU$, we thus have that 
$\tilde k\widetilde Y_i$ and $\widetilde Y_i$ are separated by $U$ or $\tilde g^nU$, and are thus distinct.

Now Lemma~\ref{lem:loxodromic_persist_1} implies that $\tilde k$ is loxodromic on $\widehat{\hyp}$, as 
required.

The above works provided $\widetilde Y_i$ is disjoint from some hyperplane.  But suppose $\widetilde Y_i$ 
intersects every hyperplane.  Then for any $\tilde t\in\pi_1X$, infinitely many hyperplanes 
intersect $\widetilde Y_i$ and $\tilde t\widetilde Y_i$, implying that $\pi_1X=\stabilizer_{\pi_1X}(\widetilde 
Y_i)$.  By essentiality, $\widetilde X=\widetilde Y_i$.
\end{proof}

Now we find asystolic elements, under the two different situations in Theorem~\ref{thmi:main}.

Given $\tilde g\in\pi_1X$, let $\tau_{\tilde g}=\inf_{\tilde a\in\widetilde 
X^{(0)}}\dist_{\widetilde X}(\tilde a,\tilde g\tilde a)$ be the 
combinatorial translation length of $\tilde g$ on $\widetilde X$.  We have $\tau_{\tilde g}\geqslant1$ whenever $\tilde g\neq 1$.

\begin{defn}[The constant $L_0$]
Let $\mathcal L_0$ be the (nonempty) set of $\tilde g\in\pi_1X$ such that $\tilde g$ is loxodromic on 
$\widetilde{\hyp}$.  Let $L_0=L_0(X)=\inf_{\tilde g\in\mathcal L_0}\tau_{\tilde g}$.
\end{defn}

\begin{rem}[Conditions enabling asystolicity]
Let $\langle X\mid\{Y_i\to X\}_{i\in\mathcal I}\rangle$ be a cubical presentation such that one of the 
following holds:
\begin{enumerate}[(I)]
 \item $\langle X\mid\{Y_i\to X\}_{i\in\mathcal I}\rangle$ satisfies the uniform $C''(\frac{1}{7L_0})$ condition.  Moreover, the cubical presentation is 
\emph{minimal} (in the sense of Definition~\ref{defn:minimal}), i.e. for all $i\in\mathcal I$ and some (hence 
any) elevation $\widetilde Y_i\to\widetilde X$ of $Y_i\to X$, the subgroups $\pi_1Y_i$ and 
$\stabilizer_{\pi_1X}(\widetilde Y_i)$ are conjugate.\label{item:regime_1}  Let $M=\frac{1}{7L_0}\inf_{i\in\mathcal I}\|Y_i\|$ be the bound 
on the diameters of abstract pieces provided by the small-cancellation condition.

\item $\langle X\mid\{Y_i\to X\}_{i\in\mathcal I}\rangle$ is a cubical presentation with a uniform bound $M$ on the 
diameters 
of abstract cone-pieces and abstract wall-pieces, and each $\stabilizer_{\pi_1X}(\widetilde Y_i)$ has infinite 
index in $\pi_1X$.  In this case, Lemma~\ref{lem:loxodromic_on_coned_contact_graph} provides a nonempty set 
$\mathcal L_1\subseteq \pi_1X$ such that each $\tilde g\in\mathcal L_1$ acts loxodromically on $\widehat{\hyp}$. 
 Let $L_1=\inf_{\tilde g\in\mathcal L_1}\tau_{\tilde g}$.\label{item:regime_2}
\end{enumerate}
For each $i\in\mathcal I$, let $\widehat Y_i\to Y_i$ be a finite connected regular cover so that the following holds:
\begin{enumerate}[(i)]
 \item \label{item:r1_cover}If~\eqref{item:regime_1} holds, then $\|\widehat Y_i\|>\max\{144,7L_0\}\cdot M$ for all 
$i\in\mathcal I$.
 \item \label{item:r2_cover}If~\eqref{item:regime_2} holds, then $\|\widehat Y_i\|>\max\{144 ,7L_1\}\cdot M$ for 
all $i\in\mathcal I$.
\end{enumerate}
\end{rem}

Note that $L_0$ does not depend on the $Y_i$.  The constant $L_1$ depends on the $Y_i$, but not on the further finite covers.

\begin{lem}[Finding fast $\hyp$--loxodromics]\label{lem:acyl_regime}
Let $\langle X\mid \{Y_i\}_{i\in\mathcal I}\rangle$ satisfy ~\eqref{item:regime_1} 
[resp.~\eqref{item:regime_2}] and let the finite covers $\widehat Y_i\to Y_i$ satisfy~\eqref{item:r1_cover} 
[resp.~\eqref{item:r2_cover}].  Let $G$ be the group 
with cubical presentation $\langle X\mid \{\widehat Y_i\}_{i\in\mathcal I}\rangle$.  Suppose that $\widetilde 
X$ contains a nontrivial wall-piece, or a pair of distinct $\widetilde Y_i,\widetilde Y_j$ with 
$\diam(\gate_{\widetilde Y_i}(\widetilde Y_j))\geqslant 1$.  Then $G$ is either virtually cyclic or acylindrically 
hyperbolic.
\end{lem}

\begin{proof}
 Choose $\tilde g\in \mathcal L_0$ as follows.  If~\eqref{item:regime_1} holds, then choose $\tilde g$ with 
$\tau_{\tilde g}=L_0$.  If~\eqref{item:regime_2} holds, choose $\tilde g$ to be in $\mathcal L_1$ and satisfy 
$\tau_g=L_1$.  

If~\eqref{item:regime_1} holds,  $\tilde g$ cannot be conjugate into any 
$\pi_1Y_i$, since $\|Y_i\|>7L_0$.  The minimality hypothesis thus implies that $\tilde g$ 
cannot stabilise any $\widetilde Y_i$.

If~\eqref{item:regime_2} holds, then $\tilde g\in\mathcal L_1$, so $\tilde g$ is loxodromic on 
$\widehat{\hyp}$ and thus cannot stabilise any $\widetilde Y_i$.

We now argue that $\tilde g$ is $\frac{17}{36}$--asystolic (recall that this means that \textbf{every} axis of $\tilde g$ is $\frac{17}{36}$--asystolic).  Suppose not.  Then, by definition, for some 
combinatorial geodesic axis $\widetilde A$ of $\tilde g$, there is a subpath $\widetilde P$ of $\widetilde A$ 
such that $\widetilde P$ lies in some $\widetilde Y_i$ and $|\widetilde P|\geqslant 17\|\widehat Y_i\|/36$.

Let $\sigma=\min_j\|\widehat Y_j\|$.  If~\eqref{item:regime_1} holds, we have $\sigma>7L_0$, 
since there is a nontrivial piece.  Likewise, if~\eqref{item:regime_2} holds, we have $\sigma>7L_1$.  So, in either case, 
$\sigma> 7\tau_g$.

Now, $\widetilde P$ contains a $0$--cube $\tilde a$ such that $\tilde a,\tilde g\tilde a,\tilde g^2\tilde 
a\in\widetilde P$, since $|\widetilde P|>7\cdot17\cdot\tau_g/36$.  Thus $\widetilde Q=\widetilde 
P\cap\tilde g\widetilde P$ is a geodesic that lies in $\widetilde Y_i\cap\tilde g\widetilde 
Y_i$.  Since $\tilde g\widetilde Y_i\neq\widetilde Y_i$, the geodesic $\widetilde Q$ is a cone-piece.

Hence $|\widetilde Q|<\frac{1}{144}\|\widehat Y_i\|$.  On the other hand, $|\widetilde Q|\geqslant |\widetilde 
P|-2\tau_g$.  So, $|\widetilde Q|>\frac{17}{36}\|\widehat Y_i\|-2\tau_g$.  So, $\tau_g>\frac{67}{288}\sigma$, 
while on the other hand, $\tau_g<\sigma/7$.  This is a contradiction, so $\tilde g$ is 
$\frac{17}{36}$--asystolic.

Let $X^*$ be the presentation complex formed from $\langle X\mid\{\widehat Y_i\}_{i\in\mathcal I}\rangle$ and let $\hyp$ be 
obtained from $\widetilde X^*$ by coning off the various hyperplane carriers and $\widehat Y_i$.  Then 
Lemma~\ref{lem:loxodromics_persist_2} implies that $\tilde g$ acts on $\hyp$ as a WPD element.  
Applying~\cite[Theorem 1.1]{Osin:acyl} shows that $G$ is virtually cyclic or acylindrically hyperbolic.
\end{proof}

\subsection{Proving Theorem~\ref{thmi:main}}To summarise the essential case, we have:

\begin{prop}[Acylindrical hyperbolicity when $\pi_1X$ acts essentially]\label{prop:essential_acylindrical}
Let $X$ be a compact nonpositively curved cube complex such that $\pi_1X$ acts essentially on $\widetilde X$.

Let $\{Y_i\to X\}_{i\in\mathcal I}$ be a (possibly infinite) set of local isometries of nonpositively curved 
cube complexes with each $Y_i$ compact.  Suppose that $\widetilde X$ does not split as a nontrivial product.

\begin{enumerate}
 \item Suppose that $\langle X\mid\{Y_i\}_{i\in\mathcal I}\rangle$ satisfies~\eqref{item:regime_1}.  Then there exists 
$\alpha_0=\alpha_0(X)$ such that the following holds.  For each $i\in\mathcal I$, let $\widehat Y_i\to Y_i$ be 
a finite regular cover such that $\langle X\mid\{\widehat Y_i\}_{i\in\mathcal 
I}\rangle$ is a $C''(\alpha)$ presentation for some $\alpha\in[0,\alpha_0]$ and let $G=\pi_1(\langle X\mid\{\widehat 
Y_i\}_{i\in\mathcal 
I}\rangle)$.  Then $G$ is finite, two-ended, or acylindrically hyperbolic.    

\item Suppose that $\langle X\mid\{Y_i\}_{i\in\mathcal I}\rangle$ satisfies~\eqref{item:regime_2}.  Then there exists 
$\alpha_1=\alpha_1(X,\{Y_i\}_{i\in\mathcal I})$ such that the following holds.  For each $i\in\mathcal I$, let $\widehat Y_i\to 
Y_i$ be 
a finite regular cover such that $\langle X\mid\{\widehat Y_i\}_{i\in\mathcal 
I}\rangle$ is a $C''(\alpha)$ presentation for some $\alpha\in[0,\alpha_1]$ and let $G=\pi_1(\langle X\mid\{\widehat 
Y_i\}_{i\in\mathcal 
I}\rangle)$.  Then $G$ is finite, two-ended, or acylindrically hyperbolic. 
\end{enumerate}

\end{prop}

\begin{rem}
The distinction between~\eqref{item:regime_1} and~\eqref{item:regime_2} is subtle.  The first says that if our initial 
(minimal) cubical presentation already satisfies a cubical small-cancellation condition depending only on $X$, then 
replacing each $Y_i$ by a suitable finite cover (of degree at least a constant independent of $X$, say $144$), we achieve 
acylindrical hyperbolicity in the quotient.  Condition~\eqref{item:regime_2} only requires our initial cubical presentation to satisfy some uniform small-cancellation condition (not depending on $X$); 
the price is that to 
achieve acylindrical hyperbolicity, we have to pass to covers $\widehat Y_i\to Y_i$ of degree at least a constant depending 
on our initial cubical presentation.
\end{rem}

\begin{proof}[Proof of Proposition~\ref{prop:essential_acylindrical}]
 We can assume that each $Y_i$ is non-contractible, by removing contractible relators, without affecting anything.  Note 
that the graph 
$\widetilde{\hyp}$ depends only on $X$, and $\widehat{\hyp}$ depends on $X$ and the $Y_i$ but not on the further finite
covers $\widehat Y_i$.  In fact, the collection of subgroups $\stabilizer_{\pi_1X}(\widetilde Y_i)$ depends 
only on the initial data $\{Y_i\to X\}_{i\in\mathcal I}$.  

We can also assume $\widetilde X$ is not a product, so $\pi_1X$ has elements acting loxodromically 
on $\widetilde{\hyp}$.  

Assume that~\eqref{item:regime_1} [resp. ~\eqref{item:regime_2}] holds and let $\{\widehat Y_i\to 
Y_i\}_{i\in\mathcal I}$ be as in the statement, where $\alpha_0$ [resp. $\alpha_1$] comes from 
condition~\eqref{item:r1_cover} [resp.~\eqref{item:r2_cover}].  Suppose that there exists $i$ such that for 
some elevation $\widetilde Y_i\subseteq\widetilde X$, there is a subcomplex $\widetilde B$ such that 
$\gate_{\widetilde Y_i}(\widetilde B)$ has diameter at least $1$ and either $\widetilde B=\widetilde 
Y_j\neq\widetilde Y_i$ or $\widetilde B$ is the carrier of a hyperplane not crossing $\widetilde Y_i$.  Then 
Lemma~\ref{lem:loxodromic_on_coned_contact_graph} implies the proposition in this case.

Otherwise, whenever $\widetilde Y_i,\widetilde Y_j$ are distinct, no hyperplane crosses both, and whenever $H$ 
is a hyperplane disjoint from $\widetilde Y_i$, no hyperplane crosses both $H$ and $\widetilde Y_i$.  

If $H$ crosses $\widetilde Y_i$, and $v\in \mathcal N(H)$ is a vertex not in $\widetilde Y_i$, then some hyperplane $V$ separates $v$ 
from $\widetilde Y_i$.  So $\gate_{\widetilde Y_i}(\mathcal N(V))$ is nontrivial (since $H$ crosses $\widetilde Y_i$ and $V$), 
leading to a contradiction.  Hence $H\subseteq \widetilde Y_i$.  Thus, $\widetilde X$ 
decomposes as a tree of spaces with each edge space a $0$--cube; each 
elevation of each relator is a vertex space.  

Hence, $X$ is a finite graph of spaces 
with trivial edge spaces and a vertex space $Z$ which is a compact nonspositively-curved cube complex having 
trivial intersection with the image of each $Y_i$; the various images of the local isometries $\widehat Y_i\to Y_i \to X$ 
are 
also vertex spaces.  Thus, either $G$ splits over the trivial group, or $G=\pi_1Z$.  In the former case, we are 
done. In the latter case, the result follows because $G$ is the fundamental group of a nonpositively curved 
cube complex.
\end{proof}

We are now ready to prove Theorem~\ref{thmi:main}.

\begin{proof}[Proof of Theorem~\ref{thmi:main}]
Let $\langle X\mid\{Y_i\}_{i\in\mathcal I}\rangle$ be as in the statement.

By~\cite[Proposition 3.5]{CapraceSageev:rank_rigidity}, there is a convex $\pi_1X$--invariant subcomplex $\widetilde 
Z\subseteq\widetilde X$ on which $\pi_1X$ acts cocompactly and essentially.  (Although it is not made explicit 
in~\cite{CapraceSageev:rank_rigidity}, $\widetilde Z$ is in general only a subcomplex of the first cubical subdivision of 
$\widetilde X$, because the action may have inversions across hyperplanes.  But, we can pass to the cubical 
subdivision without affecting the argument, and thus assume that $\widetilde Z$ is a 
subcomplex.)

Let $Z=\pi_1X\backslash \widetilde Z$, which is a compact nonpositively curved cube complex.  Each $\tilde 
g\in\pi_1Z=\pi_1X$ has the same translation length on $\widetilde Z$ as on $\widetilde X$, since $\widetilde Z$ is a 
$\pi_1X$--equivariant 1--lipschitz retract of $\widetilde X$ (use the gate map to $\widetilde Z$).  The inclusion 
$\widetilde Z\to\widetilde X$ descends to a local isometry $Z\to X$ (inducing an isomorphism on fundamental group).  

Form a collection $\{U_j\to Z\}_j$ of local isometries as follows.  For each $Y_i$ and each lift $\widetilde 
Y_i\subseteq\widetilde X$, let $\widetilde U_i=\gate_Z(\widetilde Y_i)$, which is a convex subcomplex.  Since $\pi_1X$ 
stabilises $Z$, we have $\stabilizer_{\pi_1X}(\widetilde Y_i)\subseteq\stabilizer_{\pi_1X}(\widetilde U_i)$.  The reverse 
inclusion also holds.  Indeed, suppose $\tilde g\widetilde U_i=\widetilde U_i$. We can assume that $\widetilde U_i$ is 
unbounded since we can assume $\pi_1Y_i$ is infinite (otherwise we would have discarded the simply connected 
relator $Y_i$).  So the small-cancellation conditions imply $\tilde g\widetilde Y_i=\widetilde Y_i$.  Hence $\widetilde Y_i$ 
and $\widetilde U_i$ have the same stabiliser.  Letting $U_i=\pi_1Y_i\backslash\widetilde U_i$, the inclusion $\widetilde 
U_i\to \widetilde Z$ induces a local isometry $U_i\to Z$.  The cubical presentation $\langle Z\mid\{U_i\}_{i}\rangle$ presents 
the same group as $\langle X\mid\{Y_i\}_{i\in\mathcal I}\rangle$.  (For each $Y_i$, there is a nonempty family of $U_j\to Z$ 
representing subgroups conjugate to $\pi_1Y_i$.)

Any wall-piece or cone-piece in $\widetilde U_i$ is the image under $\gate_Z$ of a piece in some $\widetilde Y_i$, so 
lengths of pieces have not increased, while systoles have not decreased.  So any small-cancellation condition satisfied by 
the original cubical presentation persists.  The theorem now follows from 
Proposition~\ref{prop:essential_acylindrical}.  % 
% 
% 
% We 
% form a new cubical presentation by attaching to $Z$ all of the components of each fiber product 
% $Y_i\otimes_{_X}Z$, using the local isometries $Y_i\otimes_XZ\to Z$ coming from the fiber product construction (see~\cite[Definition 8.8]{Wise:QCH} for the definition of the fiber product).  Pieces 
% in the new cubical presentation 
% arise by intersecting pieces in $\widetilde X$ with $\widetilde Z$, and the systoles of the relators have not decreased, so 
% any small-cancellation condition persists.  
\end{proof}

\bibliographystyle{alpha}
\bibliography{strebel}
\end{document}